%% file: main.tex
\documentclass[a4paper]{amsart}

\usepackage{cmbright}
\usepackage{mathrsfs}
\usepackage[utf8]{inputenc}
\usepackage[T1]{fontenc}
\usepackage[english]{babel}
\usepackage{amsfonts}
\usepackage{amsthm}
\usepackage{amssymb}
\usepackage{mathcomp}
\usepackage{mathrsfs}
\usepackage[bb=boondox]{mathalfa}
\usepackage{graphicx}
\usepackage{epstopdf}
\usepackage[centertags]{amsmath}
\usepackage[pdftex]{hyperref}
\usepackage{vmargin}
\usepackage{tikz}
\usepackage{tikz-cd}
\usepackage{dsfont}
\usepackage{cases}
\usepackage{mathtools}
\usepackage{shuffle}
\usepackage{adjustbox}

\include{macros}

\title{Mapping coalgebras II\\Operads}
\author{Brice Le Grignou}
\email{bricelegrignou@gmail.com}

\date{\today}

\begin{document}

\maketitle

\begin{abstract}
In this article, we describe how coalgebraic structures on operads
induce algebraic structures on their categories of algebras and coalgebras.
\end{abstract}

\setcounter{tocdepth}{1}
\tableofcontents

\section*{Introduction}

This is the second article of a series about enrichment of categories
of algebras or coalgebras by categories of coalgebras.
The first one dealt with algebras over a monad and coalgebras over a comonad;
this one deals with algebras and coalgebras over an operad.

Let us fix a ground symmetric monoidal category $\categ E$. The 2-category of
$\categ E$-enriched coloured operads has the canonical structure
of a symmetric monoidal 2-category (monoidal context) given by the Hadamard
tensor product $\operad P, \operad Q \mapsto \operad P \otimeshadamard \operad Q$
defined as
\begin{align*}
	\Ob(\operad P \otimeshadamard \operad Q) &=
	\Ob(\operad P)  \times \Ob(\operad Q)
	\\
	(\operad P \otimeshadamard \operad Q)((o_1, o'_1), \ldots, (o_n, o'_n); (o, o'))
	&= \operad P(o_1, \ldots, o_n; o) \otimes \operad Q(o'_1, \ldots, o'_n; o') .
\end{align*}
The comonoids for this tensor product are called Hopf operads.

Besides, for any symmetric monoidal enriched category $\ecateg C$,
one can produce an enriched operad $\End(\ecateg C)$, whose colours 
are the objects of $\ecateg C$ and so that for any such objects
$x_1, \ldots, x_n,x$, then
$$
\End(\ecateg C)(x_1, \ldots, x_n;x) = \ecateg C (x_1 \otimes \cdots \otimes x_n,x)
$$
Such a construction is part of a monoidal strict 2-functor (strong
context functor) from the 2-category of symmetric monoidal enriched categories
and lax functors to the 2-category of enriched operads.

Actually, the enriched operad $\End(\ecateg C)$ has the canonical structure
of a $\catoperad{E}_\infty$-monoid for the Hadamard tensor product.
Using the fact that the Yoneda 2-functor $\categ D^\op \times \categ D \to \Cats$
for a monoidal 2-category $\categ D$ is lax monoidal,
one gets the fact that
for any cocommutative Hopf enriched operad $\operad Q$,
the category of $\operad Q$-algebras in $\ecateg C$ and the category
of the category $\operad Q$-coalgebras in $\ecateg C$
\begin{align*}
	\catalg{\ecateg C}{\operad Q} &= \Operade (\operad Q, \End(\ecateg C))
	\\
	\catcog{\ecateg C}{\operad Q} &= \Operade (\operad Q, \End(\ecateg C^\op))^\op
\end{align*}
are symmetric monoidal categories. More generally, one has the following
result.

\begin{theorem*}[\ref{thm:main}]
	Let $\catoperad B, \catoperad B'$ be two categorical operads.
	Then, the lax context functor
   $$
	   \catalg{-}{-}: \Operadesmall^\op  \times \cmonlax(\Cate) \to \Cats_{\mathcal{U}}
	   $$
   induces a lax context functor that is natural with respect to $B,B'$:
	$$
	\tcatalg{\Operadesmall^\op}{\catoperad B}_{\lax} \times
	\tcatalg{\categ{CMon}_{\lax}(\Cate)}{\catoperad B'}_{\lax}
	 \to
	 \tcatalg{\Cats_{\mathcal U}}{\catoperad B \times \catoperad{B}'}_{\lax} .
	$$
\end{theorem*}

This result has further consequences. For instance, if $\operad P$ is
an enriched operad that has the structure of a left comodule over 
a Hopf enriched operad $\operad Q$ then the category $\catalg{\ecateg C}{\operad P}$
of $\operad P$-algebras in $\ecateg C$ is tensored over that 
of $\operad Q$-algebras. Similarly the category $\catcog{\ecateg C}{\operad P}$
of $\operad P$-coalgebras in $\ecateg C$ is tensored over that 
of $\operad Q$-coalgebras.

Furthermore, one can consider the case where a symmetric monoidal
enriched category $\ecateg C$ is cotensored in some sense over another
symmetric monoidal $\ecateg D^\op$. An example of this situation is provided
by the ground symmetric monoidal category $\categ E$ if it is closed. Then
the $\categ E$-enriched category $\ecateg E$ associated to $\categ E$
is cotensored over $\ecateg E^\op$ through the internal hom
\begin{align*}
	\categ E \times \categ E^\op &\to \categ E
	\\
	y,x &\mapsto [x,y].
\end{align*}
In such a context, for any Hopf enriched operad $\operad Q$ and any
enriched operad $\operad P$ equipped with the structure of a
right $\operad Q$-comodule, then
the category $\catalg{\ecateg C}{\operad P}$ of $\operad P$-algebras in $\ecateg C$
is cotensored
over the category $\catcog{\ecateg D}{\operad Q}$ of $\operad Q$-coalgebras
in $\ecateg D$.

Finally, in many contexts, the cotensorisation functor
$$
	\catalg{\ecateg C}{\operad P} \times \catcog{\ecateg D}{\operad Q}^\op
	 \to \catalg{\ecateg C}{\operad P}
$$
has a left adjoint with respects to both variables. Then 
the category $\catalg{\ecateg C}{\operad P}$ is not only
cotensored over $\catcog{\ecateg D}{\operad Q}$ but also
tensored and enriched over this monoidal category. One has actually
the same phenomenon
mutatis mutandis for the tensorisation functor mentioned above
$$
	\catcog{\ecateg C}{\operad Q} \times \catcog{\ecateg C}{\operad P}
	 \to \catcog{\ecateg C}{\operad P}.
$$

\subsection*{Universes}

As in \cite{LeGrignou22}, we consider three
universes $\mathcal U \in \mathcal{V} \in \mathcal W$.
A set is called $\mathcal U$-small if it is an element of $\mathcal U$ and
it is called $\mathcal U$-large if it is a subset of $\mathcal U$. The notion
of smallness and largeness are defined similarly for the other universes.
We thus have a hierarchy of sizes of sets
$$
\mathcal{U}\text{ small sets}
\subset \mathcal{U}\text{ large sets}
\subset \mathcal{V}\text{ small sets}
\subset \mathcal{V}\text{ large sets}
\subset \mathcal{W}\text{ small sets}
\subset \mathcal{W}\text{ large sets}.
$$
Besides, a $\mathcal U$-category is a category whose set of objects
is $\mathcal{U}$-large and whose hom sets are all $\mathcal{U}$-small.
Such a $\mathcal U$-category is called $\mathcal U$-small if its set
of object is $\mathcal U$-small. We have similar notions for
$\mathcal V$ and $\mathcal W$.

Finally, we will use the following aliases:
\begin{itemize}
    \itemt a set, also called small set will be a
    $\mathcal{W}$-small set;
    \itemt a large set will be a
    $\mathcal{W}$-large set;
    \itemt a category will be a $\mathcal W$-category;
    \itemt a small category will be a $\mathcal W$-small category.
\end{itemize}

\subsection*{Some notations}

\begin{itemize}
    \itemt For any natural integer $n$, the set of permutations of the set
    $$
\underline n = \{1, \ldots, n\}
    $$
    will be denoted $\categ S_n$.
    	\itemt   Let $\mbs$ be the groupoid of permutations whose objects are integers $n \in \mbn$ and whose morphisms are:
	$$
   \begin{cases}
	\hom_\mbs(n,m)=  \emptyset \text{ if }n \neq m\\
	 \hom_\mbs(n,n)= \mbs_n \text{ otherwise .}
   \end{cases}
   $$
\itemt For any natural integer $n$, we denote $[n]$ the poset $0 <1 < \cdots <n$.
    \itemt For any natural integer $n$, any permutation $\sigma \in \categ S_n$
    and any category $\categ C$,
    we denote $\sigma^\ast$ the following functor
    $$
    \categ C^n = \categ{Fun}(\underline{n}, \categ C)
    \xrightarrow{\sigma \circ -}
    \categ{Fun}(\underline{n}, \categ C)
    = \categ C^n .
    $$
    \itemt For a monad $M$ on a category $\categ E$, the induced monadic adjunction relating $M$-algebras
    to $\categ E$ will be denoted $T_M \dashv \forget^M$. Similarly, for a comonad
    $Q$, the induced comonadic adjunction relating $Q$-coalgebras
    to $\categ E$ will be denoted $\forget_Q \dashv L^M$.
\itemt Let $X$ be an object of a monoidal category $\categ E$. Then we will call an 
       element of $X$ a morphism from the monoidal unit to $X$. Then, the notation $x \in X$
       is a substitute for $x \in \Hom{\categ E}{\II}{X}$.
   \itemt Let $X$ be a finite set and, for any object $x \in X$, let $V_x$ be an object of a
	symmetric monoidal category $(\categ{E},\otimes, \II)$. Then we will denote
   \[
	   \bigotimes_{x \in X} V_x \coloneqq \left( \coprod_{\phi : \{1, \ldots, n\} \to X}
	   V_{\phi(1)} \otimes \cdots \otimes V_{\phi(n)} \right)_{\mbs_n}.
   \]
   \end{itemize}


\section{The monoidal context of Pseudo-commutative monoids}

Let $\categ C$ be a $\mathcal W$-small monoidal context.

\subsection{Pseudo commutative monoids within pseudo commutative monoids}

Let $\catoperad{Q}$ be a categorical monochromatic operad
so that for any natural integer $n$,
$\catoperad Q(n)$ is a contractible groupoid.

Recall that the 2-category $\tcatalg{\categ C}{\catoperad Q}_{lax}$
of $\catoperad{Q}$-algebras and lax morphisms has the structure of
a monoidal context. The tensor product of two algebras $A,A'$
is the object $A \otimes A' \in \categ C$ equipped with 
the morphism of categorical operads
$$
\catoperad Q \to \catoperad Q \times \catoperad Q 
\xrightarrow{A \times A'} \catoperad{End}(\categ C) \times
\catoperad{End}(\categ C) \to \catoperad{End}(\categ C)
$$
where the last map is induced by the tensor product of $\categ C$.
In particular, the map $(A \otimes A')(q)$ is the composition
$$
(A \otimes A')^{\otimes n} \simeq A^{\otimes n} \otimes A'^{\otimes n}
\xrightarrow{A(q) \otimes A'(q)} A \otimes A'
$$
for any $q \in \catoperad Q(n)$.

\begin{lemma}\label{lemmapseudocommobject}
	Let $A$ be an $\catoperad{Q}$-algebra and let $q \in \catoperad{Q}(n)$
	be an operation. Then the map
	$$
	A(q): A^{\otimes n} \to A
	$$
	has the canonical structure
	of a strong morphism of $\catoperad{Q}$-algebras.
	Then,
	\begin{itemize}
		\itemt for any morphism
		$\phi: q \to q'$ in 
		$\catoperad{Q}(n)$, the map
		$$
		A(\phi): A(q) \to A(q')
		$$
		is an invertible 2-morphism
		in the 2-category of $\catoperad{Q}$-algebras
		and strong morphisms;
		\itemt for any $q'' \in \catoperad{Q}(m)$,
		the following diagram of
		$\catoperad Q$-algebras and strong morphisms
		$$
		\begin{tikzcd}
			A^{\otimes n+m-1} 
			\ar[r, "{A(q \triangleleft_i  q'')}"]
			\ar[d, "\id^{\otimes i} \otimes A(q'') \otimes
			\id^{\otimes n-i-1}"']
			& A
			\\
			A^{\otimes n}
			\ar[ru, "A(q)"']
		\end{tikzcd}
		$$
		is strictly commutative;
		\itemt for any permutation $\sigma \in \categ{S}_n$, the
		following diagram of $\catoperad{Q}$-algebras and strong morphisms
		$$
		\begin{tikzcd}
			A^{\otimes n}
			\ar[r, "\sigma^\ast"]
			\ar[rd, "A(q^\sigma)"]
			& A^{\otimes n}
			\ar[d, "A(q)"]
			\\
			& A
		\end{tikzcd}
		$$
		is strictly commutative;
		\itemt the strong morphism of $\catoperad Q$-algebras
		$A(1): A \to A$ is the identity of $A$
		(and is thus a strict morphism).
	\end{itemize}
\end{lemma}

\begin{proof}
	Let us consider an element $p \in \catoperad Q(m)$.
	The action of $p$ on the $\catoperad{Q}$-algebra
	$A^{\otimes n}$ is the map
	$$
	A^{\otimes n}(p) :
	(A^{\otimes n})^{\otimes m}
	\xrightarrow{\sigma^\ast}
	(A^{\otimes m})^{\otimes n}
	\xrightarrow{A(p)^{\otimes n}}
	A^{\otimes n}.
	$$
	where $\sigma$ is the permutation
	$$
\sigma(m * i + j + 1) = n * j +i + 1, \quad 0 \leq i <n , \  0 \leq j <m.
	$$
	One has a unique isomorphism
	$$
	\phi : p \triangleleft q^{\times m}
	\simeq (q \triangleleft p^{\times n} )^\sigma
	$$
	in $\catoperad{Q}(n * m)$.
	We thus get a 2-isomorphism
	$$
		A(p) \circ A(q)^{\otimes n} = A(p \triangleleft q^{\times m})
		\xrightarrow{A(\phi)}
		A((q \triangleleft p^{\times n} )^\sigma)
		= A(q) \circ A(p)^{\otimes n} \circ \sigma^\ast
		= A(q) \circ A^{\otimes n}(p) .
	$$
	The data of this 2-isomorphism $A(\phi)$
	for any operation
	$p$ makes $A(q)$ a strong morphism of $\catoperad{Q}$-algebras.
	Indeed, the diagram commutations required by the definition of
	a strong morphism of $\catoperad{Q}$-algebras 
	just follow from the fact that
	the category $\catoperad{Q}(k)$ is a contractible groupoid
	for any
	natural integer $k$.

	The proofs of the remaining statements follow from the
	same fact that the category $\catoperad{Q}(k)$ is 
	contractible groupoid
	for any natural integer $k$.
\end{proof}

Let $f: A \to A'$ be a lax morphism of $\catoperad{Q}$-algebras in
the monoidal context $\categ C$.
The lax structure
of $f$ is the data of a 2-morphism in $\categ C$
	$$
 A'(q) \circ f^{\otimes n} \to f \circ A(q)
	$$
for any $s \in \catoperad{Q}(n)$. At the same time, both maps
$ A'(q) \circ f^{\otimes n}$ and $f \circ A(q)$
have structures of a lax morphism of $\catoperad Q$-algebras
by the previous Lemma \ref{lemmapseudocommobject}.

\begin{lemma}\label{lemmaeifntylaxstruct}
	For any lax morphism $f: A \to A'$ of $\catoperad{Q}$-algebras and
	any operation $q \in \catoperad{Q}(n)$, the map
	$$
 A'(q) \circ f^{\otimes n} \to f \circ A(q)
	$$
	is a 2-morphism in the 2-category
	$\tcatalg{\categ C}{\catoperad{Q}}_{\lax}$.
\end{lemma}

\begin{proof}
This is a consequence of the definition
of a lax morphism. Precisely, for any
operation $p \in \catoperad{Q}(m)$, the following diagram
commutes
$$
\begin{tikzcd}
	A'(p) \circ (A'(q) \circ f^{\otimes n})^{\otimes m}
	\ar[d, equal] \ar[r]
	&
	A'(p) \circ (f \circ A'(q))^{\otimes m}
	\ar[d, equal]
	\\
	A'(p) \circ A'(q)^{\otimes m} \circ f^{\otimes n * m}
	\ar[d, equal] \ar[r]
	& A'(p) \circ f^{\otimes m} \circ A(q)^{\otimes m}
	\ar[r]
	& f \circ A(p) \circ A(q)^{\otimes m}
	\ar[d, equal]
	\\
	A'(p \triangleleft q^{\times m}) \circ f^{\otimes n * m}
	\ar[rr]
	\ar[d, "\simeq"']
	&& f \circ A'(p \triangleleft q^{\times m})
	\ar[d, "\simeq"]
	\\
	A'((q \triangleleft p^{\times m})^\sigma)) \circ f^{\otimes n * m}
	\ar[rr]
	\ar[d, equal]
	&& f \circ A((q \triangleleft p^{\times m})^\sigma)
	\ar[d, equal]
	\\
	A'(q) \circ A'(p)^{\otimes n} 
	 \circ  f^{\otimes n * m} \circ \sigma^\ast
	 \ar[r]
	 & A'(q) \circ f^{\otimes n} \circ A(p)^{\otimes n }
	 \circ \sigma^\ast 
	 \ar[r]\ar[d, equal]
	 & f \circ A(q) \circ A(p)^{\otimes n }
	 \circ \sigma^\ast 
	 \ar[d, equal]
	\\
	& (A'(q) \circ f) \circ A^{\otimes n}(p)
	\ar[r]
	& (f \circ A(q)) \circ A^{\otimes n}(p)
\end{tikzcd}
$$
where $\sigma \in \categ S_{n*m}$ is the permutation defined in the proof
of Lemma \ref{lemmapseudocommobject}.
\end{proof}

For any monoidal context $\categ D$, one has a canonical
strict context functor
$$
\tcatalg{\categ D}{\catoperad{Q}}_{\lax} \to \categ D.
$$
Taking $\categ D = \tcatalg{\categ C}{\catoperad{Q}}_{lax}$,
this gives a strict context functor
$$
\tcatalg{\tcatalg{\categ C}{\catoperad{Q}}_{\lax}}{\catoperad{Q}}_{\lax}
\to \tcatalg{\categ C}{\catoperad{Q}}_{\lax}.
$$
Moreover, for any strict context functor $\categ D \to \categ D'$,
we get another strict context functor
$\tcatalg{\categ D}{\catoperad{Q}}_{lax} \to \tcatalg{\categ D'}{\catoperad{Q}}_{lax}$.
Taking the strict context functor to be the forgetful 2-functor
$$
\tcatalg{\categ C}{\catoperad{Q}}_{lax} \to \categ C,
$$
we get another strict context functor
$$
\tcatalg{\tcatalg{\categ C}{\catoperad{Q}}_{lax}}{\catoperad{Q}}_{lax}
\to \tcatalg{\categ C}{\catoperad{Q}}_{lax}.
$$

\begin{proposition}\label{corollarycmon}
	The two forgetful strict context functors described just above
	$$
\tcatalg{\tcatalg{\categ C}{\catoperad{Q}}_{lax}}{\catoperad{Q}}_{lax}
\rightrightarrows \tcatalg{\categ C}{\catoperad{Q}}_{lax}.
$$
have a canonical common section in the category of monoidal
contexts and strict context functors.
\end{proposition}

\begin{proof}
	Such a section sends
	\begin{itemize}
		\itemt a $\catoperad{Q}$-algebra $A$ in $\categ C$ to the $\catoperad Q$-algebra
		in $\tcatalg{\categ C}{\catoperad Q}_\strong \subset \tcatalg{\categ C}{\catoperad Q}_\lax$
		whose structural morphisms and 2-morphisms 
		$$
		\begin{cases}
			A(q): A^{\otimes n} \to A,  \quad q \in \catoperad{Q}(n)
			\\	
			A(\phi) : A(q) \to A(q'),  \quad \phi \in \catoperad{Q}(n)(q,q')
		\end{cases}
		$$
		are described in Lemma \ref{lemmapseudocommobject}; this same Lemma
		\ref{lemmapseudocommobject} ensures us that such maps do describe
		the structure of a $\catoperad{Q}$-algebra;
		\itemt a lax morphism of $\catoperad{Q}$-algebras $f : A \to A'$
		to the lax morphism of $\catoperad{Q}$-algebras
		in $\catoperad{Q}$-algebras whose underlying morphism of $\catoperad{Q}$-algebras
		is $f$ itself and whose lax structure is given by the lax structure
		of $f$ as in Lemma \ref{lemmaeifntylaxstruct}.
		\itemt finally any 2-morphism $a: f \to g$ between lax morphisms
		of $\catoperad Q$-algebras in
		$\categ C$ induces canonically a 2-morphism between their images
		in $\catoperad Q$-algebras in $\catoperad{Q}$-algebras.
	\end{itemize}
	One can check in a straightforward way that these constructions
	do define a strict context functor.
\end{proof}

\begin{corollary}\label{corollarycmonstrongstrict}
	The restrictions of the section
	$$
	\tcatalg{\categ C}{\catoperad{Q}}_\lax
	\to \tcatalg{\tcatalg{\categ C}{\catoperad{Q}}_\lax}{\catoperad{Q}}_\lax
	$$
	from Proposition \ref{corollarycmon}
	to $\tcatalg{\categ C}{\catoperad{Q}}_\strong$
	factorises through
	$$
	\tcatalg{\tcatalg{\categ C}{\catoperad{Q}}_\strong}{\catoperad{Q}}_\strong .
	$$
	Moreover, its restriction to $\tcatalg{\categ C}{\catoperad{Q}}_\strict$
	factorises through
	$$
	\tcatalg{\tcatalg{\categ C}{\catoperad{Q}}_\strong}{\catoperad{Q}}_\strict .
	$$
	
\end{corollary}

\begin{proof}
	Straightforward with the definitions.
\end{proof}

\begin{remark}
	The restriction of this 2-functor
	to $\tcatalg{\categ C}{\catoperad{Q}}_\strict$
	does not factorises in general through
	$$
	\tcatalg{\tcatalg{\categ C}{\catoperad{Q}}_\strict}{\catoperad{Q}}_\strict .
	$$
\end{remark}

\subsection{A model of pseudo-commutative monoids}

Remember that a pseudo-commutative monoid in the monoidal context
 $\categ C$ is an algebra over the categorical operad
 $\catoperad E_{\infty,cat}$ with one colour and so that $\catoperad E_{\infty,cat}(n)$
 is the contractible groupoid whose objects are pairs $(t, \sigma)$
 of a (equivalence
 class of) planar tree $t$ whose node have arity 0 or 2
 and a permutation $\sigma \in \categ S_n$.

Mac Lane's coherence tells us that a pseudo-commutative monoid is concretely the data of an
object $A$ equipped with morphisms
$\gamma : A \otimes A \to A$ and $\eta :\II \to A$ and natural transformations
$$
\gamma \circ (\gamma \otimes \id) \simeq \gamma \circ (\id \otimes \gamma)
, \quad \gamma \circ \kappa(C,C) \simeq \gamma,
\quad \gamma \circ (\eta \otimes \id) \simeq \id_{C} \simeq \gamma \circ (\id \otimes \eta)
$$
called the associator, the commutator, the left unitor
and the right unitor of $A$, and that satisfies coherence conditions (see \cite{LeGrignou14}).
A lax morphism
 between two pseudo-commutative monoids $A,B$
 is the data of a morphism $f : A \to B$ and 2-morphisms 
 $$
 \gamma_B \circ (f \otimes f) \to f \circ \gamma_A \quad \eta_B \to f \circ \eta_A
 $$
that satisfies coherence conditions.
A strong morphism of pseudo-commutative monoids is a lax
morphism whose structural 2-morphisms are 2-isomorphisms.
A strict morphism of pseudo-commutative monoids
is a lax morphism whose structural 2-morphisms are identities.

\begin{definition}
	Let us denote 
 $$
 \categ{CMon}_{\lax}(\categ{C}) = \tcatalg{\categ C}{\catoperad E_{\infty,cat}}_\lax
 $$
 the monoidal context of 
 pseudo-commutative monoids whose objects are $\catoperad E_{\infty,cat}$-algebras,
whose morphism are lax $\catoperad E_{\infty,cat}$-morphisms and whose 2-morphisms
are $\catoperad E_{\infty,cat}$-2-morphisms.
 Similarly, we denote
 \begin{align*}
	\categ{CMon}_{\oplax}(\categ{C}) &= \tcatalg{\categ C}{\catoperad E_{\infty,cat}}_\oplax;
	\\
	\categ{CMon}_{\strong}(\categ{C}) &= \tcatalg{\categ C}{\catoperad E_{\infty,cat}}_\strong;
	\\
	\categ{CMon}_{\strict}(\categ{C}) &= \tcatalg{\categ C}{\catoperad E_{\infty,cat}}_\strict.
 \end{align*}
\end{definition}

The structure of a monoidal context
on the 2-category of pseudo-commutative monoids and lax morphisms
$\categ{CMon}_{\lax}(\categ{C})$ is given by underlying
monoidal structure on $\categ C$; indeed, given two
pseudo-commutative monoids $A,A'$ their tensor product
$A \otimes A' \in \categ C$ has the structure of
a pseudo-commutative monoid given by the product
$$
(A \otimes A') \otimes (A \otimes A') \simeq
(A \otimes A) \otimes (A' \otimes A') \to
A \otimes A'
$$
and the unit
$$
\II_{\categ C} \simeq \II_{\categ C} \otimes \II_{\categ C}
\to A \otimes A'. 
$$



\section{The Hadamard tensor product of enriched coloured operads}

Let $(\categ{E}, \otimes)$ be a $\mathcal U$-cocomplete symmetric monoidal $\mathcal U$-category.
We assume that it is bilinear in the sense that the tensor product commutes with
colimits in each variable.

We describe here the
2-category of coloured operads enriched in $\categ E$ and its Hadamard tensor product.
We also describe monoidal properties of the Boardman--Vogt construction.
The reader may refer to \cite{LeGrignouLejay19}, \cite{LodayVallette12} or the articles \cite{BergerMoerdijk06} and \cite{BergerMoerdijk03} for definitions of operads.

From now on, a small set or just a set is a $\mathcal U$-small set and a large set is 
a $\mathcal U$-large set.

\subsection{Categories of trees}

\begin{definition}
Let $\set C$ be a (possibly large) set called the set of colours.
The inclusion $n \mapsto \{1, \ldots, n\}$ of the groupoid of
permutations $\categ S$ into the category of sets and functions
gives us the following comma category
\[
	\categ{S}_{\set C} \coloneqq \categ S \downarrow C
\]
whose objects $\underline c =(n, \phi)$ are the data
of a natural integer $n$
 and a function $\phi : \{1, \ldots, n\}\to \set C$ (or equivalently a tuple of colours $(c_1, \ldots, c_n)$)
 and whose morphisms from $\underline c=(n, \phi)$
 to $\underline c'=(n, \psi)$
 are permutations $\sigma \in \categ S_n$ so that $\phi = \psi \circ \sigma$.
\end{definition}

\begin{notation}
For any object
$\underline c$ of $\categ S_{\set C}$,  $|\underline c|$ will denote the length
of $\underline c$ while $\underline c [i]$ will denote the $i^{\mathrm{th}}$
colour of $\underline c$. For instance if, 
$\underline c=(c_1, \ldots, c_n)$, then
 $|\underline c|= n$ and $\underline c [i] = c_i$. Moreover for
any permutation
 $\sigma \in \categ S_n$, $\underline c^\sigma$ will denote the $n$-tuple 
 \[
 	\underline c^\sigma = (c_{\sigma(1)} ,\ldots, c_{\sigma(n)}) .
 \]
For any $\underline c' =(c'_1 , \ldots, c'_m) \in C^m$,
let us define
\begin{align*}
	\underline c \triangleleft_i \underline c'
	&\coloneqq(c_1 , \ldots, c_{i-1}, c'_1 , \ldots, c'_m, c_{i+1}, \ldots, c_n);
	\\
	\underline c \sqcup \underline c' &\coloneqq (c_1, \ldots, c_n , c'_1, \ldots, c'_m).
\end{align*}
\end{notation}

\begin{definition}\cite{MoerdijkWeiss09}
 We denote $\Dend$ the dendroidal category:
\begin{itemize}
 \item[$\triangleright$] its objects are trees ;
  \item[$\triangleright$] any tree induces a coloured operad in sets whose set of colours is the set of edges
  and whose operations are generated by vertices. Then a morphism of trees is a morphism of the induced
  coloured operads in sets. Any morphism $f$ induces in particular a function $\edge f$ between the sets of edges.
\end{itemize}
Moreover, let $\cDend$ the maximal subgroupoid of $\Dend$ and let $\aDend$ the subcategory of active morphisms, that
is the subcategory of $\Dend$ that contains all objects and all morphisms that sends leaves to leaves and the root to the root.
\end{definition}

\begin{remark}
 The morphisms of $\aDend$ are actually generated by isomorphisms,
 inner cofaces and codegeneracies.
\end{remark}

\begin{definition}
Let $\set C$ be a set. The mapping $Y \mapsto \edge Y$ induces a functor from $\Dend$ to $\categ{Sets}$. This gives us the comma category
\[
	\Dend_{\set C} = \Dend \downarrow \set C .
\]
that we call the $\set C$-coloured dendroidal category. The categories $\cDend_{\set C}$ and $
\aDend_{\set C}$ are defined similarly.
\end{definition}

\subsection{Coloured symmetric sequences}

\begin{definition}
A coloured symmetric sequence $M$ is the data of a (possibly large)
set of colours $\Ob(M)$ and a left $\categ S_{\Ob(M)}^\op\times \Ob(M)$-module in the category $\categ E$, that is
 the data of objects $M(\underline c ;c )$ for any tuple of colours $\underline c$ and any colour $c$, together with maps
 \[
 	\sigma^\ast : M(\underline c ;c ) \to M(\underline c^\sigma ;c )
 \]
for any permutations $\sigma \in \mbs_n$ (where $|\underline c|= n$), so that
\[
	\sigma^\ast \circ \mu^\ast = (\mu \circ \sigma)^\ast \quad , \quad 1^\ast = \id{} .
\]
\end{definition}

\begin{definition}
A morphism of coloured symmetric sequences from $M$ to $N$ is the data of a function
$$
f(-) : \Ob(M) \to \Ob(N)
$$
and morphisms in $\categ E$
$$
f(\underline c;c ) : M(\underline c;c ) \to N(\underline c;c )
$$
that commute with the action of the symmetric groups, that is
$$
 \sigma^\ast \circ f(\underline c;c)
 = f(\underline c^{\sigma};c ) \circ \sigma^\ast .
$$
\end{definition}

\begin{definition}
The data of coloured symmetric sequences and
their morphisms form a $\mathcal V$-small category that we denote $\catofmod{\mbs}$.
\end{definition}

\subsection{Enriched operads}

\begin{definition}
An enriched operad $\operad P$ is
the data of a coloured symmetric sequence $\operad P$ together with maps
\begin{align*}
 	\gamma_i :& \operad{P} (\underline c ; c ) \otimes 
	\operad{P} (\underline c' ; \underline c [i] ) \to
	\operad{P} (\underline c \triangleleft_i \underline c' ; c )
	\\
 	\eta_c :& \II \to \operad{P}(c;c);
\end{align*}
for any $c \in \Ob(\operad P), \underline{c}, \underline c' \in \categ S_{\Ob(\operad P)}$
and $1 \leq i \leq |\underline c|$,
that satisfy associativity and unitality conditions
and that are coherent with respect to the action of the
symmetric groups (see for instance \cite{LodayVallette12}, \cite{LeGrignouLejay21}).
If the set of colours $\Ob(\operad P)$
is large (resp. small), we often say that $\operad P$ is a large
(resp. small) enriched operad.
\end{definition}

\begin{definition}
Given two enriched operads $\operad P$ and $\operad Q$,
a morphism of enriched operads between them is a morphism
of coloured symmetric sequences that commutes with the
composition and the units.
\end{definition}

\begin{definition}
Given two morphisms of enriched
operads $f,g : \operad P \to \operad Q$,
a 2-morphism $A$ between them is the data of elements
$$
 A(c) \in \operad Q(f(c);g(c)),\ \forall c \in \Ob(\operad P);
$$
so that the following diagram commutes
$$
\begin{tikzcd}
     \operad P(\underline c ; c) 
     \ar[r,"f"] \ar[d,"g"']
     & \operad Q(f(\underline c); f(c))
     \ar[d]
     \\
     \operad Q(g(\underline c); g(c))
     \ar[r]
     & \operad Q(f(\underline c); g(c))
\end{tikzcd}
$$
for any $\underline c, c \in \set \Ob(M)^n\times \Ob(M)$.
\end{definition}

\begin{definition}
Let us denote $\Operade$
the $\mathcal V$-small
strict 2-category
of enriched operads and let us denote
$\Operadesmall$ its full sub 2-category 
spanned by small enriched operads. Actually
$\Operadesmall$ is a $\mathcal U$-2-category
in the sense that its set of objects is $\mathcal{U}$-large
and any of its mapping categories is $\mathcal U$-small.
\end{definition}

\subsection{The Hadamard tensor product}

\begin{definition}
For any two coloured-symmetric sequences $M$ and $N$, the Hadamard tensor product
$M \otimeshadamard N$
is the symmetric sequence whose objects are
$$
\Ob(M \otimeshadamard N)  = \Ob(M) \times \Ob(N)
$$
and so that
 \[
 	\left( M \otimeshadamard N \right)((c_1, c'_1), \ldots, (c_n, c'_n) ; (c,c') = M (c_1, \ldots, c_n ; c) \otimes N(c'_1, \ldots, c'_n ; c) .
 \]
The right action of $\sigma \in \categ S_n$ is diagonal
\[
\begin{tikzcd}
     \left( M \otimeshadamard N \right)((c_1,d_1), \ldots,(c_n,d_n) ; (c,d))
     \ar[d, equal]
     \\ M (\underline c ; c) \otimes N(\underline d ; d)\ar[d,"{\sigma^\ast \otimes \sigma^\ast}"]
    \\
     M (\underline c^\sigma ; c) \otimes N(\underline d^\sigma ; d)
    \ar[d, equal]
    \\ \left( M \otimeshadamard N \right)\left(((c_1,d_1), \ldots,(c_n,d_n))^\sigma ; (c,d)\right)
\end{tikzcd}
\]
This defines the structure of a monoidal category on the category of coloured-symmetric sequences
whose unit is the underlying symmetric sequence of the operad $\uCom$, whose associator, unitors and commutator are given by those of $\categ E$.
\end{definition}

\begin{definition}
 Let $\operad P$ and $\operad Q$ be two enriched operads.
 Then, their Hadamard tensor product
 $\operad P \otimeshadamard \operad Q$ is the operad whose underlying symmetric sequence is the
 Hadamard tensor product of the underlying symmetric sequences of $\operad P$ and $\operad Q$. The operad
 structure is given by the following maps
 $$
 \begin{tikzcd}
      \left(\operad P \otimeshadamard \operad Q \right) (\underline{(c,d)} ; (c,d)) 
      \otimes \left(\operad P \otimeshadamard \operad Q \right) (\underline{(c',d')} ; \underline (c[i],d[i]))
     \ar[d, "\simeq"]
      & \II \ar[d, "\simeq"]
      \\
      \operad P (\underline c ; c)  \otimes \operad P (\underline c' ; \underline c[i])
 \otimes  \operad Q (\underline d ; d)  \otimes \operad Q (\underline d' ; \underline d[i])
 \ar[d,"\gamma_i \otimes \gamma_i"]
 & \II \otimes \II \ar[d,"\eta_c \otimes \eta_c"]
    \\
    \operad P (\underline c \triangleleft_i \underline c' ; c) 
 \otimes  \operad Q (\underline d \triangleleft_i \underline d'; d)
 \ar[d, equal]
 & \operad P (c ; c) \otimes
 \operad Q (d ; d)  \ar[d, equal]
 \\
 \left(\operad P \otimeshadamard \operad Q \right) (\underline{(c \triangleleft_i \underline c',d \triangleleft_i \underline d')} ; (c,d)) 
 & \left(\operad P \otimeshadamard \operad Q \right)((c,d);(c,d)) .
 \end{tikzcd}
 $$
\end{definition}

\begin{proposition}
The Hadamard tensor product defines the structure of a monoidal
context on the 2-category of enriched operads that
lifts that of coloured symmetric sequences,
in the sense that the forgetful functor
$$
\mathrm{sk}(\Operade) \to \catofmod{\mbs}
$$
is a strict monoidal functor.
\end{proposition}

\begin{proof}
It is straightforward to show that the Hadamard tensor product
induces a monoidal structure on the category $\mathrm{sk}(\Operade)$.
It is then straightforward to see that it extends canonically
to the 2-category $\Operade$. Indeed,
given morphisms $f, f' : \operad P_1 \to \operad P_2$ and
$g, g' : \operad Q_1 \to \operad Q_2$
and 2-morphisms $A : f \to f'$ and $B : g \to g'$ we obtain
a 2-morphism $A \otimes B : f \otimes g \to f' \otimes g'$
given by the elements
$$
A(c) \otimes B(d) \in (\operad P_2 \otimeshadamard \operad Q_2)(f\otimeshadamard g(c,d);f' \otimeshadamard g'(c,d)) = \operad P_2 (f(c); f'(c)) \otimes \operad Q_2 (g(d); g'(d)) .
$$
for any $(c,d) \in \Ob(\operad P_1) \times \Ob(\operad Q_1)$. 
\end{proof}

\begin{definition}
A Hopf operad is a comonoid in the category $\mathrm{sk}(\Operad)$
with respect to the Hadamard tensor product. A comodule of such a comonoid
is called a Hopf comodule.
\end{definition}

\begin{remark}
Usually in the literature,
a Hopf operad is a \textit{cocommutative}
comonoid in the category of enriched operads
with respect to the Hadamard tensor product.
The reason why we change a standard name is that
many operads that we deal with are non cocommutative
comonoids: for instance, if the ground category
$\categ E$ is that of chain complexes,
then the Boardman--Vogt constructions
of a cocommutative Hopf operad is 
in general a non-cocommutative Hopf operad.
\end{remark}

\begin{remark}
 Hopf operads may be seen as operads enriched
 in coalgebras.
\end{remark}

\subsection{Operads with a fixed set of colours}

Let us consider the following commutative diagram of monoidal
categories and strict monoidal functors
$$
\begin{tikzcd}
     \mathrm{sk}(\Operade)
     \ar[r]\ar[rd,"\Ob(-)"']
     & 
     \catofmod{\categ S} \ar[d, "\Ob(-)"]
     \\
     &\Set .
\end{tikzcd}
$$

\begin{proposition}
 The two $\Ob(-)$ functors are
 symmetric monoidal fibrations in the sense that
 \begin{itemize}
     \itemt they are strict symmetric monoidal functors;
     \itemt they are cartesian fibrations;
     \itemt the tensor product preserves cartesian liftings.
 \end{itemize}
 Moreover, the forgetful functor from enriched operads
 to coloured symmetric sequences is a strict functor
 of symmetric monoidal fibrations in the sense that \begin{itemize}
     \itemt it is a strict symmetric monoidal functor;
     \itemt it sends cartesian morphisms to cartesian morphisms.
 \end{itemize} 
\end{proposition}

\begin{proof}
Let us consider a function $\phi : \set C \to \set C'$ and an operad $\operad P$ whose set of colours is $\set C'$. Then let $\phi^\ast(\operad P)$ be the $\set C$-coloured operad so that
$$
\phi^\ast(\operad P)(\underline c ; c) =  \operad P(\phi(\underline c); \phi(c)) .
$$
Then, the canonical morphism 
of enriched operads $\phi^\ast(\operad P) \to \operad P$
above the function $\phi$ is a cartesian lifting of $\phi$.

Moreover, we have a canonical isomorphism
$$
\phi^\ast (\operad P) \otimeshadamard \psi^\ast( \operad Q) = (\phi \times \psi)^\ast(\operad P \otimeshadamard \operad Q)
$$
for any operads $\operad P, \operad Q$
and any functions $\phi, \psi$
towards respectively $\Ob(\operad P)$ and $\Ob(\operad Q)$. Hence, 
the tensor product of two cartesian maps is cartesian.

The same constructions and arguments apply mutatis mutandis to the
case of coloured symmetric sequences. In particular, the
strict symmetric monoidal forgetful functor
from enriched operads to coloured symmetric sequences
sends cartesian maps to cartesian maps.
\end{proof}

\begin{definition}
Let $\set C$ be a set. Let us denote
\begin{itemize}
    \itemt $\catofmod{\categ S_C}$ the category of $\set C$-coloured symmetric sequences that is the fiber of the colours functor $\catofmod{\categ S} \to \Set$ over the set $\set C$;
    \itemt $\Operad_{\categ E,\set C}$ the category of $\set C$-coloured operads that is the fiber of the functor $\mathrm{sk}(\Operade) \to \Set$ over the set $\set C$.
\end{itemize}
\end{definition}

\begin{proposition}\cite[Theorem 12.7]{Shulman08}
For any set $\set C$, the category $\Operad_{\categ E,\set C}$ inherit from $\Operade$
the structure of a symmetric monoidal category.
Moreover, for any function $\phi : \set C \to \set D$, the functor
$$
\phi^\ast : \Operad_{\categ E,\set D} \to \Operad_{\categ E,\set C}
$$
has the canonical structure of a strong symmetric monoidal functor.
Finally, for any two composable functions $\phi, \psi$,
the natural îsomorphism
$(\phi \circ \psi)^\ast \simeq \psi^\ast \circ \phi^\ast$
is a monoidal natural transformation.
The same phenomenon holds if we replace enriched operads by
coloured symmetric sequences.
\end{proposition}

Let $\set C$ be a set. Let us describe the structure
of a symmetric monoidal category on $\Operad_{\categ E, \set C}$.
The tensor product is
$$
\operad P \otimes_{\set C} \operad Q
= \mathrm{diag}_{\set C}^\ast( \operad P \otimeshadamard \operad Q)
$$
where $\mathrm{diag}_{\set C}$ is the diagonal map of $\set C$.
The tensor unit is $\pi_{\set C}^\ast(\uCom)$  where $\pi_{\set C}$ is the unique function from $\set C$ to $\ast$.

\begin{definition}
 Let $M$ be a coloured symmetric sequence whose set of colours is $\set C$. Then, for any $\set C$-coloured tree $Y$ we define
 $$
 \bigotimes_Y M = \bigotimes_{v \in \verte(T)} M(v)
 $$
 where
 $$
 M(v) = \varinjlim_{(\underline c; c) \simeq v} M(\underline c;c) .
 $$
 This defines a bifunctor
 $$
 \catofmod{\categ S_C} \times (\cDend_{\set C})^\op \to E .
 $$
\end{definition}

\begin{proposition}
The construction $\operad P, Y \mapsto \bigotimes_Y \operad P$ defines a bifunctor
$$
\Operad_{\categ E,\set C} \times (\aDend_{\set C})^\op \to \categ E .
$$
\end{proposition}

\begin{proof}
The inner face maps correspond to composition within the operad, while the degeneracies correspond to the units.
\end{proof}

One can also describe $\set C$-coloured operads as monoids for a
monoidal structure on $\set C$-coloured symmetric sequences.
Let $M$ and $N$ be two $\set C$-coloured symmetric sequences.
Their \textit{composite product} $M \triangleleft N$ is
the symmetric sequence given by the following formula 
\[
 	M \triangleleft N (\underline c;c) = \int_{\underline c'\in \categ S_{\set C}} M(\underline c'; c) \otimes \int_{(\underline c'_i)_{i = 1}^{|\underline{c}'|}} \left(\bigotimes_{i=1}^{|\underline c'|}
	N(\underline c'_i ; \underline c'[i])\right) \otimes \Hom{\categ S_{\set C}}{\underline c}{\sqcup_i \underline c'_i} .
\]

\begin{proposition}\cite{LodayVallette12}\cite{LeGrignouLejay21}
 The bifunctor $-\triangleleft -$ is the tensor product
 of a monoidal structure on symmetric sequences
 whose unit $\III_{\set C}$ is
 \[
 	\III_{\set C}(\underline c ; c)=	
\begin{cases}
 \II \text{ if }\underline c =c ;\\
 \emptyset \text{ otherwise.}
\end{cases}
 \]
 Moreover, the category of $\set C$-coloured
 enriched
 operads is canonically isomorphic to the category
 of $\triangleleft$-monoids.
\end{proposition}

\begin{remark}
The fact that the Hadamard tensor product on $\set C$-coloured symmetric sequences extends
to $\set C$-coloured enriched operads is linked to the
following facts.
\begin{itemize}
 \itemt The monad that encodes $\set C$-coloured enriched operads
 within $\set C$-coloured symmetric sequences
 is a Hopf monad with respect to the Hadamard tensor product
 on symmetric sequences;
 see \cite{Moerdijk02} and \cite{LeGrignou22}
 for the definition of a Hopf monad.
 \itemt The composition tensor product of
 $\set C$-coloured symmetric sequences $\triangleleft$
 distributes the Hadamard tensor product.
\end{itemize}
\end{remark}

\subsection{Segments and intervals}

\begin{definition}[Segments and intervals]\cite{BergerMoerdijk06}
A segment of $\categ{E}$ is an object $H$ of $\categ{E}$ together with maps
 \[
 \II \sqcup \II \xrightarrow{(\delta^0_H ,\delta^1_H)} H \xrightarrow{\sigma_H} \II
 \]
 that factorise the morphism
 $\II \sqcup \II \to \II$ together with a map
 $\gamma_H: H \otimes H \to H$ such that 
\begin{itemize}
 \item[$\triangleright$] the product $\gamma_H$ is associative, that is $\gamma_H (Id_H \otimes \gamma_H)=\gamma_H (\gamma_H\otimes Id_H)$ and $\gamma_H \circ \tau = \gamma_H$,
 \item[$\triangleright$] the product has a unit given by $\delta^0_H: \II \to H$,
 \item[$\triangleright$] the morphism $\sigma_H$ is a morphism of monoids (that is an augmentation),
 \item[$\triangleright$] the morphism $\delta^1_H: \II \to H$ is absorbing, that is the following diagram commutes
  \[
	\begin{tikzcd}[ampersand replacement=\&]
		H \simeq \II \otimes H \arrow[r,"\delta^1_H \otimes Id" ] \arrow[d, "\sigma_H"']
		\& H \otimes H \arrow[d, "\gamma_H"]
		\& H \simeq H \otimes \II \arrow[l,"Id \otimes \delta^1_H"'] \arrow[d,"\sigma_H"] \\
		\II \arrow[r,"\delta^1_H"'] 
		\& H
		\& \II \arrow[l,"\delta^1_H"]\ .
	\end{tikzcd}
\]

\end{itemize}
A morphism of segments from $(H,\delta^0_H,\delta^1_H,\sigma_H,\gamma_H)$
to $(H',\delta^0_{H'},\delta^1_{H'},\sigma_{H'},\gamma_{H'})$ is a
morphism of $f: H \to H'$ that commutes with all the structures.
We denote $\categ{Seg}$ the category of segments.
\end{definition}

\begin{definition}
 A segment $(H,\delta^0_H,\delta^1_H,\sigma_H,\gamma_H)$ is said to
 be commutative if the product $\gamma_H$ is commutative.
\end{definition}

\begin{definition}
If $\categ E$ is a monoidal model category, an interval is a segment such that the map $(\delta^0_H ,\delta^1_H)$ is a cofibration and the map $\sigma_H$ is a weak equivalence.
\end{definition}

\begin{remark}
 Note that these are exactly the notions of segments and intervals introduced in \cite{BergerMoerdijk06} but they differ from  the notions of segments and
 intervals that I used in \cite{LeGrignou18}.
\end{remark}

The tensor product of $\categ{E}$ induces a monoidal structure on
the category of segments as well as on the full subcategory
of commutative segments. For instance, the tensor product of a segment
$(H,\delta^0_H,\delta^1_H,\sigma_H,\gamma_H)$ with a segment
$(H',\delta^0_{H'},\delta^1_{H'},\sigma_{H'},\gamma_{H'})$ is the segment
\[
	\mathbb 1 \sqcup  \mathbb 1 \simeq  \mathbb 1\otimes \mathbb 1 \sqcup \mathbb 1 \otimes \mathbb 1
	 \xrightarrow{(\delta^0_H\otimes \delta^0_{H'}, \delta^1_H\otimes \delta^1_{H'})} H \otimes H' \xrightarrow{\sigma_{H} \otimes \sigma_{H'}} 
	 \mathbb 1  \otimes \mathbb 1 \simeq \mathbb 1\ ;
\]

\begin{definition}
A Hopf segment in $\categ{E}$ is a comonoid in the monoidal category of segments. In a monoidal model category, a Hopf interval is a Hopf segment whose
underlying segment is an interval.
\end{definition}

\begin{remark}[\cite{Maltsiniotis09}\cite{LeGrignou18}]
Since the monoidal category $\categ E$ is a bilinear, the category of
segments is canonically equivalent to the category of
left adjoint functors
from cubical sets with connections to $\categ E$.
We refer the reader to \cite{Maltsiniotis09} and \cite{Cisinski14}
for a definition of the category of cubical sets with connections.
\end{remark}

\subsection{The Boardman--Vogt construction}

\begin{definition}
Let $H$ be a segment in $\categ{E}$ and let $Y$ be a coloured tree. Then
let
 \[
 	\bigotimes^Y H := \bigotimes_{\inner{Y}}H .
 \]
where $\inner Y$ is the set of inner edges
of $Y$. In particular, if $Y$ has no internal edge,
then $\bigotimes^Y H := \II$.
\end{definition} 
 
\begin{proposition}
Let $\set C$ be a set. The formula $\bigotimes^Y H$ induces a bifunctor
 \[
 	\bigotimes^{-} - : \aDend_{\set C} \times \categ{Seg} \to \categ{E}\ . 
 \]
\end{proposition}

\begin{proof}
 It is clear that it induces a bifunctor $\cDend_{\set C} \times \categ{Seg} \to \categ{E}$.
 Besides, let $\delta : Y \to Y'$ be an inner coface. The map
 \[
 	\bigotimes^Y H \simeq \II \otimes \bigotimes^Y H
	\xrightarrow{\delta_0 \otimes \id} 
	H \otimes \bigotimes^Y H \simeq \bigotimes^{Y'} H\ ,
 \]
 gives us the functoriality of the formula $\bigotimes^{-} H$ with respect to this inner coface.
Let $\sigma : Y \to Y'$
 be a codegeneracy. If the two edges $e$ and $e'$ (with $e$ below $e'$) of $Y$ that will merge in $Y'$ are internal, then 
 the functor $\bigotimes^- H$ sends $\sigma$ to the map
 \[
 	 \bigotimes^Y H= \bigotimes_{\inner Y} H \simeq \bigotimes_{e,e'}H \otimes \bigotimes_{\inner Y - \{e,e'\}} H
	\xrightarrow{m \otimes \id^{n-2}} H \otimes \bigotimes_{\inner Y - \{e,e'\}} H
	\simeq \bigotimes_{\inner{Y'}}  H = \bigotimes^{Y'} H\ .
 \]
If one of these two edges is external and the other one is internal (for instance $e$ is internal),
 the functor $\bigotimes^- H$ sends $\sigma$ to the map
 \[
 	 \bigotimes^Y H= \bigotimes_{\inner Y} \simeq H \otimes \bigotimes_{\inner Y - \{e\}} H
	\xrightarrow{\sigma_H \otimes \id^{n-1}} \bigotimes_{\inner Y - \{e\}} H \simeq  \bigotimes_{\inner{Y'}} H
	= \bigotimes^{Y'} H,
\]
Finally, if these two edges are external, which is only possible if $Y$ is the corolla with one leaf,
then the functor $\bigotimes^- H$ sends $\sigma$ to the map $\II \simeq \II$.
\end{proof}

\begin{remark}
For codegeneracies that will merge two inner edges $e$ below $e'$ to a new inner edge $e''$,
the labelling of $e''$ is obtained from that of $e$ and $e''$ using the product $m$.
The first input of $m$ refers to the labelling of $e$ while the second input of $m$ refers
to the labelling of $e'$. This is a convention and we could have make the other choice.
\end{remark}

\begin{definition}[\cite{BergerMoerdijk06}]
For any enriched operad $\operad{P}$ and any segment
$H$, the Boardman--Vogt construction of $\operad{P}$ with respect to
$H$ is an enriched operad with the same set of colours and whose underlying
coloured symmetric sequence is given by the following formula
 \[
 	W_H \operad{P} (\underline c ; c):= \int_{Y \in \aDend_{\set C}\op} \left(\bigotimes^{Y} H \otimes \bigotimes_Y \operad{P} \right)
	\otimes \mathrm{Iso}_{\categ{Set}/\set C}(\underline c , \leaves{Y}) \ .
 \]
 The unit maps are
 \[
 	\eta_c : \II \xrightarrow{\simeq} \bigotimes^{|_c}H \otimes \bigotimes_{|_c} \operad P \to W_H \operad{P} (c ; c) ,
	\quad c \in \Ob(\operad P).
 \]
 The composition maps $\gamma_i$ are given on generators by the following composition
 \[
\begin{tikzcd}
	\bigotimes^Y H \otimes \bigotimes_Y \operad P \otimes \{\phi\}
	\otimes \bigotimes^{Y'} H \otimes \bigotimes_{Y'}\operad P \otimes \{ \psi\}
	\arrow[d, "\simeq"]
	\\
	\bigotimes^Y H \otimes \II \otimes \bigotimes^{Y'} H
	\otimes \bigotimes_Y \operad P \otimes \bigotimes_{Y'}\operad P
	\otimes \{\phi \triangleleft_i \psi\}
	\arrow[d,"\id \otimes \delta_1 \otimes \id"]
	\\
	\bigotimes^Y H \otimes H \otimes \bigotimes^{Y'} H
	\otimes \bigotimes_Y \operad P \otimes \bigotimes_{Y'}\operad P
	\otimes \{\phi \triangleleft_i \psi\}
	\arrow[d, equal]
	\\
	\bigotimes^{Y \sqcup_{\phi(i)} Y' } H
	\otimes \bigotimes_{Y \sqcup_{\phi(i)} Y'} \operad P 
	\otimes \{\phi \triangleleft_i \psi\}
\end{tikzcd}
 \]
 where the tree $Y \sqcup_{\phi(i)} Y'$ is the tree obtained
 from $Y$ and $Y'$ by gluing the root edge of $Y'$ with the leaf $\phi(i)$
 of $Y$. Moreover, the map
 $$
 \phi \triangleleft_i \psi : \underline c \triangleleft_i \underline c'
 \to \leaves{Y \sqcup_{\phi(i)} Y'}
 $$
 sends an element $\underline{c}[k]$ to its image through $\phi$
 for $k \neq i$
 and an element $\underline{c}[j]$ to its image through $\psi$.
\end{definition}

\begin{proposition}\cite{BergerMoerdijk06}
The construction $H, \operad P \mapsto W_H \operad P$ defines a functor
 \[
 	W_- - : \categ{Seg} \times \mathrm{sk}(\Operade) \to \mathrm{sk}(\Operade).
 \]
\end{proposition}

\begin{proof}
Such a functor sends a pairs of morphisms
$(u,f) : (H, \operad P) \to (H', \operad P')$
to the morphism of enriched operads induced by the maps
$$
	\{\phi\} \otimes \bigotimes^Y H \otimes \bigotimes_Y \operad P
	\xrightarrow{\id \otimes \bigotimes^Y u \otimes \bigotimes_Y f } \{\phi\} \otimes \bigotimes^{f(Y)} H' \otimes \bigotimes_{f(Y)} \operad Q
$$
where $f(Y)$ is the $\Ob(\operad Q)$-coloured tree whose underlying tree is the same as
$Y$ and whose colouring is the map
$$
\edge Y \to \Ob(\operad P) \xrightarrow{f} \Ob(\operad Q) .
$$
\end{proof}

\subsection{Monoidal structure of the Boardman--Vogt construction}

\begin{proposition}
 The functor $W_- - : \categ{Seg} \times \mathrm{sk}(\Operade)
 \to \mathrm{sk}(\Operade)$
 is oplax symmetric monoidal for
 the symmetric monoidal structure $\otimes \times \otimeshadamard$ on
 the product category $\categ{Seg} \times \mathrm{sk}(\Operade)$.
\end{proposition}

\begin{proof}
For any two segments $H,H'$, any two enriched operads
$\operad{P},\operad{P}'$ and any
$\Ob(\operad P) \times \Ob(\operad P)$-coloured
tree
$Y$, let us consider the following map
\small
\[
\begin{tikzcd}
 	\left(\bigotimes^{Y} (H\otimes H') \otimes \bigotimes_Y 
	(\operad{P}\otimeshadamard \operad{P}') \right)
	\otimes \mathrm{Iso}_{\categ{Set}/(\Ob(\operad P) \times \Ob(\operad P'))}
	(((c_1, d_1), \ldots, (c_n, d_n)) , \leaves{Y}) 
	\arrow[d]\\
	\left(\bigotimes^{Y} H \otimes \bigotimes_Y (\operad{P})
	\otimes \mathrm{Iso}_{\categ{Set}/\Ob(\operad P)}((c_1, \ldots, c_n), \leaves{Y}) 
	\right)
	\otimes \left( \bigotimes^{Y} H' \otimes \bigotimes_Y (\operad{P}')
	\otimes \mathrm{Iso}_{\categ{Set}/\Ob(\operad P')}((d_1, \ldots, d_n), \leaves{Y}) 
	\right)
	\arrow[d]\\
	W_H \operad{P} (c_1, \ldots, c_n;c) \otimes W_{H'} \operad{P}' (d_1, \ldots, d_n;d)
	\ar[d, equal]
	\\
	(W_H \operad{P} \otimeshadamard W_{H'} \operad{P}')((c_1, d_1), \ldots,(c_n,d_n) ; (c,d))\ .
\end{tikzcd}
 \]
 \par
Taking the coend over varying trees, we get a moprhism of
$\Ob(\operad P) \times \Ob(\operad P')$-coloured 
symmetric sequences
\[
	W_{H\otimes H'} (\operad{P} \otimeshadamard \operad{P}')
	\to W_H \operad{P} \otimeshadamard W_{H'} \operad{P}'\ .
\]
This is actually of morphism of $\Ob(\operad P) \times \Ob(\operad P')$-coloured 
enriched operads. Besides, we have a canonical isomorphism
\[
	W_{\mathbb 1} \uCom \simeq \uCom\ .
\]
 A long but straightforward checking shows that this defines
 the structure of an
 oplax symmetric monoidal functor on $W_- - $.
\end{proof}

\begin{remark}
If one works with planar operads,
the Boardman--Vogt construction becomes a strong symmetric monoidal
functor.
\end{remark}

\begin{corollary}
Let $H$ be a Hopf segment. Then the endofunctor $W_H -$
of $\sk(\Operade)$ is oplax monoidal. 
\end{corollary}

\begin{proof}
 The oplax monoidal structure is given by the following natural map
 \[
 	W_H (\operad{P} \otimeshadamard \operad{P}') \to W_{H \otimes H} (\operad{P} \otimeshadamard \operad{P}')
	\to W_H \operad{P} \otimeshadamard W_H \operad{P}' \ .
 \]
\end{proof}

\begin{corollary}\label{corollaryhopf}
Let $\operad{Q}$ be a Hopf operad and let $H$ be a Hopf segment.
Then $W_H \operad{Q}$ inherits the canonical structure of a Hopf operad.
Moreover, for any $\operad{Q}$-left comodule
$\operad{P}$, $W_H \operad{P}$ inherits the structure of a
$W_H \operad{Q}$-comodule. The same holds for
right comodules and bi-comodules.
\end{corollary}

\subsection{Planar operads}
 
\begin{definition}
 A planar enriched operad $\operad P$
 is the data of a set of colours $\Ob(\operad P)$
 and elements $\operad P(\underline c ;c)$ of the
 ground category $\categ E$
 for any tuple of colours $\underline c$ and any colour
 $c \in \Ob(\operad P)$ together with maps
\begin{align*}
 	\gamma_i : \operad P (\underline c ; c) \otimes \operad P (\underline c'; \underline c[i]) 
	\to \operad P (\underline c \triangleleft_i \underline c' ; c) ;
	\\
	\eta_c : \II \to \operad P (c ; c) ;
\end{align*}
that satisfies the same associativity and unitality conditions as enriched operads.
A morphisms of planar enriched operads
from $\operad P$ to $\operad Q$
is the data of a function $\phi: \Ob(\operad P) \to \Ob(\operad Q)$
together with morphisms
\[
	\operad P (\underline c ; c) \to \operad Q (f(\underline c); f(c))
\]
that commutes with the structural maps $\gamma_i$ and $\eta_c$.
We denote $\mathsf{pl}\Operade$
the category of planar enriched operads.
\end{definition}

We have an adjunction relating coloured planar operad to coloured operad
\[
\begin{tikzcd}
 	\mathsf{pl}\Operade
	\arrow[rr, shift left, "-\otimes \categ S"]
	&& \sk(\Operade)
	\arrow[ll, shift left, "\forget^{\categ S}"]
\end{tikzcd}
\]
whose right adjoint is the forgetful functor and whose left adjoint sends
a planar enriched operad $\operad P$
to the enriched operad $\operad P_{\categ S}$ with the same set of colours
and defined as
\begin{align*}
	\operad P_{\categ S} (\underline c;c) 
	&= \coprod_{\underline c'\in {\set C}^n}
	\operad P(\underline c';c) \otimes \Hom{\categ S_{\set C}}{\underline c}{\underline c'}
	\\
	&=\coprod_{\sigma \in \categ S_n} \operad P(\underline c^{\sigma^{-1}};c) \otimes \{\sigma\} ,
\end{align*}
where $n= |\underline c|$. The action
$\mu^\ast : \operad P_{\categ S}(\underline c;c) 
\to \operad P_{\categ S} (\underline c^\mu;c)$ of 
the symmetric groups is defined by maps of the form
\[
	\operad P(\underline c^{\sigma^{-1}};c) \otimes \{\sigma\} \to \operad P(\underline c^{\sigma^{-1}};c) \otimes \{\sigma \circ \mu\}
	= \operad P((\underline c^\mu)^{(\sigma \circ \mu)^{-1}};c) \otimes \{\sigma \circ \mu\} ,
\]
while the operadic composition $\gamma_i$ is defined as
\[
\begin{tikzcd}
	{\left( \operad P(\underline c^{\sigma^{-1}};c) \otimes \{\sigma\} \right) \otimes 
	\left( \operad P(\underline c'^{\mu^{-1}};\underline c[i]) \otimes \{\mu\} \right) }
	\arrow[d, equal]
	\\
	{ \operad P(\underline c^{\sigma^{-1}};c) \otimes 
	\operad P(\underline c'^{\mu^{-1}};\underline c[i]) \otimes \{\sigma \triangleleft_i \mu\} }
	\arrow[d,"{\gamma_{\sigma(i)}^{\mathrm{pl}}}"]
	\\
	\operad P(\underline c^{\sigma^{-1}}\triangleleft_{\sigma(i)} \underline c'^{\mu^{-1}};c) \otimes \{\sigma \triangleleft_i \mu\}
	\arrow[d, equal]
	\\
	\operad P((\underline c \triangleleft_{i} \underline c')^{(\sigma \triangleleft_i \mu)^{-1}};c) \otimes \{\sigma \triangleleft_i \mu\} ,
\end{tikzcd}
\]
where ${\gamma_{\sigma(i)}^{\mathrm{pl}}}$ denotes the planar operadic composition in the planar operad $\operad P$
and where $\sigma \triangleleft_i \mu$ is the permutation
$$
\begin{tikzcd}
	\{1, \ldots, n+m-1\}
	\ar[d, equal]
	\\
	\{1, \ldots, i-1 \} \sqcup \{i, \ldots, i+m-1\}
	\sqcup \{i+m, \ldots, n+m-1\}
	\ar[d, "\simeq"]
	\\
	\{1, \ldots, i-1 \} \sqcup \{1, \ldots, m\}
	\sqcup \{i+ 1, \ldots, n\}
	\ar[d, equal]
	\\
	\{1, \ldots, n \}\backslash \{i\} \sqcup \{1, \ldots, m\}
	\ar[d,  "\sigma \sqcup \mu"]
	\\
	\{1, \ldots, n \}\backslash \{\sigma(i)\} \sqcup \{1, \ldots, m\}
	\ar[d, "\simeq"]
	\\
	\{1, \ldots, n+m-1\}
\end{tikzcd}
$$
where $n=|\underline{c}|$ and $m=|\underline{c'}|$.


 \section{Enriched symmetric monoidal categories and algebras}
 
In this section, we describe symmetric monoidal categories enriched
over the ground category $\categ E$. We show how they are related to enriched
operads and we describe how coalgebraic 
structures on enriched operads
induce monoidal structures on their categories of algebras
in a symmetric monoidal enriched category $\ecateg C$.

 \subsection{The monoidal context of enriched categories}
 
 \begin{definition}
 An enriched category $\ecateg C$ is the data of a set (possibly large) called
 the set of objects (or colours) and denoted $\Ob(\ecateg C)$ together with, for any two colours
 $x,y$, an object $\ecateg C(x,y) \in \categ E$ and with an associative composition
 $$
 \ecateg C (y,z) \otimes \ecateg C (x,y)
 \to \ecateg C (x,z)
 $$
 with units $\id_x \in \ecateg C (x,x)$.
 \end{definition}

 \begin{definition}
	An enriched category $\ecateg C$ is called small if
	its set of objects is small.
	\end{definition}
 
 \begin{definition}
 A functor of enriched categories from $\ecateg C$ to
 $\ecateg D$ is the data of
 \begin{itemize}
     \itemt a function $f(-) : \Ob(\ecateg C) \to \Ob(\ecateg D)$;
     \itemt for any two colours of $\ecateg C$ $x,y$,
     a morphism
     $$
     f(x,y) : \ecateg C (x,y) \to \ecateg C (f(x),f(y))
     $$
     that commutes with units and compositions.
 \end{itemize}
 \end{definition}
 
 \begin{definition}
 Given two functors of enriched categories $f,g :\ecateg C \to \ecateg D$,
 a natural transformation between them is the data of elements
 $$
 A(x) \in \ecateg D (f(x), g(x)), \  x \in \Ob(\ecateg C) ,
 $$
 so that the following square is commutative
 $$
 \begin{tikzcd}
      \ecateg{C}(x,y)
      \ar[r, "{g(x,y) \otimes A(x)}"] \ar[d, "{A(y) \otimes f(x,y)}"]
      & \ecateg{D}(g(x),g(y)) \otimes \ecateg{D}(f(x),g(x))
      \ar[d]
      \\
      \ecateg{D}(f(y),g(y)) \otimes \ecateg{D}(f(x),f(y))
      \ar[r]
      & \ecateg{D}(f(x),g(y))
 \end{tikzcd}
 $$
 for any objects $x,y \in \Ob(\ecateg C)$.
 \end{definition}

 \begin{definition}
 We denote $\categ{Cat}_{\categ E}$ the strict 2-category of
enriched categories . This a $\mathcal V$ strict 2-category and
 hence a $\mathcal W$-small strict 2-category.
 \end{definition}
 
 \begin{definition}
 Let $\ecateg C, \ecateg D$ be two enriched categories.
 Their tensor product $\ecateg C \otimes \ecateg D$ is the enriched category
so that
 \begin{itemize}
     \itemt $\Ob(\ecateg C \otimes \ecateg D) = \Ob(\ecateg C) \times \Ob(\ecateg D)$;
     \itemt $(\ecateg C \otimes \ecateg D) ((x,x'),(y,y')) =
	 \ecateg{C} (x,y) \otimes \ecateg{D} (x',y')$;
     \itemt the composition is given as follows
     $$
     \begin{tikzcd}
     (\ecateg C \otimes \ecateg D) ((y,y'),(z,z')) \otimes (\ecateg C \otimes \ecateg D) ((x,x'),(y,y'))
     \ar[d, equal]
     \\
     \ecateg{C} (y,z) \otimes \ecateg{D} (y',z') \otimes \ecateg{C} (x,y) \otimes \ecateg{D} (x',y')
     \ar[d,"\simeq"]
     \\
     \ecateg{C} (y,z) \otimes \ecateg{C} (x,y) \otimes \ecateg{D} (y',z') \otimes \ecateg{D} (x',y')
     \ar[d]
     \\
     \ecateg{C} (x,z) \otimes \ecateg{D} (x',z')
     \ar[d, equal]
     \\
     (\ecateg C \otimes \ecateg D) ((y,y'),(z,z')) ;
     \end{tikzcd}
     $$
     \itemt the unit of the object $(x,x')$ is $\id_x \otimes \id_{x'}$.
 \end{itemize}
 \end{definition}
 
 \begin{definition}
  Let $\eII$ be the enriched category with one object $\ast$ and so that
  $$
  \eII(\ast, \ast) = \II_{\categ E}.
  $$
 \end{definition}
 
 \begin{proposition}
 The tensor product described just above induces the structure of a monoidal context on the 2-category
 $\categ{Cat}_{\categ E}$ whose unit is $\eII$.
 \end{proposition}
 
 \begin{proof}
 This tensor product induces the structure of a monoidal
 category on $\mathrm{sk}(\categ{Cat}_{\categ E})$.
 Moreover, this monoidal structure extends to the 2-category
 $\categ{Cat}_{\categ E}$. For instance, given two enriched functors
 $f,f' : \ecateg C \to \ecateg C'$ and $g,g' : \ecateg D \to \ecateg D'$
 and natural transformations $A : f \to f'$ and $B : g \to g$,
 the natural transformation $A \otimes B : f \otimes g \to f' \otimes g'$
 is given by the elements $A(c) \otimes B(d)$ for
 $(c,d) \in \Ob(\ecateg C) \times \Ob(\ecateg D)$.
 \end{proof}
 
 \begin{definition}
  For an enriched category $\ecateg C$, we denote $\cat(\ecateg C)$ the underlying category, that is the category whose set of objects is $\Ob(\ecateg C)$ and so that
  $$
  \Hom{\cat(\ecateg C)}{x}{y} = \Hom{\categ E}{\II_{\categ E}}{\ecateg C(x,y)} .
  $$
  This defines a lax context functor
  $$
  \cat : \categ{Cat}_{\categ E} \to \categ{Cats}_{\mathcal U} .
  $$
 \end{definition}
 
 \subsection{The monoidal context of enriched symmetric monoidal categories}
 
 \begin{definition}
Let us recall that $\categ{CMon}_{\mathrm{lax}}
(\categ{Cat}_{\categ E})$ is the monoidal context
of pseudo-commutative monoids and lax morphisms
in the monoidal context of $\categ E$-enriched categories.
We will call its objects, morphisms and 2-morphisms respectively
symmetric monoidal enriched categories,
lax symmetric monoidal functors and
monoidal enriched natural transformations.
 \end{definition}
 
 Unravelling the definitions, the data of a symmetric monoidal enriched category is equivalent to the data of
  \begin{itemize}
      \itemt an enriched category $\ecateg C$;
      \itemt a functor of enriched categories
\begin{align*}
 \gamma : \ecateg C \otimes \ecateg C &\to \ecateg C
 \\
 x,y & \mapsto x \otimes_{\ecateg C} y ;
\end{align*}
      \itemt an element $\II_{\ecateg C}$;
      \itemt invertible enriched natural transformations
      $$
      x \otimes_{\ecateg C} (y \otimes_{\ecateg C} z) \simeq (x\otimes_{\ecateg C} y ) \otimes_{\ecateg C} z,
      \quad
      x \otimes_{\ecateg C} y \simeq y \otimes_{\ecateg C} x,
      \quad
      x \otimes_{\ecateg C} \II_{\ecateg C} \simeq x;
      $$
      \itemt so that the induced bifunctor
      $$
      \cat (\ecateg C) \times \cat(\ecateg C) \to \cat (\ecateg C \otimes \ecateg C) \xrightarrow{\cat(\gamma)} \cat(\ecateg C) .
      $$
      the element $\II_{\ecateg C}$ and the induced natural transformations involving $\cat (\ecateg C)$ instead of
      $\ecateg C$ make $\cat (\ecateg C)$ a symmetric monoidal category.
  \end{itemize}
 Moreover, the data of a lax symmetric monoidal functor
 between symmetric monoidal enriched categories $\ecateg C$ and $\ecateg D$ is equivalent to the data of
  \begin{itemize}
      \itemt an enriched functor $f: \ecateg C \to \ecateg D$;
      \itemt an enriched natural transformation $l: \gamma_D \circ (f \otimes f) \to f \circ \gamma_C$;
      \itemt a map $\II_{\ecateg D} \to f(\II_{\ecateg C})$; 
      \itemt so that the induced functor $\cat(f)$, the induced natural transformation $\cat(l)$ and the previous map
      form a lax symmetric monoidal functor between symmetric monoidal categories.
  \end{itemize}
  Finally, a 2-morphism in $\categ{CMon}_{\lax}(\categ{Cat}_{\categ E})$ between two lax symmetric monoidal functors
  $f$ and $f'$ is the data of an enriched natural transformation $A : f \to f'$ so that the following diagrams commute
  $$
  \begin{tikzcd}
       \gamma \circ (f \otimes f)
       \ar[r] \ar[d]
       & f \circ \gamma
       \ar[d]
       \\
       \gamma \circ (f' \otimes f')
       \ar[r]
       & f' \circ \gamma
  \end{tikzcd}
  \quad
    \begin{tikzcd}
       \II_{\ecateg D}
       \ar[r] \ar[rd]
       & f(\II_{\ecateg C})
       \ar[d, "A(\II_{\ecateg C})"]
       \\
       & f'(\II_{\ecateg C}).
  \end{tikzcd}
  $$

\subsection{Multiple tensors}

Let $\ecateg C$ be a symmetric monoidal enriched category and let $n \geq 3$.
	In the same way as in \cite[Definition 23 and Definition24]{LeGrignou22},
	one gets an enriched functor
	from $\ecateg C^{\otimes n}$
	to $\ecateg C$ for any binary planar tree with $n$-leaves:
	roughly, one composes operadically the tensor product
	$-\otimes -:\ecateg C \otimes \ecateg C \to \ecateg C$
	along this tree. The associators induces 2-isomorphisms
	between these functors so that we get a contractible groupoid
	(by Mac Lane's coherence, see \cite{MacLane63},\cite[Proposition 22]{LeGrignou22}).

\begin{definition}
	The colimit in $\Cate(\ecateg C^{\otimes n}, \ecateg C)$
	of this contractible groupoid
	is the $n$-tensor product of $\ecateg C$
	\begin{align*}
		\otimes^n: \ecateg C^{\otimes n} & \to \ecateg C
		\\
		(x_1, \ldots, x_n) &\mapsto x_1 \otimes \cdots \otimes x_n .
	\end{align*}
	We also define the 2-tensor product of $\ecateg C$ to be the tensor product, the
	1-tensor product to be the identity of $\ecateg C$ and the 0-tensor
	product to be the unit $\III \to \ecateg C$.
\end{definition}

Again, for any planar tree $t$ with $n$ leaves
and any permutation $\sigma \in \categ S_n$,
one obtains an enriched functor from $\ecateg C^{\otimes n}$
to $\ecateg C$ by first permuting the tensors using $\sigma^\ast$
and then applying the operadic composition of multi-tensors along the
planar tree $t$.
Then, one can consider the groupoid whose objects are such
pairs $(t, \sigma)$ and whose morphisms are those of
$\Cate(\ecateg C^{\otimes n}, \ecateg C)$ between the induced enriched functors
and that are generated by operadic compositions
of the associator, the commutator, the unitors
and the reduction of binary subtrees
into multi-tensors.
One can check after Mac Lane's coherence that such a groupoid is contractible.

 \subsection{The End 2-functor}
 
 \begin{definition}
 For any symmetric monoidal $\categ E$-enriched category $\ecateg C$,
 let $\End(\ecateg C)$ be the enriched operad whose colours are $\Ob(\ecateg C)$
 and so that
 $$
 \End(\ecateg C)(x_1, \ldots, x_n ; x) = \ecateg C (x_1 \otimes \cdots \otimes x_n , x) .
 $$
 The operadic composition is induced by the composition in $\ecateg C$
 and the action of symmetric groups is induced by the commutator of $\ecateg C$.
 \end{definition}
 
 \begin{proposition}
 The $\End$ construction induces a
 strict 2-functor from the
 2-category of symmetric monoidal enriched categories and lax enriched functors
 to the 2-category of enriched operads:
 $$
\End : \cmonlax(\Cate) \to \Operade .
 $$
 \end{proposition}
 
 \begin{proof}
 For any lax monoidal functor between symmetric monoidal enriched categories
 $f :\ecateg C \to \ecateg D$, we obtain a morphism of enriched operads
 $$
\End(f) : \End(\ecateg C) \to \End(\ecateg D)
 $$
 whose underlying function on colours is the underlying function of $f$ on objects,
 and which is given by the maps
 $$
 \ecateg C(x_1 \otimes \cdots \otimes x_n , x)
 \to \ecateg D(f(x_1 \otimes \cdots \otimes x_n) , f(x))
 \to \ecateg D(f(x_1) \otimes \cdots \otimes f(x_n) , f(x)).
 $$
 For any monoidal natural transformation $A: f\to g$ between lax symmetric monoidal
 functors $f,g: \ecateg C \to \ecateg D$, we obtain an 2-morphism of enriched operads
 given by the same elements $A(x) \in \ecateg D(f(x), g(x))$ for
 $x \in \Ob(\ecateg C)$.
 
 A straightforward check shows that this defines a 2-functor.
 \end{proof}
 
 \begin{proposition}
 The 2-functor $\End$ is strictly fully faithful. 
 \end{proposition}
 
 \begin{proof}
 This amounts to prove that for any symmetric monoidal enriched categories
 $\ecateg C, \ecateg D$,
 the functor 
  $$
  \cmonlax(\categ{Cat}_{\categ E})(\ecateg C, \ecateg D)
  \to \Operade(\End(\ecateg C), \End(\ecateg D))
  $$
  is an isomorphism of categories.
  Let us describe its inverse. Let us consider a morphism of operads
  $f : \End(\ecateg C) \to \End(\ecateg D)$. By restriction to arity 1 elements,
  we obtain a functor $r(f)$ between the underlying enriched categories of
  $\End(\ecateg C)$ and $\End(\ecateg D)$ that are just respectively $\ecateg C$ and
  $\ecateg D$. This functor $r(f)$ has the structure of
  a lax functor given by the images through
  $$
  \End(\ecateg C)(x_1, \ldots, x_n ; x_1\otimes  \cdots \otimes x_n) \to \End(\ecateg D)(f(x_1), \ldots, f(x_n) ; f(x_1\otimes  \cdots \otimes x_n))
  $$
  of the identity of $x_1\otimes  \cdots \otimes x_n$. Besides, for any 2-morphism
  of enriched operads $A : f \to g$,
  it is straightforward to check that the induced enriched natural transformation from $r(f)$ to $r(g)$ is monoidal.
  This construction $f \mapsto r(f)$ actually defines a functor 
  $$
  r: \Operade(\End(\ecateg C), \End(\ecateg D)) \to 
  \cmonlax(\Cate)(\ecateg C, \ecateg D).
  $$
It is clear that $r \circ \End = \id$. The fact that $\End \circ r = \id$
follows from the commutation of the squares
$$
\begin{tikzcd}
 \ecateg C(x_1 \otimes \cdots \otimes x_n,x)
  \ar[r, "f"]  \ar[d, "- \triangleleft \id(x_1 \otimes \cdots \otimes x_n)"']
 &
 \ecateg D(f(x_1 \otimes \cdots \otimes x_n),f(x))
 \ar[d, "- \triangleleft f(\id(x_1 \otimes \cdots \otimes x_n))"]
 \\
 \End(\ecateg C)(x_1, \ldots, x_n;x)
 \ar[r, "f"']
 &  \End(\ecateg D)(f(x_1), \ldots, f(x_n);f(x))
\end{tikzcd}
$$
 for any morphism of enriched operads
 $f: \End(\ecateg C) \to \End(\ecateg D)$.
 \end{proof}

 \begin{proposition}\label{propositionendmonoidal}
 The $\End$ 2-functor has the canonical structure of a
 strong context functor from the monoidal context $\cmonlax(\Cate)$ of symmetric monoidal 
 enriched categories
 and lax enriched functors to the monoidal context of enriched operads
 $\Operade$.
 \end{proposition}
 
 \begin{proof}
 Such a structure is given by the canonical isomorphisms
 $$
 \End(\ecateg C) \otimeshadamard \End(\ecateg D) \simeq \End(\ecateg C \otimes \ecateg D)
 $$
 given by the maps
 $$
\begin{tikzcd}
\left( \End(\ecateg C) \otimeshadamard \End(\ecateg D) \right)
((x_1, y1),\ldots , (x_n, y_n); (x, y))
\ar[d, equal]
\\
\End(\ecateg C)
(x_1,\ldots , x_n; x) \otimes \End(\ecateg D)
(y_1,\ldots , y_n; y)
\ar[d, equal]
\\
\ecateg C (x_1 \otimes \cdots \otimes x_n , x) \otimes \ecateg D (y_1 \otimes \cdots \otimes y_n , y)
\ar[d, equal]
\\
(\ecateg C \otimes \ecateg D)\left((x_1 \otimes \cdots \otimes x_n, y_1 \otimes \cdots \otimes y_n), (x,y)\right)
\ar[d, "\simeq"]
\\
(\ecateg C \otimes \ecateg D)\left((x_1, y_1) \otimes \cdots \otimes (x_n, y_n), (x,y)\right)
\ar[d, equal]
\\
\End(\ecateg C \otimes \ecateg D)((x_1, y1),\ldots , (x_n, y_n); (x, y)).
\end{tikzcd}
 $$
 \end{proof}

\subsection{The Yoneda 2-functor and algebras over an operad}

\begin{definition}
	Let $\categ C$ be a $\mathcal V$ strict 2-category. The
	canonical strict 2-functor
	\begin{align*}
		\categ{C}^{\op} \times \categ{C}  &\to \Cats_{\mathcal V-\mathrm{small}}
		\\
		X, Y & \mapsto \categ{C}(X, Y)
		\end{align*}
	is called the Yoneda 2-functor
	of $\categ C$.
\end{definition}

\begin{lemma}\label{lemmaconvolution}
	Let $\categ C$ be a $\mathcal V$ monoidal context. Then
	$\categ C \times\categ C^\op$
	gets the canonical structure of a $\mathcal V$ monoidal context and the
	Yoneda strict 2-functor
	$$
	X, Y  \mapsto \categ{C}(X, Y)
	$$
	has the canonical structure of a lax context functor.
	\end{lemma}
	
	\begin{proof}
	 The lax structure is given by the naturality of the tensor product, that is the map
	 $$
	 \categ C(X, Y) \times \categ C(X', Y') \to \categ C(X \otimes X', Y \otimes Y') ,
	 $$
	 and by the identity of the monoidal unit of $\categ C$:
	 $$
	 \ast \to  \categ C(\II_{\categ C}, \II_{\categ C}) .
	 $$
	\end{proof}

\begin{definition}
	Let $\catalg{-}{-}$ be the lax context functor that is the composition
	of the strong context functor $\id \times \End$ with Yoneda 2-functor
	 of 
$\Operad_{\categ E}$ (that is a lax context functor):
	$$
	\catalg{-}{-}:\Operade^\op \times \cmonlax(\categ{Cat}_{\categ E})
	\xrightarrow{\id \times \End}
	\Operade^\op \times \Operade
	\to \Cats_{\mathcal V-\mathrm{small}}.
	$$
	For any symmetric monoidal enriched category $\ecateg C$
	and any enriched operad $\operad P$,
	the $\mathcal V$-small category $\catalg{\ecateg C}{\operad P}$
	is called the category of $\operad P$-algebras in $\ecateg C$.
\end{definition}

\begin{definition}
	For any symmetric monoidal enriched category $\ecateg C$ and any
	morphism of enriched operads $f :\operad P \to \operad P'$, we denote
	$f^\ast$ the functor
	$$
f^\ast = \catalg{\ecateg C}{f} : \catalg{\ecateg C}{\operad P'} \to
\catalg{\ecateg C}{\operad P}.
	$$
\end{definition}

\begin{proposition}
	Let $\operad P$ be a $\mathcal U$-small enriched operad and let
	$\ecateg C$ be a symmetric monoidal enriched category.
	Then the $\mathcal V$-small category $\catalg{\ecateg C}{\operad P}$
	is actually a $\mathcal U$-category.
\end{proposition}

\begin{proof}[Sketch of the proof]
	A $\operad P$-algebra $A$ in $\ecateg C$
	is a tuple
	$$
\left((A_o)_{o \in \Ob(\operad P)},
((\gamma(o_1, \ldots, o_n; o)_{o_1, \ldots, o_n, o \in \Ob(\operad P)})_{n \in \mathbb N} \right)
	$$
	that satisfies some conditions, where
	\begin{itemize}
		\itemt $A_o$ is an object of $\ecateg C$ and thus an element of $\mathcal U$;
  		\itemt $\gamma(o_1, \ldots, o_n; o)$ is an element of
		  $$
		  \hom_{\categ E}(\operad P(o_1, \ldots, o_n;o),
		  \ecateg C(A_{o_1}\otimes \cdots \otimes A_{o_n};A_{o}))
		  $$
		  and thus an element of $\mathcal U$.
	\end{itemize}
	
	Hence, such an algebra is a $\mathcal U$-small collection
	of elements of $\mathcal U$. Thus it is an element of $\mathcal U$.
	Hence the set of morphisms of enriched operads from $\operad P$ to $\End(\ecateg C)$
	is $\mathcal U$-large.
	
	Besides, for any two $\operad P$-algebras $A,A'$ in $\ecateg C$,
	the set of morphisms of algebras from $A$ to $A'$ is a subset of
	$$
	\prod_{o \in \Ob(\operad P)}\hom_{\cat(\ecateg C)}(A_o, A'_o) ;
	$$
	which is small. It is thus small.
\end{proof}

\begin{corollary}
	The restriction of the lax context functor $\catalg{-}{-}$
	to $\cmonlax(\categ{Cat}_{\categ E})  \times \Operadesmall^\op$
	factorises through $\Cats_{\mathcal U}$
	$$
	\begin{tikzcd}
		\Operade^\op \times \cmonlax(\categ{Cat}_{\categ E})
		\ar[r, "\catalg{-}{-}"]
		& \Cats_{\mathcal V-\mathrm{small}}
		\\
		\Operadesmall^\op \times \cmonlax(\categ{Cat}_{\categ E})
		\ar[u, hookrightarrow] \ar[r]
		& \Cats_{\mathcal U}.
		\ar[u, hookrightarrow]
	\end{tikzcd}
	$$
	Moreover, since the vertical arrows are strictly fully faithful strict context functors,
	then the horizontal bottom arrow is also a lax context functor.
\end{corollary}

\subsection{Monoidal structures on categories of algebras}

Recall from \cite[Proposition 12]{LeGrignou22}
that the functor between $\mathcal W$-small categories:
$$
\tcatalg{-}{-}_{\lax}: \Operad_\cat^\op \times \categ{Context}_\lax \to \categ{Context}_\lax
$$
is lax monoidal. Hence, for any two categorical operads
$\catoperad B, \catoperad B'$
and any two monoidal contexts $\categ M, \categ M'$,
we get a canonical 2-functor
$$
\tcatalg{\categ M}{\catoperad B}_\lax \times \tcatalg{\categ M'}{\catoperad B'}_\lax \to
\tcatalg{\categ M \times \categ M'}{\catoperad B \times \catoperad B}_\lax ,
$$
that sends a $\catoperad B$-algebra $A$ in $\categ M$ and a $\catoperad B'$-algebra $A'$
in $\categ M'$ to the
$\catoperad B\times \catoperad B'$-algebra in $\categ M \times \categ M'$
given by the composition
$$
\catoperad B \times \catoperad B' \xrightarrow{A \times A'}
\catEnd{(\categ M)} \times \catEnd{(\categ M')} \simeq \catEnd(\categ M \times \categ M') . 
$$
Applying this result to $\categ M = \Operadesmall^\op$ and $\categ M' = \cmonlax(\Cate)$,
we get the following result.

\begin{theorem}\label{thm:main}
 Let $\catoperad B, \catoperad B'$ be two categorical operads.
 Then, the lax context functor
$$
    \catalg{-}{-}: \Operadesmall^\op  \times \cmonlax(\Cate) \to \Cats_{\mathcal{U}}
    $$
induces a lax context functor that is natural with respect to $B,B'$:
 $$
 \tcatalg{\Operadesmall^\op}{\catoperad B}_{\lax} \times
 \tcatalg{\categ{CMon}_{\lax}(\Cate)}{\catoperad B'}_{\lax}
  \to
  \tcatalg{\Cats_{\mathcal U}}{\catoperad B \times \catoperad{B}'}_{\lax} .
 $$
\end{theorem}

\begin{proof}
This lax context functor is the composition
$$
\begin{tikzcd}
\tcatalg{\Operadesmall^\op}{\catoperad B}_{\lax}
\times \tcatalg{\cmonlax(\categ{Cate})}{\catoperad B'}_{\lax}
 \ar[d]
 \\
\tcatalg{\Operadesmall^\op  \times \cmonlax(\Cate)}{\catoperad B \times \catoperad B'}_{\lax}
\ar[d, "\catalg{-}{-}"]
\\
\tcatalg{\Cats_{\mathcal U}}{\catoperad B \times \catoperad B'}_{\lax} .
\end{tikzcd}
$$
The first arrow proceeds from the structure of
a lax monoidal functor on $\tcatalg{-}{-}_{\lax}$ as described above
and the second
arrow is the image through the functor
$\tcatalg{-}{\catoperad B \times \catoperad B'}_{\lax}$
of the lax context functor 
$\catalg{-}{-}$.
\end{proof}

\begin{corollary}
	Let us use the notations of Theorem \ref{thm:main}.
	The 2-functor
	$$
 \tcatalg{\Operadesmall^\op}{\catoperad B}_{\lax} \times
 \tcatalg{\categ{CMon}_{\lax}(\Cate)}{\catoperad B'}_{\lax}
  \to
  \tcatalg{\Cats_{\mathcal U}}{\catoperad B \times \catoperad{B}'}_{\lax} .
 $$
 sends a pair $(f,g)$ of a strict morphism of $\catoperad B$-algebras $f$
 and a strict morphism of $\catoperad B'$-algebras $g$ to a strict morphism
 of $\catoperad B \times \catoperad{B}'$-algebras. The same results holds for strong morphisms.
 One thus gets a commutative diagram
 $$
\begin{tikzcd}
	\tcatalg{\Operadesmall^\op}{\catoperad B}_{\lax} \times
 \tcatalg{\categ{CMon}_{\lax}(\Cate)}{\catoperad B'}_{\lax}
  \ar[r]
  &\tcatalg{\Cats_{\mathcal U}}{\catoperad B \times \catoperad{B}'}_{\lax}
  \\
  \tcatalg{\Operadesmall^\op}{\catoperad B}_{\strong} \times
 \tcatalg{\categ{CMon}_{\lax}(\Cate)}{\catoperad B'}_{\strong}
  \ar[r] \ar[u, hookrightarrow]
  &\tcatalg{\Cats_{\mathcal U}}{\catoperad B \times \catoperad{B}'}_{\strong}
  \ar[u, hookrightarrow]
  \\
  \tcatalg{\Operadesmall^\op}{\catoperad B}_{\strict} \times
 \tcatalg{\categ{CMon}_{\strict}(\Cate)}{\catoperad B'}_{\strict}
  \ar[r] \ar[u, hookrightarrow]
  &\tcatalg{\Cats_{\mathcal U}}{\catoperad B \times \catoperad{B}'}_{\strict}.
  \ar[u, hookrightarrow]
\end{tikzcd}
 $$
\end{corollary}

\begin{proof}
	This is a direct consequence of \cite[Proposition 12]{LeGrignou22}.
\end{proof}

\subsection{The op construction and coalgebras}

\begin{definition}
	For any enriched category $\ecateg C$, let $\ecateg C^\op$ be the enriched
	category with the same set of objects and so that
	$$
	\ecateg C^\op (x,y) = \ecateg C(y,x), \quad x,y \in \Ob(\ecateg C).
	$$
	The units and compositions are given by those of $\ecateg C$.
	This defines two inverse strict context functors 
	$$
	-^\op: \Cate \to \Cate^{co} ; \quad -^\op: \Cate^{co} \to \Cate .
	$$
\end{definition}

\begin{definition}
	Let
	$$
	\catcog{-}{-}: \Operadesmall^{coop} \times \cmonoplax(\Cate)
	 \to \Cats_{\mathcal U}
	$$
	be the lax context functor defined as the composition
	$$
	\begin{tikzcd}
		\Operadesmall^{coop} \times \cmonoplax(\Cate)
		\ar[d, equal]
		\\
		\Operadesmall^{coop} \times \cmonlax(\Cate^{co})^{co}
		\ar[d, equal]
		\\	
		\left(\Operadesmall^{op} \times \cmonlax(\Cate^{co})\right)^{co}
		\ar[d, "\left(\id \times \cmonlax(-^{\op})\right)^{co}"]
		\\
		\left(\Operadesmall^{op} \times \cmonlax(\Cate)\right)^{co}
		\ar[d, "(\catalg{-}{-})^{co}"]
		\\
		\Cats_{\mathcal U}^{co}
		\ar[d, "-^\op"]
		\\ 
		\Cats_{\mathcal U}.
	\end{tikzcd}
	$$
	In particular, for any enriched operad and any symmetric monoidal
	enriched category, we have 
	$$
	\catcog{\ecateg C}{\operad P} = \catalg{\ecateg C^\op}{\operad P}^\op .
	$$
\end{definition}

\begin{definition}
	For any symmetric monoidal enriched category $\ecateg C$ and any
	morphism of operads $f :\operad P \to \operad P'$, we denote
	$f^\ast$ the functor
	$$
f^\ast = \catcog{\ecateg C}{f} : \catcog{\ecateg C}{\operad P'} \to
\catcog{\ecateg C}{\operad P}.
	$$
\end{definition}

\begin{corollary}
	Let $\catoperad B, \catoperad B'$ be two categorical operads.
 Then, the lax context functor
$$
    \catcog{-}{-}: \Operadesmall^{coop}  \times \cmonoplax(\Cate) \to \Cats_{\mathcal{U}}
    $$
induces a lax context functor that is natural with respect to $\catoperad B,\catoperad B'$:
 $$
 \tcatalg{\Operadesmall^{coop}}{\catoperad B}_{\lax} \times
 \tcatalg{\cmonoplax(\Cate)}{\catoperad B'}_{\lax}
  \to
  \tcatalg{\Cats_{\mathcal U}}{\catoperad B \times \catoperad{B}'}_{\lax} .
 $$
 Moreover, this 2-functor sends pairs of strict (resp. strong)
 morphisms of algebras to strict (resp. strong) morphisms of 
 algebras.
\end{corollary}

\begin{proof}
	This follows from the same arguments as those used in the proof of
	Theorem \ref{thm:main}.
\end{proof}


\subsection{Structure of pseudo-commutative algebras on endomorphisms operads}

The 2-functor $\End$ factorises as
$$
\cmonlax(\Cate) \to \cmonlax\cmonlax(\Cate)
\xrightarrow{\cmonlax(\End)} \cmonlax(\Operade)
\to \Operade 
$$
where the first component is the 2-functor described in
Proposition \ref{corollarycmon} and the last component is just the
forgetful 2-functor.

In particular, for any symmetric monoidal enriched category
$\ecateg C$, the enriched operad $\End(\ecateg C)$ is a
pseudo commutative monoid in the monoidal context $\Operade$.
Unfolding the definitions, its product is the morphism
of enriched operads from $\End(\ecateg C) \otimeshadamard \End(\ecateg C)$
to $\End(\ecateg C)$ whose underlying function
sends a pair $(x,y)$ to $x \otimes y$ and whose structural
morphisms in $\categ E$ are the maps
$$
\begin{tikzcd}
	\End(\ecateg C) \otimeshadamard \End(\ecateg C)
\left((x_1, y_1), \ldots, (x_n, y_n);(x_0, y_0)\right)
\ar[d, equal]
\\
\ecateg C(x_1\otimes \cdots \otimes x_n,x_0) \otimes \ecateg C(y_1\otimes \cdots \otimes y_n,y_0)
\ar[d]
\\
\ecateg C((x_1\otimes \cdots \otimes x_n) \otimes (y_1\otimes \cdots \otimes y_n),x_0 \otimes y_0)
\ar[d, "\simeq"]
\\
\ecateg C((x_1\otimes y_1) \otimes \cdots \otimes (x_1\otimes y_1),x_0 \otimes y_0)
\ar[d, equal]
\\
\End(\ecateg C)\left((x_1 \otimes y_1), \ldots, (x_n \otimes y_n);(x_0 \otimes y_0)\right).
\end{tikzcd}
$$
The unit, associator, commutator and unitors of $\End(\ecateg C)$
are just given by those of $\ecateg C$.

Besides, for any lax functor between two symmetric
monoidal enriched categories $f: \ecateg C \to \ecateg D$,
the morphism $\End(f)$ has the canonical structure
of a lax morphism of pseudo-commutative monoids. This lax structure
is actually just given by that of $f$.


\section{Algebras and coalgebras in a symmetric monoidal enriched category}

Let $\ecateg C$ be a symmetric monoidal enriched category.
The goal of this section is to describe monoidal structures of categories
of algebras (or coalgebras) in $\ecateg C$
over a small enriched operad $\operad P$ induced by some coalgebraic structures
on $\operad P$.

From now on, $\Cats$ will denote the $\mathcal U$-categories (that were previously
denoted $\Cats_{\mathcal U}$).

\subsection{Monoidal structure on algebras and coalgebras}

Remember from Corollary \ref{corollarycmonstrongstrict}
that $\ecateg C$ has the canonical structure of an
$\catoperad{E}_{\infty, \cat}$-algebra
(and in particular the structure of an
$\catoperad{A}_{\infty, \cat}$-algebra)
within the monoidal context of symmetric monoidal
enriched categories and strong monoidal functors.

\begin{corollary}\label{corollaryalgebra}
 Let $\catoperad B$ be an operad in sets (thus it is a categorical operad).
 Then, the 2-functor from $\Operadesmall^\op$ to $\Cats$
 that sends an enriched operad $\operad Q$ to
 its category of algebras in $\ecateg C$ yields a functor
 \begin{align*}
	\sk\left(\tcatcog{\Operadesmall}{\catoperad B}_{\strict}\right)^\op =
	\sk\left(\tcatalg{\Operadesmall^\op}{\catoperad B}_{\strict}\right)
	&\to
	\sk\left(\tcatalg{\categ{Cats}}{\catoperad{E}_{\infty,\cat} \times \catoperad{B}}_{\strict}\right) ;	 
	\\
	(\operad P_o)_{o \in \Ob(\catoperad{B})} & \mapsto \left(\catalg{\ecateg C}{\operad P_o}\right)_{o \in \Ob(\catoperad{B})}
 \end{align*}
 that is natural with respect to $\catoperad B$.
\end{corollary}

\begin{proof}
	Such a functor may be actually decomposed as
	$$
\begin{tikzcd}
	\sk\left(\tcatalg{\Operadesmall^\op}{\catoperad B}_{\strict}\right)
	\ar[d, equal]
	\\
	\sk\left(\tcatalg{\Operadesmall^\op}{\catoperad B}_{\strict}\right) \times \ast
	\ar[d, "\id \times \ecateg C"]
	\\
	\sk\left(\tcatalg{\Operadesmall^\op}{\catoperad B}_{\strict}\right) \times
	\sk\left(\cmonstrict(\Cate)\right)
	\ar[d]
	\\
	\sk\left(\tcatalg{\Operadesmall^\op}{\catoperad B}_{\strict}\right) \times
	\sk\left(\tcatalg{\cmonstrong(\Cate)}{\catoperad E_{\infty, \cat}}_{\strict}\right)
	\ar[d, hookrightarrow]
	\\
	\sk\left(\tcatalg{\Operadesmall^\op}{\catoperad B}_{\strict}\right) \times
	\sk\left(\tcatalg{\cmonlax(\Cate)}{\catoperad E_{\infty, \cat}}_{\strict}\right)
	\ar[d]
	\\
	\sk\left(\tcatalg{\categ{Cats}}{\catoperad{E}_{\infty,\cat} \times \catoperad{B}}_{\strict}\right)
\end{tikzcd}
	$$
	The third map is described in Corollary \ref{corollarycmonstrongstrict}
	and the last map is described in Theorem \ref{thm:main}.
\end{proof}

One can work with the opposite symmetric monoidal enriched
category
$\ecateg C^\op$.

\begin{corollary}\label{corollarycoalgebra}
	Let $\catoperad B$ be an operad in sets.
	Then, the 2-functor from $\Operadesmall^{coop}$ to $\categ{Cats}$
	that sends an enriched operad $\operad Q$ to
	its category of coalgebras in $\ecateg C$ yields a functor
	$$
\begin{tikzcd}
	\sk\left(\tcatcog{\Operadesmall}{\catoperad B}_\strict\right)^\op
	\ar[d, equal]
	\\
	{\sk\left(\tcatalg{\Operadesmall^\op}{\catoperad B}_{\strict}\right)}
	\ar[d]
	& {\left(\operad P_o\right)_{o \in \Ob(\catoperad{B})}}
	\ar[d, mapsto]
	\\
	\sk\left(\tcatalg{\Cats}{\catoperad{E}_{\infty,\cat} \times \catoperad{B}}_{\strict}\right)
	\ar[d, "-^\op"]
	& \left(\catalg{\ecateg C^\op}{\operad P_o} \right)_{o \in \Ob(\catoperad{B})}
	\ar[d, mapsto]
	\\
	\sk\left(\tcatalg{\categ{Cats}}{\catoperad{E}_{\infty,\cat} \times \catoperad{B}}_{\strict}\right) ;	 
	& \left(\catcog{\ecateg C}{\operad P} \right)_{o \in \Ob(\catoperad{B})}
\end{tikzcd}
$$
	that is natural with respect to $\catoperad B$.
\end{corollary}

\begin{corollary}
	The statements of Corollary \ref{corollaryalgebra}
	and Corollary \ref{corollarycoalgebra}
	remain valid if $\catoperad B$ is a planar categorical operad and if
	we replace the categorical operad
	$\catoperad E_{\infty, \cat}$ by the planar categorical operad
	$\catoperad A_{\infty, \cat}$.
	\end{corollary}

\subsection{Consequences}

We draw here some consequences of Corollary \ref{corollaryalgebra}
and Corollary \ref{corollarycoalgebra}.

\subsubsection{Hopf operads}

Remember that by a Hopf operad, we mean
a counital coassociative comonoid in the category
$\mathrm{sk}(\Operad_{\categ E})$
of enriched operads for the Hadamard tensor product.

Applying Corollary \ref{corollaryalgebra} to the planar categorical
operad $\catoperad{uAs}$
that encodes unital associative algebras, one gets that the construction 
$$
\operad P \in \sk(\Operadesmall) \mapsto \catalg{\ecateg C}{\operad P}
$$
induces a functor
$$
\categ{HopfOperad}_{\categ E, \smalll}^\op = \sk(\tcatalg{\Operadesmall^\op}{\catoperad{uAs}}_\strict)^\op
\to \sk(\tcatalg{\Operadesmall^\op}{\catoperad{A}_{\infty, \cat}}_\strict)
= \categ{MonoidalCategories}.
$$
More concretely, for any small Hopf operad $\operad Q$, the category $\catalg{\ecateg C}{\operad Q}$
has a structure of a monoidal category whose tensor product $A \otimes A'$
is given by 
$$
\operad Q \to 
 \operad Q \otimeshadamard \operad Q \to \End(\ecateg C) \otimes \End(\ecateg C)
 \simeq \End(\ecateg C \otimes \ecateg C) \to \End(\ecateg C) . 
$$
Moreover, for any morphism of Hopf operads $f:\operad Q \to \operad Q'$,
the induced functor
$$
f^\ast : \catalg{\ecateg C}{\operad Q'} \to \catalg{\ecateg C}{\operad Q}
$$
is strict monoidal.

One has the same results mutatis mutandis
for coalgebras
over a Hopf operad
using Corollary \ref{corollaryalgebra}.

\subsubsection{Cocommutative Hopf operad}

Pursuing the link between Hopf operads and monoidal structures,
for any cocommutative Hopf operad $\operad Q$,
the monoidal category of $\operad Q$-algebras in $\ecateg C$ is
actually a symmetric monoidal category.
For any morphism of cocommutative Hopf operads $f:\operad Q \to \operad Q'$,
the induced functor
$$
f^\ast : \catalg{\ecateg C}{\operad Q} \to \catalg{\ecateg C}{\operad P}
$$
remains a strict symmetric monoidal functor. Again,
one has the same results mutatis mutandis
for coalgebras over a cocommutative Hopf operad.

\subsubsection{Comodules and tensorisation}

Let $\catoperad{LMod}$ the planar coloured operad in sets
that encodes a unital associative algebra and a left module
of this algebra. It has two colours $(a,m)$ so that 
\begin{align*}
    \catoperad{LMod}(a, \ldots, a ; a) &= \ast;
    \\
    \catoperad{LMod}(a, \ldots, a, m ; m) &= \ast;
    \\
    \catoperad{LMod}(\ldots ; -) &= \emptyset \text{ otherwise}.
\end{align*}
A $\catoperad{LMod}$-coalgebra in $\sk(\Operadesmall)$ is a pair
$(\operad Q, \operad P)$ of a small Hopf operad $\operad Q$ and small left $\operad Q$-comodule
$\operad P$. Given such a pair, the pair of categories
$$
(\catalg{\ecateg C}{\operad Q}, \catalg{\ecateg C}{\operad P})
$$
has the structure of a $\catoperad{A}_{\infty, \cat} \times \catoperad{LMod}$-algebra
in the cartesian monoidal context of categories.

We recover in particular, the structure of a
$\catoperad{A}_{\infty, \cat}$-algebra (that is the structure of a monoidal category)
on $\catalg{\ecateg C}{\operad Q}$ described above. Moreover, the whole structure
of a $\catoperad{A}_{\infty, cat} \times \catoperad{LMod}$-algebra on the pair 
$(\catalg{\ecateg C}{\operad Q}, \catalg{\ecateg C}{\operad P})$ induces
a tensorisation of the category $\catalg{\ecateg C}{\operad P}$ over $\operad Q$-algebras.
Such a tensorisation is given by the map
$$
\operad P \to \operad Q \otimeshadamard \operad P \xrightarrow{A \otimeshadamard A'}
\End(\ecateg C) \otimeshadamard \End(\ecateg C) \to \End(\ecateg C)
$$
for any $\operad Q$-algebra $A$ and any $\operad P$-algebra $A'$. One can notice
that the structural natural transformations of the tensorisation are invertible
since they proceed from the associators and unitors of $\ecateg C$.

Similarly, the category $\catcog{\ecateg C}{\operad P}$ is tensored
over $\operad Q$-coalgebras.


\subsection{The Grothendieck construction perspective}
\label{section:grothendieckconstruction}
Since $\ecateg C$ is a pseudo-commutative monoid in
$\cmonlax(\Cate)$,
it should correspond to some weak notion of a lax $\catoperad E_{\infty,\cat}$-morphism
$$
\ast \to \cmonlax(\Cate)
$$
that is less tight than our notion of lax $\catoperad E_{\infty,\cat}$-morphism.
Indeed, we can notice that a lax $\catoperad E_{\infty,\cat}$-morphism
as defined in \cite{LeGrignou22} would not give a pseudo-commutative monoid
but a strictly commutative monoid.
Then, the composite map
$$
\Operadesmall^\op \simeq \Operadesmall^\op \times \ast
\xrightarrow{\id \times \ecateg C}
\Operadesmall^\op \times \cmonlax(\Cate)
\xrightarrow{\catalg{-}{-}} \Cats_{\mathcal U}
$$
should have the same structure of a weak lax $\catoperad E_{\infty,\cat}$-morphism.
In particular the composite 2-functor
$$
\sk(\Operadesmall^\op) \to \Operadesmall^\op \to\Cats_{\mathcal U}
$$
should also have this structure of a weak lax $\catoperad E_{\infty,\cat}$-morphism.
Such a structure is related
by the Grothendieck construction to the notion of a symmetric monoidal
fibration as described in \cite{Shulman08}.
Let us first describe the Grothendieck construction
$\categ{OpAlg}_{\ecateg C}$
of the 2-functor $\catalg{\ecateg C}{-}:\sk(\Operadesmall)^\op \to \Cats_{\mathcal{U}}$.

\begin{definition}
 Let $\categ{OpAlg}_{\ecateg C}$
 be the category whose objects are pairs $(\operad P, \algebra A)$
of a small enriched operad $\operad P$ and a $\operad P$-algebra $\algebra A$
in $\ecateg C$.
A morphism from $(\operad P, \algebra A)$ to $(\operad P', \algebra A')$ is the data of a morphism
of enriched operads $f: \operad P \to \operad P'$ and a morphism of $\operad P$-algebras,
\[
	g : \algebra A \to f^\ast \algebra A' .
\]
In particular, the fiber of the forgetful functor
$\categ{OpAlg}_{\ecateg C} \to \sk(\Operadesmall)$
above an enriched operad $\operad P$ is the category $\catalg{\ecateg C}{\operad P}$.
\end{definition}

\begin{proposition}\label{prop:symmonoidalgc}
	The category $\categ{OpAlg}_{\ecateg C}$
	admits the structure of a symmetric monoidal category so that 
	the forgetful functor towards $\sk(\Operadesmall)$ is a symmetric
	monoidal cartesian fibration (\cite{Shulman08}) in the sense that
   \begin{itemize}
	\itemt it is a cartesian fibration;
	\itemt it is a strict symmetric monoidal functor;
	\itemt for any two cartesian lifting maps in $\categ{OpAlg}_{\ecateg C}$,
	their tensor product
	is also cartesian.
   \end{itemize}
   \end{proposition}
   
   \begin{proof}
First, the forgetful functor towards $\sk(\Operadesmall)$
is a cartesian fibration since, for any morphism of enriched operads
$f: \operad P \to \operad P'$
and any $\operad P'$-algebra $A'$, a cartesian lifting
of $f$ at $A'$ is given by the pair $(f, \id_{f^\ast(A')})$.

	The tensor product $(\operad P_1, \algebra A_1) \otimes (\operad P_2, \algebra A_2)$
   is the pair
   $(\operad P_1 \otimeshadamard \operad P_2, \algebra A_1 \lozenge \algebra A_2)$, where 
   $\algebra A_1 \lozenge \algebra A_2$ is the algebra given by the map
   \[
	   \operad P_1 \otimeshadamard \operad P_2
	   \xrightarrow{\algebra A_1 \otimeshadamard \algebra A_2} \End (\ecateg{C})
	   \otimeshadamard  \End (\ecateg{C})
	   \simeq \End (\ecateg C \otimes \ecateg C) 
	   \xrightarrow{\End(-\otimes_{\ecateg C} -)}
	   \End (\ecateg C).
   \]
   The unit is the pair $(\uCom, \II_{\ecateg C})$, where the monoidal unit
   $\II_{\ecateg C}$ of $\ecateg C$ is equipped with its canonical structure
   of a commutative algebra.
   The associator, commutator and unitors proceed from those of $\ecateg{C}$
   and from those of $\Operadesmall$.
   It is clear that the forgetful functor towards operads is a
   strict symmetric monoidal functor.
   Finally, one can notice that cartesian morphism are pairs $(f, g)$ so that $g$ is an
   isomorphism.
   It is then clear that the tensor product of two cartesian morphisms is cartesian.
   \end{proof}

The monoidal structures on categories of algebras described
above (Corollary \ref{corollaryalgebra})
may be rethinked in terms of the symmetric monoidal cartesian fibration
$\categ{OpAlg}_{\ecateg C} \to \Operadesmall$.
For instance, let us consider a Hopf enriched
operad $\operad Q$. We know from Corollary \ref{corollaryalgebra}
that the category of
$\operad Q$-algebras in $\ecateg C$ has the structure of a monoidal
category. Then, denoting $-\otimes_Q-$ the underlying tensor product
and $\II_Q$ the monoidal unit we have:
\begin{itemize}
	\itemt  For any two $\operad Q$-algebras $A_1, A_2$ in $\ecateg C$,
 we have a cartesian map in $\categ{OpAlg}_{\ecateg C}$
 $$
(\operad Q, A_1 \otimes_{Q} A_2) \to
(\operad Q, A_1 ) \otimes (\operad Q,  A_2)
= (\operad Q \otimeshadamard \operad Q, \algebra A_1 \lozenge \algebra A_2)
 $$
 above the coproduct map $\operad Q \to \operad Q \otimeshadamard \operad Q$.
 This actually defines the tensor
 product
 $A_1 \otimes_{Q} A_2$ up to a unique isomorphism.
  \itemt We have a cartesian map
  $$
  (\operad Q, \II_{Q}) \to (\uCom, \II_{\ecateg C}) = \II_{\categ{OpAlg}_{\ecateg C}}
   $$
  above the counit map $\operad Q \to \uCom$.  This actually defines the monoidal unit
  $\II_Q$ up to a unique isomorphism.
  \itemt For any three $\operad Q$-algebras $A_1, A_2, A_3$ in $\ecateg C$,
  the horizontal maps of the following diagram
  $$
\begin{tikzcd}
	(\operad Q, (A_1 \otimes_Q A_2) \otimes_Q A_3)
	\ar[r]
	& \left((\operad Q, A_1) \otimes (\operad Q, A_2)\right) \otimes (\operad Q, A_3)
	\ar[d]
	\\
	(\operad Q, A_1 \otimes_Q (A_2 \otimes_Q A_3))
	\ar[r]
	& (\operad Q, A_1) \otimes \left( (\operad Q, A_2) \otimes (\operad Q, A_3) \right)
\end{tikzcd}
  $$
  are cartesian since cartesian maps are stable through composition and tensor product.
  Moreover, the right vertical map is an isomorphism. Hence, the two maps
  $$
  \begin{tikzcd}
	  (\operad Q, (A_1 \otimes_Q A_2) \otimes_Q A_3)
	  \ar[rd] \ar[d, dotted]
	  \\
	  (\operad Q, A_1 \otimes_Q (A_2 \otimes_Q A_3))
	  \ar[r]
	  & (\operad Q, A_1) \otimes \left( (\operad Q, A_2) \otimes (\operad Q, A_3) \right)
  \end{tikzcd}
	$$
are cartesian liftings of the same morphism of operads. Thus, we get a unique isomorphism
in $\categ{OpAlg}_{\ecateg C}$ from $(\operad Q, (A_1 \otimes_Q A_2) \otimes_Q A_3)$
to $(\operad Q, A_1 \otimes_Q (A_2 \otimes_Q A_3))$ above the identity of $\operad Q$; this is equivalently
an isomorphism in $\catalg{\ecateg C}{\operad Q}$ which is actually the associator of
the monoidal structure.
  \itemt The unitors may be recovered in a similar way from the fact that the forgetful functor
  $\categ{OpAlg}_{\ecateg C} \to \Operadesmall$ is a symmetric monoidal cartesian fibration.
\end{itemize}

One can make the same work for coalgebras.

\begin{definition}
Let $\categ{OpCog}_{\ecateg C}$
be the category whose objects are pairs $(\operad P, V)$
of a small enriched operad $\operad P$ and a $\operad P$-coalgebra $V$
in $\ecateg C$. A morphism from $(\operad P, V)$
to $(\operad P', V')$ in this category is the data
of a morphism of enriched operads $f :\operad P \to \operad P'$ and
a morphism of $\operad P$-coalgebras $V \to f^\ast (V')$.
\end{definition}

\begin{corollary}
	The category $\categ{OpCog}_{\ecateg C}$ admits the structure
	of a symmetric monoidal category so that the forgetful functor
	towards $\Operadesmall$
	is a symmetric monoidal cartesian fibration (Proposition \ref{prop:symmonoidalgc}
	and \cite{Shulman08}).
\end{corollary}

\begin{proof}
	This follows from the same arguments as those used to prove
	Proposition \ref{prop:symmonoidalgc}.
\end{proof}

Again, the monoidal structures
on categories of coalgebras described
above (Corollary \ref{corollarycoalgebra})
may be rethinked using the symmetric monoidal cartesian fibration
$\categ{OpCog}_{\ecateg C} \to \Operadesmall$.


\subsection{The tensored case and the cotensored context}

Let us suppose that for any object $x \in \ecateg C$ the functor 
\begin{align*}
	\cat(\ecateg C) & \to \categ E
	\\
	y & \mapsto \ecateg C(x,y)
\end{align*}
has a left adjoint $z \mapsto z \boxtimes x$.
In this context, we get natural
maps
$$
(z \otimes_E z') \boxtimes (x \otimes_{\ecateg C} x')
\to (z \boxtimes x) \otimes_{\ecateg C} (z' \boxtimes x').
$$
as the adjoint of the composition
$$
z \otimes z' \to C(x, z \boxtimes x) \otimes
C(x', z' \boxtimes x')
\to C(x \otimes x', (z \boxtimes x) \otimes (z' \boxtimes x')).
$$
Moreover, for any $x \in \Ob(\ecateg C)$, the map
$$
\II_{\categ E} \xrightarrow{\id_x} \ecateg C(x,x)
$$
gives a morphism $\II_{\categ E} \boxtimes x \to x$.

Let $\operad P$ be a small enriched operad. A
$\operad P$-algebra $\algebra A$ in $\ecateg C$ is the data of objects
$$
\algebra A_o \in \ecateg C, \quad o \in \Ob(\operad P)
$$
and morphisms in $\cat(\ecateg C)$
$$
\operad P (o_1, \ldots, o_n;o) \boxtimes
(\algebra A_{o_1} \otimes \cdots, \otimes \algebra A_{o_n})
\to \algebra A_{o}
$$
that satisfies conditions with respect to the structure
of an operad on $\operad P$.
In the context where the forgetful functor
$$
\catalg{\ecateg C}{\operad P} \to \cat(\ecateg C)^{\Ob(\operad P)}
$$
is monadic, with monad $M_{\operad P}$, then the following diagram is commutative
$$
\begin{tikzcd}
	\operad P (o_1, \ldots, o_n;o) \boxtimes
(\algebra A_{o_1} \otimes \cdots, \otimes \algebra A_{o_n})
\ar[r] \ar[d]
& \operad P (o_1, \ldots, o_n;o) \boxtimes
(M_{\operad P}(\algebra A)_{o_1} \otimes \cdots, \otimes M_{\operad P}(\algebra A)_{o_n})
\ar[d]
\\
\algebra A_o
& M_{\operad P}(\algebra A)_o. \ar[l]
\end{tikzcd}
$$

Let $\operad P'$ be another small enriched operad and let $A'$
be a $\operad P'$-algebra in $\ecateg C$. The tensor product 
in the category $\categ{OpAlg}_{\ecateg C}$ of $(\operad P, A)$
with $(\operad P', A')$ gives the pair 
$$
(\operad P, A) \otimes (\operad P', A') =
(\operad P \otimeshadamard\operad P', A \lozenge A')
$$
where $A \lozenge A'$ is the $\operad P \otimeshadamard\operad P'$-algebra
so that:
$$
(A \lozenge A')_{(o,o')} = A_o \otimes A'_{o'}
\quad \forall(o,o') \in \Ob(\operad P) \times \Ob(\operad P')
= \Ob(\operad P \otimeshadamard\operad P').
$$
Moreover, the structure of an algebra is given by the maps
$$
\begin{tikzcd}
	\left(\operad P(o_1, \ldots, o_n;o) \otimes \operad P'(o'_1, \ldots, o'_n;o')\right)
	\boxtimes \left( (A_{o_1} \otimes A'_{o'_1}) \otimes \cdots \otimes (A_{o_n} \otimes A'_{o'_n}) \right)
	\ar[d, "\simeq"]
	\\
	\left(\operad P(o_1, \ldots, o_n;o) \otimes \operad P'(o'_1, \ldots, o'_n;o')\right)
	\boxtimes \left( (A_{o_1} \otimes \cdots \otimes A_{o_n}) \otimes (A'_{o'_1} \otimes \cdots \otimes A'_{o'_n}) \right)
	\ar[d]
	\\
	\left( \operad P(o_1, \ldots, o_n;o) 
	\boxtimes (A_{o_1} \otimes \cdots \otimes A_{o_n})\right)
	\otimes
	\left(\operad P'(o'_1, \ldots, o'_n;o') 
	\boxtimes (A'_{o'_1} \otimes \cdots \otimes A'_{o'_n})\right)
	\ar[d]
	\\
	A_o \otimes A'_{o'}.
\end{tikzcd} 
$$
The structure of a $\uCom$-algebra on the monoidal unit $\II_{\ecateg C}$ of $\ecateg C$
(that yields the monoidal unit of $\categ{OpAlg}_{\ecateg C}$) is given by the map
$$
\uCom (n) \boxtimes \II_{\ecateg C}^{\otimes n} = 
\II_{\categ E} \boxtimes \II_{\ecateg C}^{\otimes n} \to \II_{\ecateg C}^{\otimes n}
\to \II_{\ecateg C}.
$$

One sees, mutatis mutandis the same phenomena for coalgebras over
small enriched operads in the context where 
for any object $y \in \ecateg C$ the functor 
\begin{align*}
	\cat(\ecateg C)^\op & \to \categ E
	\\
	x & \mapsto \ecateg C(x,y)
\end{align*}
has a left adjoint $z \mapsto \langle y, z \rangle$.

\subsection{Mapping coalgebras}

Let us consider two small enriched operads $\operad P_0, \operad P_1$.
We have a bifunctor
$$
\catcog{\ecateg C}{\operad P_0} \times \catcog{\ecateg C}{\operad P_1}
\xrightarrow{- \lozenge -}
\catcog{\ecateg C}{\operad P_0 \otimeshadamard \operad P_1}
$$
which is just the restriction
to $\catcog{\ecateg C}{\operad P_0} \times \catcog{\ecateg C}{\operad P_1}$
of the tensor product of $\categ{OpCog}_{\ecateg C}$.

\begin{proposition}\label{proposition:tensoradjoint}
	Let us suppose that :
	\begin{itemize}
		\itemt the symmetric monoidal category
		$\cat(\ecateg C)$ is a closed symmetric monoidal category
		with internal hom object denoted $[-,-]$;
  \itemt the category $\cat(\ecateg C)$ has small products;
  \itemt the functors
  \begin{align*}
	\catcog{\ecateg C}{\operad P_0}
	&\to \cat(\ecateg C)^{\Ob(\operad P_0)}
	\\
	\catcog{\ecateg C}{\operad P_0 \otimeshadamard \operad P_1}
	&\to \cat(\ecateg C)^{\Ob(\operad P_0 \otimeshadamard \operad P_1)}
  \end{align*}
  are comonadic;
  \itemt the category $\catcog{\ecateg C}{\operad P_0}$ has coreflexive
equalisers.
	\end{itemize}
	Then, for any $\operad P_1$-coalgebra $V_1$
	the functor
	$$
V_0 \in \catcog{\ecateg C}{\operad P_0}
\mapsto
V_0 \lozenge V_1 \in \catcog{\ecateg C}{\operad P_0 \otimeshadamard \operad P_1}
	$$
	has a right adjoint.
The result still holds mutatis mutandis if we swap $\operad P_0$
and $\operad P_1$ in the proposition.
\end{proposition}

\begin{proof}
	This is consequence of the adjoint lifting theorem
	(\cite{Johnstone75}, \cite[Theorem 4]{LeGrignou22}).
	Indeed, let us denote $Q_0$ and $Q_{0,1}$ the comonad related
	to the enriched operads respectively
	$\operad P_0$ and $\operad P_0 \otimeshadamard \operad P_1$,
	let us denote $U_{Q_0} \dashv L^{Q_0}$ and $U_{Q_{0,1}}\dashv L^{Q_{0,1}}$
	the corresponding comonadic adjunctions,
	and let us consider the following commutative square diagram of categories
	$$
	\begin{tikzcd}
		\catcog{\ecateg C}{\operad P_0}
		\arrow[rrrr, "V \mapsto V \lozenge V_1"] \arrow[d, "U_{Q_0}"']
		&&&& \catcog{\ecateg C}{\operad P_0 \otimeshadamard \operad P_1}
		\arrow[d, "U_{Q_{0,1}}"]
		\\
		\cat(\ecateg C)^{\Ob(\operad P_0)}
		\ar[rrrr, "X \mapsto(X_o \otimes V_{1,o'})_{(o,o')\in \Ob(\operad P_0) \times \Ob(\operad P_1)}"']
		&&&& \cat(\ecateg C)^{\Ob(\operad P_0) \times \Ob(\operad P_1)} .
	\end{tikzcd}
	$$
	The functor on the bottom horizontal arrow 
	has a right adjoint $R$ defined as:
	$$
	R(Y)_o = \prod_{o' \in \Ob(\operad P_1)} [V_{1, o'}, Y_{o, o'}] .
	$$
	Then, the functor that sends a $\operad P_0 \otimeshadamard \operad P_1$-coalgebra $W$
	to the coreflexive equaliser of the pair of maps of $\operad P_0$-coalgebras
	$$
	L^{Q_0} R U_{Q_{0,1}} (W) \rightrightarrows L^{Q_0} R Q_{0,1} U_{Q_{0,1}} (W)
	$$
	is right adjoint to the functor on the top horizontal arrow of the diagram.
\end{proof}

Let $f: \operad P \to \operad P'$ be a morphism of small enriched operads.

\begin{proposition}\label{proposition:adjointexistence}
	Let us suppose that
	\begin{itemize}
		\itemt The category $\cat(\ecateg C)$ is complete;
  		\itemt The functors 
    \begin{align*}
		\catcog{\ecateg C}{\operad P} &\to \cat(\ecateg C)^{\Ob(\operad P)}
		\\
		\catcog{\ecateg C}{\operad P'} &\to \cat(\ecateg C)^{\Ob(\operad P')}
	\end{align*}
	are comonadic and the related comonad
	preserve coreflexive equalisers.
	\end{itemize}
	Then the functor $f^\ast: \catcog{\ecateg C}{\operad P} 
	\to \catcog{\ecateg C}{\operad P'}$ has a right adjoint $f^!$. Moreover,
	$f^\ast$ is comonadic if the function $f : \Ob(\operad P) \to \Ob(\operad P')$
	is surjective.
\end{proposition}

\begin{proof}
	The functor
	\begin{align*}
		f^\ast : \cat(\ecateg C)^{\Ob(\operad P')} &\to \cat(\ecateg C)^{\Ob(\operad P)}
		\\
		X &\mapsto \left(\Ob(\operad P) \xrightarrow{f} \Ob(\operad P') \xrightarrow{X} \ecateg C\right)
	\end{align*}
	has a right adjoint $R$ given by
	$$
	R(Y)_{o'} = \prod_{o \in \Ob(\operad P) | f(o) = o'} Y_o
	$$
	Moreover, the category $\catcog{\ecateg C}{\operad P'}$
	has coreflexive equalisers (\cite[Corollary 15]{LeGrignou22}).
	The right adjoint of $f^\ast: \catcog{\ecateg C}{\operad P} 
	\to \catcog{\ecateg C}{\operad P'}$ is then built using the adjoint lifting
	theorem
	(\cite{Johnstone75}, \cite[Theorem 4]{LeGrignou22}).

	Besides, let us assume that the underlying function of $f$ on colours is
	surjective. In that context, the functor $f^\ast: \catcog{\ecateg C}{\operad P} 
	\to \catcog{\ecateg C}{\operad P'}$ is conservative and preserves coreflexive
	equalisers. Since it is left adjoint, then it is comonadic.
\end{proof}

\begin{corollary}
	Let us assume that the symmetric monoidal category $\cat(\ecateg C)$
	is closed and that it is a complete category. Moreover, let us assume
	that for any small enriched operad $\operad P$, the functor
	$$
	\catcog{\ecateg C}{\operad P} \to \cat(\ecateg C)^{\Ob(\operad P)}
	$$
	is comonadic and that the related comonad preserves coreflexive equalisers.
	Then for any morphism of small enriched operads
	$$
	f: \operad P_2 \to \operad P_0 \otimeshadamard \operad P_1
	$$
	any $\operad P_0$-coalgebra $V$ in $\ecateg C$ and any
	$\operad P_1$-coalgebra $W$ in $\ecateg C$, the functors
	\begin{align*}
		V' \in \catcog{\ecateg C}{\operad P_0} &\mapsto f^\ast(V' \lozenge W) \in \catcog{\ecateg C}{\operad P_2}
		\\
		W' \in \catcog{\ecateg C}{\operad P_1} &\mapsto f^\ast(V \lozenge W') \in \catcog{\ecateg C}{\operad P_2}
	\end{align*}
	are both left adjoint functors.
\end{corollary}

\begin{proof}
	This is just a consequence of
	Proposition \ref{proposition:tensoradjoint} and
	Proposition \ref{proposition:adjointexistence}
\end{proof}

\begin{corollary}\label{corollary:adjointexistence}
	Let us suppose that
	\begin{itemize}
		\itemt The category $\cat(\ecateg C)$ is cocomplete;
  		\itemt The functors 
    \begin{align*}
		\catalg{\ecateg C}{\operad P} &\to \cat(\ecateg C)^{\Ob(\operad P)}
		\\
		\catalg{\ecateg C}{\operad P'} &\to \cat(\ecateg C)^{\Ob(\operad P')}
	\end{align*}
	are monadic and the related monads preserve reflexive coequalisers.
	\end{itemize}
	Then the functor $f^\ast: \catalg{\ecateg C}{\operad P} 
	\to \catalg{\ecateg C}{\operad P'}$ has a left adjoint $f_!$. Moreover,
	$f^\ast$ is monadic if the function $f : \Ob(\operad P) \to \Ob(\operad P')$
	is surjective.
\end{corollary}

\begin{proof}
	This is Proposition \ref{proposition:adjointexistence} written for $\ecateg C^\op$ instead of $\ecateg C$.
\end{proof}

Now, let us consider a small Hopf enriched operad $\operad Q$ and a
small left
$\operad Q$-comodule $\operad P$. We know that the category of
$\operad Q$-coalgebras in $\ecateg C$ has the structure of a monoidal category
and that the category of $\operad P$-coalgebras in $\ecateg C$ is tensored
over that of $\operad Q$-coalgebras.
We denote as follows the tensorisation functor:
\begin{align*}
\catcog{\ecateg C}{\operad Q} \times \catcog{\ecateg C}{\operad P}
&\to \catcog{\ecateg C}{\operad P}
\\
(Z, V) & \mapsto Z \boxtimes V . 
\end{align*}

\begin{theorem}
	Let us suppose that the functor
	$$
		V \in \catcog{\ecateg C}{\operad P} \mapsto  Z \boxtimes V
		\in \catcog{\ecateg C}{\operad P}
	$$
	has a right adjoint $W \mapsto \langle W, Z \rangle $
	for any $\operad Q$-coalgebra $Z$. Then, the construction
	$$
	(W, Z) \in \catcog{\ecateg C}{\operad P} \times \catcog{\ecateg C}{\operad Q}^\op
	\mapsto \langle W, Z \rangle \in \catcog{\ecateg C}{\operad P}
	$$
	is natural. Moreover, the category of $\operad P$-coalgebras is cotensored
	over the monoidal category of $\operad Q$-coalgebras
	through this bifunctor.
\end{theorem}

\begin{proof}
	This follows from the fact that the structure of a tensorisation
	on $- \boxtimes -$ is equivalent to the structure of a cotensorisation
	on $\langle -, - \rangle$.
\end{proof}

\begin{theorem}
	Let us suppose that the functor
	$$
		Z \in \catcog{\ecateg C}{\operad Q} \mapsto  Z \boxtimes V
		\in \catcog{\ecateg C}{\operad P}
	$$
	has a right adjoint $W \mapsto \{V, W \}$
	for any $\operad P$-coalgebra $V$. Then, the construction
	$$
	(V, W) \in \catcog{\ecateg C}{\operad P}^\op \times \catcog{\ecateg C}{\operad P}
	\mapsto \{ V, W\}  \in \catcog{\ecateg C}{\operad Q}
	$$
	is natural. Moreover, the category of $\operad P$-coalgebras is enriched
	over the monoidal category of $\operad Q$-coalgebras
	through this bifunctor.
\end{theorem}

\begin{proof}
	This follows from the fact that the structure of a tensorisation
	on $- \boxtimes -$ is equivalent to the structure of an enrichment
	on $\{ -, - \}$.
\end{proof}

\begin{remark}
	Why are we talking about mapping coalgebras and not
	about mapping algebras? After all, coalgebras are just algebras
	in the opposite category. So talking about coalgebras
	instead of algebras is just a matter of conventions.
	Actually, when dealing with usual categories (sets, simplicial
	sets, chain complexes, \ldots), one encounters much more
	mapping coalgebras than mapping algebras (see for instance 
	Section \ref{section:hopfoperadsinaction}). One explanation for this
	is the fact that the operation
	$$
x \in \ecateg C \mapsto x \otimes y
	$$
	is often left adjoint.
\end{remark}



\section{Algebras, coalgebras and convolution}

\subsection{Oplax right modules}

Remember from \cite{LeGrignou22} that $\catoperad{RM}_\oplax$ is the planar categorical operad
that encodes
a pair of a pseudo-monoid $A$ and an oplax right module $M$ of $A$.
For instance, a $\catoperad{RM}_\oplax$-algebra in the monoidal context
of categories is the data of a monoidal category together with a category
cotensored over the opposite of this monoidal category.

It has two colours $(a,m)$ and is generated by operations
$$
\begin{cases}
	\mu \in \catoperad{RM}_\oplax(a, a;a)
	\\
	\iota \in \catoperad{RM}_\oplax (;a)
	\\
	\kappa \in \catoperad{RM}_\oplax(m, a;m)
\end{cases}
$$
together with isomorphisms
\begin{align*}
	\mu \triangleleft (\mu, 1_a) &\simeq \mu \triangleleft (1_a,\mu)
	\\
	\mu \triangleleft (\iota, 1_a) &\simeq 1_a \simeq  \mu \triangleleft (1_a,\iota)
\end{align*}
that satisfy the same relation as in the definition of $\catoperad{A}_{\infty,\cat}$
(the pentagon identity and and triangle identities), and morphisms
\begin{align*}
	\kappa \triangleleft (\kappa, 1_a) &\to \kappa \triangleleft (1_m, \mu)
	\\
	1_m & \to \kappa \triangleleft (1_m, \iota)
\end{align*}
so that the following diagrams are commutative
$$
\begin{tikzcd}
	\left(\kappa \triangleleft (\kappa, 1_m) \right) \triangleleft (\kappa, 1_a, 1_a)
	\ar[r, equal] \ar[d]
	&\kappa \triangleleft \left(\kappa \triangleleft (\kappa, 1_a), 1_a\right)
	\ar[d]
	\\
	\left(\kappa \triangleleft (1_m ,\mu) \right) \triangleleft (\kappa, 1_a, 1_a)
	\ar[d, equal]
	&\kappa \triangleleft \left(1_a, \kappa \triangleleft (1_m, \mu)\right)
	\ar[d, equal]
	\\
	\left(\kappa \triangleleft (\kappa, 1_a) \right) \triangleleft (1_m, 1_a, \mu)
	\ar[d]
	& \left(\kappa \triangleleft (\kappa, 1_a) \right) \triangleleft (1_m, \mu, 1_a)
	\ar[d]
	\\
	\left(\kappa \triangleleft (1_m, \mu) \right) \triangleleft (1_m,  1_a, \mu)
	\ar[d, equal]
	& \left(\kappa \triangleleft (1_m, \mu) \right) \triangleleft (1_m, \mu, 1_a)
	\ar[d, equal]
	\\
	\kappa \triangleleft \left(1_m, \mu \triangleleft (1_a, \mu)\right)
	\ar[r, "\simeq"]
	& \kappa \triangleleft \left(1_m,\mu \triangleleft (\mu, 1_a)\right)
\end{tikzcd}
$$
$$\begin{tikzcd}
	\kappa
	\ar[d, equal]
	\\
	\kappa \triangleleft (1_m, 1_a)
	\ar[r, "\simeq"]\ar[d]
	& \kappa \triangleleft \left(1_m, \mu \triangleleft (\iota, 1_a), 1_m\right)
	\ar[dd, equal]
	\\
	\kappa \triangleleft \left(\kappa \triangleleft (1_m, \iota), 1_a \right)
	\ar[d, equal]
	\\
	\kappa \triangleleft (\kappa, 1_a) \triangleleft (1_m, \iota, 1_a)
	\ar[r]
	& \kappa \triangleleft (1_m, \mu) \triangleleft (1_m, \iota, 1_a)
\end{tikzcd}
$$
$$\begin{tikzcd}
	\kappa
	\ar[d, equal]
	\ar[r, equal]
	& \kappa \triangleleft (1_m, 1_a)
	\ar[d, "\simeq"]
	\\
	1_m \triangleleft \kappa
	\ar[d]
	& \kappa \triangleleft \left(1_m, \mu \triangleleft (1_a, \iota), 1_m\right)
	\ar[dd, equal]
	\\
	\left(\kappa \triangleleft  (1_m, \iota)\right) \triangleleft \kappa
	\ar[d, equal]
	\\
	\kappa \triangleleft (\kappa, 1_a) \triangleleft 
	\ar[r]
	& \kappa \triangleleft (1_m, \mu) \triangleleft (1_m, 1_a, \iota).
\end{tikzcd}
$$
Moreover, an oplax right module $M$ of a pseudo-monoid $A$ 
is called a strong right module if the canonical 2-morphisms
$$
\begin{tikzcd}
	M \otimes (A \otimes A)
	\ar[r, "\simeq"]
	\ar[d]
	& (M \otimes A) \otimes A
	\ar[ddl, Rightarrow]
	\ar[d]
	\\
	M \otimes A
	\ar[d]
	& M \otimes A
	\ar[d]
	\\
	M
	\ar[r, equal]
	& M 
\end{tikzcd}
\quad
\begin{tikzcd}
	M \otimes \II
	\ar[r, "\simeq"]
	\ar[d]
	& M 
	\ar[dl, Rightarrow]
	\ar[d, equal]
	\\
	M \otimes A \ar[r]
	& M 
\end{tikzcd}
$$
are invertible. The categorical operad $\catoperad{RM}_\strong$
that encodes a pseudo-monoid $A$ and a strong right module $M$
is built from 
$\catoperad{RM}_\oplax$ by imposing that the morphisms
\begin{align*}
	\kappa \triangleleft (\kappa, 1_a) &\to \kappa \triangleleft (1_m, \mu)
	\\
	1_m & \to \kappa \triangleleft (1_m, \iota)
\end{align*}
are invertible.

\begin{lemma}
	Let us suppose that the symmetric monoidal
category $\categ E$ is closed and let us denote $[-,-]$ its internal hom.
In that context, the enriched category $\ecateg E$
induced by $\categ E$
has the canonical structure of a symmetric monoidal enriched category.
\end{lemma}

\begin{proof}
	The tensor product of $\ecateg E$ is the morphism of enriched categories
	$\ecateg E \otimes \ecateg E \to \ecateg E$ given by the function on objects
	$$
	(x,y)\in \Ob(\ecateg E \otimes \ecateg E) \mapsto x \otimes_{\categ E} y
	\in \Ob(\ecateg E)
	$$
	and the structural morphisms in $\categ E$
	$$
	(\ecateg E \otimes \ecateg E) ((x,x'), (y,y'))=
	[x,y] \otimes_{\categ E} [x',y'] \to [x \otimes_{\categ E} x', y \otimes_{\categ E} y']
	= \ecateg E (x \otimes_{\categ E} x', y \otimes_{\categ E} y').
	$$
	Moreover, its monoidal unit, its associators,
	unitors and commutators are given by those of the symmetric monoidal category $\categ E$.
	They satisfy the conditions to form a symmetric monoidal enriched category
	since they proceed from the structural maps making $\categ E$ a
	symmetric monoidal category.
\end{proof}

Thus, in the context where the symmetric monoidal category $\categ E$
is closed, and using Proposition \ref{corollarycmon} and Corollary \ref{corollarycmonstrongstrict},
the symmetric monoidal enriched categories $\ecateg E$ and $\ecateg E^\op$ have canonical
structures
of $\catoperad{E}_{\infty,\cat}$-algebras within $\cmonstrong(\Cate)
\subset \cmonlax(\Cate)$. In particular, they are $\catoperad{A}_{\infty,\cat}$-algebras.

	One has an enriched functor
	$$
	\langle -, - \rangle: \ecateg E \otimes \ecateg E^\op \to \ecateg E
	$$
	whose underlying function
	sends a pair of objects $(x,x')$ to $[x', x]$ and whose structural
	morphisms in $\categ E$ are the maps
	$$
	\begin{tikzcd}
		(\ecateg E \otimes \ecateg E^\op)((x,x'), (y, y'))	
		\ar[d, equal]
		\\
		{[x, y] \otimes [y', x']}
		\ar[d]
		\\
		{\left[[x',x], [y',y] \right]}
		\ar[d, equal]
		\\
		{\ecateg E([x',x], [y',y])	.}
	\end{tikzcd}
	$$
	that are adjoint to the maps
	$$
	([x, y] \otimes [y', x'] )\otimes [x',x]
	\simeq  [x, y]\otimes [x',x] \otimes [y', x']
	\to  [y', y].
	$$

	\begin{lemma}\label{lemma:laxstructure}
		The natural maps
	\begin{align*}
		[x, y] \otimes [x' , y'] &\to [x \otimes x' , y \otimes y']
		\\
		\II_{\categ E} &\to  [\II_{\categ E}, \II_{\categ E}].
	\end{align*}
	form two 2-morphisms in $\Cate$
	$$
\begin{tikzcd}
	(\ecateg E \otimes \ecateg E^\op)
	\otimes (\ecateg E \otimes \ecateg E^\op)
	\ar[d, swap, "\simeq"]
	\ar[rr, "{\langle - , - \rangle \otimes \langle -,- \rangle}"]
	&&  \ecateg E \otimes \ecateg E
	\ar[ddll, Rightarrow]
	\ar[dd, "{-\otimes_E-}"]
	\\
	(\ecateg E \otimes \ecateg E)
	\otimes (\ecateg E^\op \otimes \ecateg E^\op)
	\ar[d, swap, "(-\otimes_E -) \otimes (- \otimes_{E^\op}-)"]
	\\
	\ecateg E \otimes \ecateg E^\op
	\ar[rr, swap, "{\langle -,-\rangle}"]
	&& \ecateg E
\end{tikzcd}
\quad
\begin{tikzcd}
	\eII
	\ar[r, equal]
	\ar[d, swap, "1_{E \otimes E^\op}"]
	& \eII
	\ar[dl, Rightarrow]
	\ar[d, "1_E"]
	\\
	\ecateg E \otimes \ecateg E^\op
	\ar[r, swap, "{\langle -,-\rangle}"]
	& \ecateg E.
\end{tikzcd}
	$$
	\end{lemma}

\begin{proof}
For the first map, this amounts to prove that the following diagram
is commutative
$$
\begin{tikzcd}
	{[y_0, y'_0] \otimes [x'_0, x_0]
	\otimes [y_1, y'_1] \otimes [x'_1, x_1]}
	\ar[r] \ar[d, "\simeq"']
	& {\left[[x_0, y_0], [x'_0, y'_0]\right] \otimes 
	\left[[x_1, y_1], [x'_1, y'_1]\right]}
	\ar[dd]
	\\
	{[y_0, y'_0] \otimes [y_1, y'_1] \otimes [x'_0, x_0]
	 \otimes [x'_1, x_1]}
	 \ar[d]
	\\
	{[y_0 \otimes y_1, y'_0 \otimes y'_1] \otimes [x'_0 \otimes x'_1, x_0
	 \otimes x_1]}
	 \ar[d]
	 &
	{\left[[x_0, y_0]\otimes [x_1, y_1], [x'_0, y'_0]
	\otimes [x'_1, y'_1]\right]}
	\ar[d]
	\\
	{\left[[x_0 \otimes x_1, y_0\otimes y_1],
	[x'_0 \otimes x'_1, y'_0\otimes  y'_1]\right]}
	\ar[r]
	& {\left[[x_0, y_0]\otimes [x_1, y_1],
	[x'_0 \otimes x'_1, y'_0\otimes  y'_1]\right]}
\end{tikzcd}
$$
for any objects $x_i, x'_i, y_i, y'_i$ (for $i\in \{0,1\}$).
This follows from the fact that both maps are adjoint to the
same morphism
$$
[y_0, y'_0] \otimes [x'_0, x_0]
	\otimes [y_1, y'_1] \otimes [x'_1, x_1]
	\otimes [x_0, y_0]\otimes [x_1, y_1]
	\to [x'_0 \otimes x'_1, y'_0\otimes  y'_1] .
$$
The naturality of the second map is clear.
\end{proof}

\begin{proposition}\label{prop:cotensorlax}
	The enriched functor
	$
	\langle -, - \rangle: \ecateg E \otimes \ecateg E^\op \to \ecateg E
	$
	has the structure of a lax monoidal enriched functor whose structural
	maps are those of Lemma \ref{lemma:laxstructure}.
\end{proposition}

\begin{proof}
	The results amounts the commutation of the following
	diagrams
	$$
\begin{tikzcd}
	{\left([x, y] \otimes [x',y']\right) \otimes [x'',y'] }
	\ar[r, "\simeq"] \ar[d]
	& {[x, y] \otimes \left([x',y'] \otimes [x'',y']\right) }
	\ar[d]
	\\
	{[x \otimes x', y \otimes y'] \otimes [x'',y']}
	\ar[d]
	& {[x, y] \otimes  [x' \otimes x'',y' \otimes y'] }
	\ar[d]
	\\
	{[(x \otimes x') \otimes x'', (y \otimes y') \otimes y'']}
	\ar[r, "\simeq"']
	& {[x \otimes (x' \otimes x''), y \otimes (y' \otimes y'')]}
\end{tikzcd}
$$
$$
\begin{tikzcd}
	{[x, y] \otimes [x',y']}
	\ar[r, "\simeq"] \ar[d]
	& {[x', y'] \otimes [x,y]}
	\ar[d]
	\\
	{[x \otimes x', y \otimes y'] }
	\ar[r, "\simeq"']
	& {[x' \otimes x, y' \otimes y] }
\end{tikzcd}
\quad
\begin{tikzcd}
	{1_{\categ E} \otimes [x,y]}
	\ar[r, "\simeq"] \ar[d, "\simeq"']
	& {[1_{\categ E}, 1_{\categ E}] \otimes [x,y]}
	\ar[d]
	\\
	{[x,y]} 
	\ar[r, "\simeq"']
	& {[1_{\categ E} \otimes x, 1_{\categ E} \otimes y]}
\end{tikzcd}
	$$
	which is straightforward to prove.
\end{proof}

\begin{proposition}\label{proposition:cotensorlaxtwo}
	The maps
	\begin{align*}
		[x \otimes x', y] &\simeq [x' \otimes x, y] \simeq [x' , [x, y]]
		\\
		[1_E , x] &\simeq x
	\end{align*}
	form two 2-morphisms in the 2-category of symmetric monoidal enriched categories
	and lax morphisms
	$$
\begin{tikzcd}
\ecateg E \otimes \ecateg E^\op \otimes \ecateg E^\op
\ar[r] \ar[d]
& \ecateg E \otimes \ecateg E^\op
\ar[d] \ar[ld, Rightarrow]
\\
\ecateg E \otimes \ecateg E^\op
\ar[r]
& \ecateg E
\end{tikzcd}
\quad
\begin{tikzcd}
\ecateg E \otimes \eII
\ar[r, "\id \otimes 1_E"] \ar[d, "\simeq"']
& \ecateg E \otimes \ecateg E^\op
\ar[d]
\\
\ecateg E
\ar[r, equal]
& \ecateg E .
\end{tikzcd}
	$$
\end{proposition}

\begin{proof}
	Let us consider the first set of maps $[x \otimes x', y] \simeq [x , [x', y]]$.
	On the one hand, it form a 2-morphism in $\Cate$ since the following diagram
	is commutative
	$$
\begin{tikzcd}
	{[y,y'] \otimes [x'_0, x_0] \otimes [x'_1, x_1]}
	\ar[r] \ar[d]
	& {\left[ [x_1, [x_0, y]], [x'_1, [x'_0, y']] \right]}
	\ar[d]
	\\
	{\left[ [x_0 \otimes x_1, y], [x'_0 \otimes x'_1, y'] \right]}
	\ar[r]
	& {\left[ [x_0 \otimes x_1, y], [x'_1, [x'_0, y']] \right]}
\end{tikzcd}
	$$
	for any objects $x_0,x'_0, x_1, x'_1, y, y'$ of $\categ E$.
	Indeed, both maps from the top left to the bottom right of the
	square are adjoint to the same map
	$$
	{[y,y'] \otimes [x'_0, x_0] \otimes [x'_1, x_1]} \otimes 
	[x_0 \otimes x_1, y] \otimes x'_1 \otimes x'_0
	\to y' .
	$$
	On the other hand, this set of maps form a 2-morphism in $\cmonlax(\Cate)$
	since the following square diagrams are commutative
	$$
	\begin{tikzcd}
		{[x_0 \otimes x_1, y] \otimes [x'_0 \otimes x'_1, y']}
		\ar[r] \ar[d]
		& {[(x_0 \otimes x_1) \otimes (x'_0 \otimes x'_1), y \otimes y']}
		\ar[r]
		& {[(x_0 \otimes x'_0) \otimes (x_1 \otimes x'_1), y \otimes y']}
		\ar[d]
		\\
		{[x_1 ,[x_0, y]] \otimes [x'_1 ,[x'_0, y']]}
		\ar[r]
		& {[x_1 \otimes x'_1 ,[x_0, y]\otimes [x'_0, y']] }
		\ar[r]
		& {[x_1 \otimes x'_1 ,[x_0 \otimes x'_0, y\otimes y']] }
	\end{tikzcd}
	$$
	$$
	\begin{tikzcd}
		1_E
		\ar[d, equal] \ar[r]
		& {[1_E \otimes 1_E, 1_E]}
		\ar[d]
		\\
		1_E
		\ar[r]
		& {[1_E ,[1_E, 1_E]]}
	\end{tikzcd}
	$$
	for any objects $x_0,x'_0, x_1, x'_1, y, y'$ of $\categ E$.

	The fact that the second set of maps $[1_E , x] \simeq x$
	forms a 2-morphism in the 2-category of symmetric monoidal enriched categories
	and lax morphisms follows from similar arguments.
\end{proof}

\begin{theorem}\label{thm:closedcase}
	If the symmetric monoidal
	category $\categ E$ is closed, then
	the cotensorisation of $\categ E$
	over itself canonically makes $\ecateg E$ a strong right module over the pseudo-monoid
	$\ecateg E^\op$ in the monoidal
	context $\cmonlax(\Cate)$.
\end{theorem}

\begin{proof}
	Proposition \ref{prop:cotensorlax} provides us with a lax morphism
	$$
\langle - , - \rangle : \ecateg E\otimes \ecateg E^\op \to \ecateg E
	$$
	and Proposition \ref{proposition:cotensorlaxtwo}
	provides us with invertible 2-morphisms
	\begin{align*}
		\langle - , - \rangle \circ \left(\langle - , - \rangle \otimes \id \right)
		& \to \langle - , - \rangle \circ \left(\id \otimes (- \otimes_{\ecateg E^\op} -) \right)
		\\
		\id &\to \langle - , - \rangle \circ \left(\id \otimes 1_E \right)
	\end{align*}
	The fact that they define the structure of a strong
	right module on $\ecateg E$ over the pseudo-monoid $\ecateg E^\op$
	 within the monoidal context $\cmonlax(\Cate)$ follows from the
	commutation of the following diagrams
	$$
\begin{tikzcd}
	{[x'' ,[x',[x, y]]]}
	\ar[r] \ar[d]
	& {[x'' ,[x \otimes x', y]]}
	\ar[dd]
	\\
	{[x'  \otimes x'',[x, y]]]}
	\ar[d]
	\\
	{[(x \otimes x') \otimes x'', y]}
	\ar[r]
	&{[x \otimes (x' \otimes x''), y]}
\end{tikzcd}
	$$
	$$
	\begin{tikzcd}
	{[x,y]}
	\ar[r] \ar[rd, "\simeq"']
	& {[1_E, [x,y]]}
	\ar[d]
	\\
	&{[x \otimes 1_E, y]}
	\end{tikzcd}
	\quad
	\begin{tikzcd}
		{[x,y]}
		\ar[r] \ar[rd, "\simeq"']
		& {[x,[1_E, y]]}
		\ar[d]
		\\
		&{[1_E \otimes x, y]}
		\end{tikzcd}
	$$
	for any objects $x,x',x'',y$ of $\categ E$.
\end{proof}

\begin{corollary}
	Let us suppose that the symmetric monoidal
	category $\categ E$ is closed and let us denote $[-,-]$ its internal hom.
	In that context the cotensorisation of $\categ E$
over itself canonically makes $\categ E$ a strong right module over the pseudo-monoid
$\categ E^\op$ in the monoidal
context $\cmonlax(\Cats)$.
\end{corollary}


\subsection{Operads and cotensorisation}

Let $\ecateg D$ be a symmetric monoidal enriched category and let us equip it with
its canonical structure of an $\catoperad{A}_{\infty, \cat}$-algebra within
the monoidal context of symmetric monoidal enriched categories and lax functors
$\cmonlax(\Cate)$
(Proposition \ref{corollarycmon}). Then
let $\ecateg C$ be a symmetric monoidal enriched category equipped with the 
structure of a $\ecateg D^\op$ oplax right module within the monoidal context $\cmonlax(\Cate)$.

\begin{corollary}\label{corollary:convcase}
	Let $\operad Q$ be a Hopf operad and let $\operad P$ be a right $\operad Q$-comodule.
	Then the category $\catalg{\ecateg C}{\operad P}$ is cotensored over the monoidal
	category $\catcog{\categ D}{\operad P} = \catalg{\categ D^\op}{\operad P}^\op$.
	Moreover, for any morphism of $\catoperad{RM}$-coalgebras
	$$
	(f,g):(\operad Q, \operad P) \to (\operad Q', \operad P')
	$$
	the functor $g^\ast: \catalg{\categ D}{\operad P'} \to \catalg{\categ D}{\operad P}$
	commutes strictly with the cotensorisation structures. 
\end{corollary}

\begin{proof}
	Such a construction is given by the following 2-functor
	$$
	\begin{tikzcd}
		\sk\left(\tcatcog{\Operadesmall}{\catoperad{RM}}_\strict\right)^\op
		\ar[d, "\simeq"]
		\\
		\sk\left(\tcatcog{\Operadesmall}{\catoperad{RM}}_\strict\right)^\op \times \ast
		\ar[d, "{\id \times (\ecateg C, \ecateg D^\op)}"]
		\\
		\sk\left(\tcatcog{\Operadesmall}{\catoperad{RM}}_\strict\right)^\op \times
		\sk\left(\tcatalg{\cmonlax(\Cate)}{\catoperad{RM}_{\oplax}}_\strict\right)
		\ar[d]
		\\
		\sk\left(\tcatalg{\Cats}{ \catoperad{RM} \times \catoperad{RM}_{\oplax}}_\strict\right)
		\ar[d]
		\\
		\sk\left(\tcatalg{\Cats}{\catoperad{RM}_{\oplax}}_\strict\right)
	\end{tikzcd}
	$$
	where the third map proceeds from Theorem \ref{thm:main} and the last map
	proceeds from the canonical morphism of categorical operads
	$$
	\catoperad{RM}_{\oplax} \to \catoperad{RM} \times \catoperad{RM}_{\oplax}.
	$$
\end{proof}

\subsection{The Grothendieck construction perspective}

\begin{proposition}
	The category $\categ{OpAlg}_{\ecateg C}$ is canonically cotensored over the monoidal
	category $\categ{OpAlg}_{\ecateg D^\op}^\op$.
\end{proposition}

\begin{proof}
	This is equivalent to the fact that $\categ{OpAlg}_{\ecateg C}$ is an oplax
	right module over the pseudo-monoid $\categ{OpAlg}_{\ecateg D^\op}$ within
	the monoidal context of categories.
	For any two elements $X = (\operad P, A) \in \categ{OpAlg}_{\ecateg C}$
	and $Y =(\operad P', V) \in \categ{OpAlg}_{\ecateg D^\op}$, let us denote
	$\langle X, Y \rangle$ the element
	$$
	(\operad P \otimeshadamard \operad P', \langle A,V\rangle) \in \categ{OpAlg}_{\ecateg C} 
	$$
	where $\langle A,V\rangle$ is the $\operad P \otimeshadamard \operad P'$-algebra in $\ecateg C$
	given by the composition
	$$
	\operad P \otimeshadamard \operad P' \xrightarrow{A \otimeshadamard V}
	\End(\ecateg C) \otimeshadamard \End(\ecateg D^\op)
	\simeq \End(\ecateg C \otimes\ecateg D^\op)
	\xrightarrow{\End(\langle -, - \rangle)} \End(\ecateg C) .
	$$
	This construction is natural and thus defines a functor
	$$
	\categ{OpAlg}_{\ecateg C} \times \categ{OpAlg}_{\ecateg D^\op}
	\to \categ{OpAlg}_{\ecateg C}.
	$$
	Moreover, one has natural maps
	\begin{align*}
		(\operad P, A) &\to \langle (\operad P, A), \II_{\categ{OpAlg}_{\ecateg D^\op}} \rangle
		\\	
		\langle \langle (\operad P, A), (\operad Q, V) \rangle, (\operad Q', V'') \rangle
		&\to \langle (\operad P, A), (\operad Q, V) \otimes_{\categ{OpAlg}_{\ecateg D^\op}} (\operad Q', V'') \rangle
	\end{align*}
given by the following 2-morphisms in the 2-category of enriched operads
	$$
\begin{tikzcd}
	\operad P
	\ar[r, "A"] \ar[d, "\simeq"']
	& \End(\ecateg C)
	\ar[r, "\simeq"]
	\ar[d]
	& \End(\ecateg C \otimes \eII)
	\ar[r, "\simeq"] \ar[ddd]
	& \End(\ecateg C)
	\ar[ddd, equal]
	\ar[dddl, Rightarrow]
	\\
	\operad P \otimeshadamard \uCom
	\ar[r]
	& \End(\ecateg C) \otimeshadamard \uCom
	\ar[d]
	\\
	& \End(\ecateg C) \otimeshadamard \End(\ecateg D^\op)
	\ar[d]
	\\
	& \End(\ecateg C \otimes \ecateg D^\op)
	\ar[r, equal]
	& \End(\ecateg C \otimes \ecateg D^\op)
	\ar[r, "{\End(\langle -,- \rangle)}"']
	& \End(\ecateg C)
\end{tikzcd}
	$$
	$$
	\begin{tikzcd}
		\operad P \otimeshadamard \operad Q \otimeshadamard \operad Q'
		\ar[d, "{A \otimeshadamard V' \otimeshadamard V''}"]
		\\
		\End(\ecateg C) \otimeshadamard \End(\ecateg D^\op) \otimeshadamard \End(\ecateg D^\op)
		\ar[d, "\simeq"]
		\\
		\End(\ecateg C \otimes \ecateg D^\op \otimes \ecateg D^\op)
		\ar[d, "\End(\id \otimes (- \otimes_D -))"'] \ar[r, "{\End(\langle -, - \rangle \otimes \id)}"]
		& \End(\ecateg C \otimes \ecateg D^\op)
		\ar[d, "{\End(\langle -, - \rangle)}"]
		\ar[ld, Rightarrow]
		\\
		\End(\ecateg C \otimes \ecateg D^\op)
		\ar[r, "{\End(\langle -, - \rangle)}"']
		& \End(\ecateg C)
	\end{tikzcd}
	$$
	that proceeds from the structure of an oplax right module of $\ecateg C$
	over $\ecateg D^\op$.

	This functor and these two natural
	transformations define a cotensorisation of
	$\categ{OpAlg}_{\ecateg C}$
	over $\categ{OpAlg}_{\ecateg D^\op}^\op$ since these natural
	transformations proceeds from the structure of an oplax right module of $\ecateg C$
	over $\ecateg D^\op$.
\end{proof}

Besides, using the canonical morphism of categorical operads
$$
\catoperad{RM}_{\oplax} \to \catoperad{A}_{\infty, \cat}
$$
and the monoidal structure on $\sk(\Operadesmall)$,
we get that the pair of categories $(\sk(\Operadesmall), \sk(\Operadesmall))$
has the structure of a $\catoperad{RM}_{\oplax}$-algebra in the monoidal
context of categories.

\begin{corollary}
	The forgetful functors 
	\begin{align*}
		\categ{OpAlg}_{\ecateg C} &\to \sk(\Operadesmall)
		\\
		\categ{OpAlg}_{\ecateg D^\op}&\to \sk(\Operadesmall)
	\end{align*}
	are cartesian fibrations that form
	a strict morphism of $\catoperad{RM}_{\oplax}$-algebras.
	Moreover, the bifunctor
	$$
\langle - , - \rangle : \categ{OpAlg}_{\ecateg C} \times \categ{OpAlg}_{\ecateg D^\op} \to
\categ{OpAlg}_{\ecateg C}
	$$
	sends pairs of cartesian maps to cartesian maps.
\end{corollary}

\begin{proof}
	We already know from Proposition \ref{prop:symmonoidalgc}
	that they are
	cartesian fibrations. The fact that they form
	a strict morphism of $\catoperad{RM}_{\oplax}$-algebras
	follows from the definitions.

	Finally the bifunctor $\langle - , - \rangle$
	sends pairs of cartesian maps to cartesian maps
	because for any two morphisms of enriched operads
	\begin{align*}
		f:& \operad P \to \operad P'
		\\
		g:& \operad Q \to \operad Q',
	\end{align*}
	any $\operad P'$-algebra $A$ in $\ecateg C$ and any $\operad Q'$-coalgebra
	$V$ in $\ecateg D$, then 
	$$
	\langle f^\ast(A), g^\ast (V) \rangle  = (f \otimeshadamard g)^\ast(\langle A, V \rangle).
	$$

\end{proof}

As in Section \ref{section:grothendieckconstruction},
the monoidal structures relating
categories of coalgebras in $\ecateg D$ and categories of algebras
in $\ecateg C$ described
above (Corollary \ref{corollary:convcase})
may be rethinked in terms of the fibrations
	\begin{align*}
		\categ{OpAlg}_{\ecateg C} &\to \sk(\Operadesmall);
		\\
		\categ{OpAlg}_{\ecateg D^\op}&\to \sk(\Operadesmall).
	\end{align*}
For instance, let us consider a small Hopf enriched
operad $\operad Q$ and an small enriched operad $\operad P$
that has the structure of right $\operad Q$-comodule.
We know
that the category of
$\operad P$-algebras in $\ecateg C$ is cotensored over the monoidal
category of $\operad Q$-coalgebras. Then, denoting $\langle -, -\rangle_{Q}$
the underlying cotensorisation bifunctor, we have:
\begin{itemize}
	\itemt  For any $\operad P$-algebra $A$ in $\ecateg C$
	and any $\operad Q$-coalgebra $V$ in $\ecateg D$
 	we have a cartesian map in $\categ{OpAlg}_{\ecateg C}$
 $$
(\operad P, \langle A,V \rangle_Q) \to
\langle (\operad P, A ) , (\operad Q,  V) \rangle
 $$
 above the decomposition map $\operad P \to \operad P \otimeshadamard \operad Q$.
 This actually defines the cotensor
 $\langle A,V \rangle_Q$ up to a unique isomorphism.
  \itemt For any $\operad P$-algebra $A$ in $\ecateg C$
  and any two $\operad Q$-coalgebras $V_1, V_2$ in $\ecateg D$,
  the horizontal maps of the following diagram
  $$
\begin{tikzcd}
	(\operad P, \langle \langle A, V_1 \rangle_Q, V_2 \rangle_Q)
	\ar[r]
	& \langle \langle (\operad P, A), (\operad Q, V_1)\rangle , (\operad Q, V_2) \rangle
	\ar[d]
	\\
	(\operad P, \langle A, V_1 \otimes_Q V_2 \rangle_Q)
	\ar[r]
	& \langle  (\operad P, A), (\operad Q, V_1) \otimes (\operad Q, V_2) \rangle
\end{tikzcd}
  $$
  are cartesian since cartesian maps are stable through composition, cotensorisation
  and tensor product.
  Thus, we get a canonical map in $\categ{OpAlg}_{\ecateg C}$
  $$
  (\operad P, \langle \langle A, V_1 \rangle_Q, V_2 \rangle_Q)
   \to (\operad P, \langle A, V_1 \otimes_Q V_2 \rangle_Q)
	$$
	above the identity of $\operad P$. This gives us the structural map
	$$
	\langle \langle A, V_1 \rangle_Q, V_2 \rangle_Q
	\to \langle A, V_1 \otimes_Q V_2 \rangle_Q.
	$$
  \itemt The structural map
 $$
 A \to \langle  A, \II_{\catcog{\ecateg D}{\operad Q}} \rangle_Q
 $$
  may be recovered in a similar way.
\end{itemize}

\subsection{Mapping coalgebras}

Let us consider two small enriched operads $\operad P, \operad Q$.
We have a bifunctor
$$
\catalg{\ecateg C}{\operad P} \times \catcog{\ecateg D}{\operad Q}^\op
\xrightarrow{\langle -, - \rangle}
\catalg{\ecateg C}{\operad P \otimeshadamard \operad Q}
$$
which is just the restriction
to $\catalg{\ecateg C}{\operad P} \times \catcog{\ecateg D}{\operad Q}$
of the cotensorisation bifunctor
$$
\langle -, -\rangle : \categ{OpAlg}_{\ecateg C} \times
\categ{OpAlg}_{\ecateg D^\op} \to \categ{OpAlg}_{\ecateg C}.
$$

\begin{corollary}\label{proposition:cotensoradjointalg}
	Let us suppose that :
	\begin{itemize}
		\itemt For any object $y \in \Ob(\ecateg C)$,
	the functor 
	\begin{align*}
		\langle y, - \rangle : \cat(\ecateg D)^\op &\to \cat(\ecateg C)
		\\
		x & \mapsto \langle y, x \rangle
	\end{align*}
	has a left adjoint $z \mapsto \lVert z, y\rVert$;
  \itemt the category $\cat(\ecateg D)$ has small products;
  \itemt the functor
  $$
	\catalg{\ecateg C}{\operad P \otimeshadamard \operad Q}
	\to \cat(\ecateg C)^{\Ob(\operad P) \times \Ob(\operad Q)}
	$$
	is monadic and the functor
	$$
	\catcog{\ecateg D}{\operad Q}
	\to \cat(\ecateg D)^{\Ob(\operad Q)}
	$$
  is comonadic;
  \itemt the category $\catcog{\ecateg C}{\operad Q}$ has coreflexive
equalisers.
	\end{itemize}
	Then, for any $\operad P$-algebra $A$
	the functor
	$$
V \in \catcog{\ecateg C}{\operad Q}^\op
\mapsto
\langle A, V\rangle \in \catalg{\ecateg C}{\operad P \otimeshadamard \operad Q}
	$$
	has a left adjoint.
\end{corollary}

\begin{proof}
	This follows from the same arguments as those used to prove
	Proposition \ref{proposition:tensoradjoint}.
\end{proof}

\begin{corollary}\label{proposition:cotensoradjointalg2}
	Let us suppose that :
	\begin{itemize}
		\itemt For any object $x \in \Ob(\ecateg D)$,
	the functor 
	\begin{align*}
		\langle -, x \rangle : \cat(\ecateg C) &\to \cat(\ecateg C)
		\\
		y & \mapsto \langle y, x \rangle
	\end{align*}
	has a left adjoint $z \mapsto x \boxtimes z$;
  \itemt the category $\cat(\ecateg C)$ has small coproducts;
  \itemt the functors
  \begin{align*}
	\catalg{\ecateg C}{\operad P \otimeshadamard \operad Q}
	&\to \cat(\ecateg C)^{\Ob(\operad P) \times \Ob(\operad Q)}
	\\
	\catalg{\ecateg C}{\operad P}
	&\to \cat(\ecateg C)^{\Ob(\operad P)}
  \end{align*}
	are monadic;
  \itemt the category $\catalg{\ecateg C}{\operad P}$ has reflexive
coequalisers.
	\end{itemize}
	Then, for any $\operad Q$-coalgebra $V$ in $\ecateg D$
	the functor
	$$
A \in \catalg{\ecateg C}{\operad P}
\mapsto
\langle A, V\rangle \in \catalg{\ecateg C}{\operad P \otimeshadamard \operad Q}
	$$
	has a left adjoint.
\end{corollary}

\begin{proof}
	Again, this follows from the same arguments as those used to prove
	Proposition \ref{proposition:tensoradjoint}.
\end{proof}

\begin{corollary}
	Let us assume that the category $\cat(\ecateg C)$
	is cocomplete, that the category $\cat(\ecateg D)$ is complete
	and that both functors
	\begin{align*}
		\langle y, - \rangle &: \cat(\ecateg D)^\op \to \cat(\ecateg C)
		\\
		\langle -, x \rangle &: \cat(\ecateg C) \to \cat(\ecateg C)
	\end{align*}
	have a left adjoint for any $x \in \Ob(\ecateg D),y
	\in \Ob(\ecateg C)$. Moreover, let us assume
	that for any small enriched operad $\operad P$, the functor
	$$
	\catalg{\ecateg C}{\operad P} \to \cat(\ecateg C)^{\Ob(\operad P)}
	$$
	is monadic and that the related monad preserves reflexive coequalisers
	and that the functor
	$$
	\catcog{\ecateg D}{\operad P} \to \cat(\ecateg C)^{\Ob(\operad P)}
	$$
	is comonadic and that the related comonad preserves coreflexive equalisers.
	Then for any morphism of small enriched operads
	$$
	f: \operad P_2 \to \operad P_0 \otimeshadamard \operad P_1
	$$
	any $\operad P_1$-coalgebra $V$ in $\ecateg D$ and any
	$\operad P_0$-algebra $A$ in $\ecateg C$, the functors
	\begin{align*}
		A' \in \catalg{\ecateg C}{\operad P_0} &\mapsto f^\ast(\langle A', V \rangle) \in \catalg{\ecateg C}{\operad P_2}
		\\
		V' \in \catcog{\ecateg D}{\operad P_1}^\op &\mapsto f^\ast(\langle A, V' \rangle) \in \catalg{\ecateg C}{\operad P_2}
	\end{align*}
	both have a left adjoint functor.
\end{corollary}

\begin{proof}
	This is just a consequence of
	Corollary \ref{proposition:cotensoradjointalg},
	Corollary \ref{proposition:cotensoradjointalg2} and
	Corollary \ref{corollary:adjointexistence}.
\end{proof}

Now, let us assume that $\operad Q$ is a small Hopf enriched operad
and that $\operad P$ is a small right
$\operad Q$-comodule. We know that the category of
$\operad Q$-coalgebras in $\ecateg D$ has the structure of a monoidal category
and that the category of $\operad P$-algebras in $\ecateg C$ is cotensored
over that of $\operad Q$-coalgebras through the bifunctor
$$
\langle -,- \rangle: \catalg{\ecateg C}{\operad P} \times \catcog{\ecateg D}{\operad Q}^\op
\to \catalg{\ecateg C}{\operad P}.
$$

\begin{theorem}
	Let us suppose that the functor
	$$
		V \in \catcog{\categ D}{\operad Q}^\op \mapsto  \langle A, V \rangle
		\in \catalg{\ecateg C}{\operad P}
	$$
	has a left adjoint $A' \mapsto \{ A', A \} $
	for any $\operad P$-algebra $A$ in $\ecateg C$. Then, the construction
	$$
	(A', A) \in \catalg{\ecateg C}{\operad P}^\op \times \catalg{\ecateg C}{\operad P}
	\mapsto \{ A',A \} \in \catcog{\ecateg D}{\operad Q}
	$$
	is natural. Moreover, the category of $\operad P$-algebras
	in $\ecateg C$ is enriched
	over the monoidal category of $\operad Q$-coalgebras
	in $\ecateg D$
	through this bifunctor.
\end{theorem}

\begin{proof}
	This follows from the fact that the structure of an 
	enrichment on $\{-,-\}$ is equivalent to the structure of a cotensorisation
	on $\langle -, - \rangle$.
\end{proof}

\begin{theorem}
	Let us suppose that the functor
	$$
		A \in \catalg{\ecateg C}{\operad P} \mapsto  \langle A, V \rangle
		\in \catalg{\ecateg C}{\operad P}
	$$
	has a left adjoint $A \mapsto V \boxtimes A$
	for any $\operad Q$-coalgebra $V$ in $\ecateg D$. Then, the construction
	$$
	(V, A) \in \catcog{\categ D}{\operad Q} \times \catalg{\ecateg C}{\operad P}
	\mapsto  V \boxtimes A \in \catalg{\ecateg C}{\operad P}
	$$
	is natural. Moreover, the category of $\operad P$-algebras
	in $\ecateg C$ is tensored
	over the monoidal category of $\operad Q$-coalgebras
	in $\ecateg D$
	through this bifunctor.
\end{theorem}

\begin{proof}
	This follows from the fact that the structure of a tensorisation
	on $- \boxtimes -$ is equivalent to the structure of a cotensorisation
	on $\langle -, - \rangle$.
\end{proof}

\subsection{The case of the ground category}

Let us suppose that the ground symmetric monoidal category $\categ E$
is closed with induced enriched category denoted $\ecateg E$.
Let us also suppose that $\ecateg C = \ecateg D = \ecateg E$
and that the structure of a $\catoperad{RM}_\oplax$-algebra on the pair
$(\ecateg E^\op, \ecateg E)$ is the one induced by the internal hom
$[-,-]$ of $\categ E$ (see Theorem \ref{thm:closedcase}).

Let $A$ be a $\operad P$-algebra and let $V$ be a $\operad Q$-coalgebra.
In that context, $\langle A, V \rangle$ is the $\operad P \otimeshadamard \operad Q$-algebra
whose underlying objects are 
$$
[V_{o'}, A_o], \quad (o,o') \in \Ob(\operad P \otimeshadamard \operad Q)
= \Ob(\operad P) \times \Ob(\operad Q)
$$
and whose structural maps are adjoint to the maps
$$
\begin{tikzcd}
	{\operad P(o_1, \ldots, o_n;o) \otimes \operad Q(o'_1, \ldots, o'_n;o')
	\otimes \left( [ V_{o'_1}, A_{o_1}]  \otimes \cdots 
	\otimes [V_{o'_n}, A_{o_n} ] \right) \otimes V_{o'}}
	\ar[d]
	\\
	{\operad P(o_1, \ldots, o_n;o) \otimes 
	[ V_{o'_1} \otimes \cdots \otimes V_{o'_n}, 
	A_{o_1} \otimes \cdots \otimes A_{o_n}]  
	\otimes \operad Q(o'_1, \ldots, o'_n;o') \otimes  V_{o'}}
	\ar[d]
	\\
	{\operad P(o_1, \ldots, o_n;o) \otimes 
	[ V_{o'_1} \otimes \cdots \otimes V_{o'_n}, 
	A_{o_1} \otimes \cdots \otimes A_{o_n}]  
	\otimes V_{o'_1} \otimes \cdots \otimes V_{o'_n}}
	\ar[d]
	\\
	{\operad P(o_1, \ldots, o_n;o) \otimes 
	A_{o_1} \otimes \cdots \otimes A_{o_n} }
	\ar[d]
	\\
	A_o.
\end{tikzcd}
$$


\section{Hopf operads in action}
\label{section:hopfoperadsinaction}

In the final section of this part, we suppose that the ground symmetric monoidal category
$(\categ E, \otimes , \II)$ 
is a closed symmetric monoidal category and that is
complete and cocomplete. We denote $\ecateg E$ the induced $\categ E$-enriched
category. Moreover, we suppose that for any small enriched operad 
$\operad P$ the functor
$$
\catalg{\ecateg E}{\operad P} \to \categ E^{\Ob(\operad P)}
$$
is monadic and that the related monad preserves reflexive coequalisers. We also
suppose that the functor
$$
\catcog{\ecateg E}{\operad P} \to \categ E^{\Ob(\operad P)}
$$
is comonadic and that the related comonad preserves coreflexive equalisers.
In that context, for any morphism of small enriched operads
$$
f:\operad P_2 \to \operad P_0 \otimeshadamard \operad P_1
$$
the functors
\begin{align*}
	A \in \catalg{\ecateg E}{\operad P_0} &\mapsto f^\ast\langle A, V_1 \rangle \in \catalg{\ecateg E}{\operad P_2}
	\\
	V \in \catcog{\ecateg E}{\operad P_1}^\op &\mapsto f^\ast\langle A_0, V \rangle \in \catalg{\ecateg E}{\operad P_2}
\end{align*}
have a left adjoint for all $V_1 \in \catcog{\ecateg E}{\operad P_1}$ and all $A_0 \in \catalg{\ecateg E}{\operad P_0}$
and the functors
\begin{align*}
	V \in \catcog{\ecateg E}{\operad P_0} &\mapsto f^\ast(V \lozenge V_1) \in \catcog{\ecateg E}{\operad P_2}
	\\
	V \in \catcog{\ecateg E}{\operad P_1}^\op &\mapsto f^\ast(V_0 \lozenge V) \in \catcog{\ecateg E}{\operad P_2}
\end{align*}
have a right adjoint for all $V_1 \in \catcog{\ecateg E}{\operad P_1}$ and all $V_0 \in \catcog{\ecateg E}{\operad P_0}$.

The goal of this section is to draw some more concrete consequences
of these results for algebras and coalgebras
in $\categ E$.

\subsection{Mapping cocommutative coalgebras and linear duality}

Any small enriched operad $\operad P$ is a right and a left comodule
over the cocommutative Hopf operad
$\uCom$.

\begin{corollary}
	For any $\uCom$-algebra $A$, one has a linear duality functor
	\begin{align*}
		\catcog{\ecateg E}{\operad P}^\op &\to \catalg{\ecateg E}{\operad P}
		\\
		V & \mapsto \langle A, V \rangle
	\end{align*}
	that has a left adjoint.
\end{corollary}

\begin{corollary}
The category of $\uCom$-coalgebras is closed symmetric monoidal.
Moreover, the category of $\operad P$-algebras in $\categ E$
and the category of $\operad P$-coalgebras in $\categ E$
are both tensored-cotensored and enriched over $\uCom$-coalgebras.
\end{corollary}

\begin{corollary}
	The monoidal category of $\uCom$-coalgebras is cartesian.
   \end{corollary}
   
   \begin{proof}
	Let $\coalgebra V$, $\coalgebra W$ and $\coalgebra Z$ be three $\uCom$-coalgebras.
	Given two morphisms $f: \coalgebra Z \to \coalgebra V$ and $g: \coalgebra Z \to \coalgebra W$,
	we can build a morphism $f +g$ as follows
	\[
		\coalgebra Z \to \coalgebra Z \otimes \coalgebra Z \xrightarrow{f \otimes g} \coalgebra V \otimes \coalgebra W.
	\]
	Moreover, given a morphism $t=  \coalgebra Z \to \coalgebra V \otimes \coalgebra W$, we can build two morphisms
   \begin{align*}
		d_1 (t): \coalgebra Z \xrightarrow{t} \coalgebra V \otimes \coalgebra W \to \coalgebra V \otimes \II \simeq \coalgebra V ;
	   \\
	   d_2 (t): \coalgebra Z \xrightarrow{t} \coalgebra V \otimes \coalgebra W \to \II \otimes \coalgebra W \simeq \coalgebra W.
   \end{align*}
   It is clear that $d_1(f+g)=f$ and $d_2(f+g)=g$. Moreover, the fact that the following diagram commutes
   \[
   \begin{tikzcd}
		\coalgebra Z
	   \arrow[r,"t"] \arrow[dd]
	   & \coalgebra V \otimes \coalgebra W
	   \arrow[d] \arrow[ddrr,equal]
	   \\
	   & \coalgebra V \otimes \coalgebra V \otimes \coalgebra W \otimes \coalgebra W
	   \arrow[d,"\simeq"]
	   \\
	   \coalgebra Z \otimes \coalgebra Z
	   \arrow[r,"t \otimes t"']
	   & \coalgebra V \otimes \coalgebra W \otimes \coalgebra V \otimes \coalgebra W
	   \arrow[r]
	   & \coalgebra V \otimes \II \otimes \II \otimes \coalgebra W
	   \arrow[r,"\simeq"']
	   & \coalgebra V \otimes \coalgebra W
   \end{tikzcd}
   \]
   implies that $t= d_1(t) + d_2(t)$. Thus, we have a bijection
   \[
	   \hom_{\catcog{\ecateg E}{\uCom}}(\coalgebra Z, \coalgebra V) \times \hom_{\catcog{\ecateg E}{\uCom}}(\coalgebra Z, \coalgebra W)
	   \simeq \hom_{\catcog{\ecateg E}{\uCom}}(\coalgebra Z, \coalgebra V \otimes \coalgebra W).
   \]
   So, $\coalgebra V \otimes \coalgebra W$ is the product of $\coalgebra V$ with $\coalgebra W$.
   Finally, $\II$ is the final object and the associators, unitors and commutators 
are those of the cartesian monoidal structure.
   \end{proof}

\subsection{Planar operads are $\uAs$-comodules}

For any planar enriched operad $\operad P$, the induced enriched operad
$\operad P_{\categ S}$
has the natural structure of a $\uAs$-comodule as follows
\[
	\operad P(\underline c^{\sigma^{-1}};c) \otimes \{\sigma\} \to \operad P(\underline c^{\sigma^{-1}};c) \otimes \{\sigma\} \otimes \{\sigma\}
	\hookrightarrow \operad P_{\categ S}(\underline c;c) \otimes \uAs(n).
\]
Thus the functor $-_{\categ S}: \sk(\mathsf{pl}\Operade) \to \sk(\Operade)$ factorises through a functor
\[
	\sk(\mathsf{pl}\Operade) \to \mathsf{Comodules}_{\sk(\Operade)}\left(\uAs\right),
\]
followed by the forgetful functor.

\begin{theorem}\label{theoremplanaroperad}
Let us suppose that coproducts in $\categ E$ are disjoint and pullback-stable. Then, this functor
from planar enriched operads to $\uAs$-comodules in enriched operads
is an equivalence of categories.
\end{theorem}

\begin{remark}
Recall that the fact that pullbacks are disjoint means that any map $X \to X \sqcup Y$ is a monomorphism
and that any square of the form
\[
\begin{tikzcd}
 	\emptyset
	\arrow[r] \arrow[d]
	& X
	\arrow[d]
	\\
	Y \arrow[r]
	& X \sqcup Y
\end{tikzcd}
\]
is a pullback.
\end{remark}

\begin{proof}
It is clear that this functor is faithful. Let us show that it is full.
Consider two planar enriched operads $\operad P, \operad P'$ and
a morphism $f$ of comodules from $\operad P_{\categ S}$ to $\operad P'_{\categ S}$.
Let us denote $\phi$ the underlying function of $f$ from $\Ob(\operad P)$ to $\Ob(\operad P')$.
Then, $f$ is given by maps of the form
\[
	f(\underline c;c): \coprod_{\sigma \in \categ S_n} \operad P(\underline c^{\sigma^{-1}};c) \otimes \{\sigma\}
	\to
	\coprod_{\sigma \in \categ S_n} \operad P'(\phi(\underline c)^{\sigma^{-1}};\phi(c)) \otimes \{\sigma\}
\]
Since this is a morphism of comodules, then the following diagram commutes
\[
\begin{tikzcd}
 	\operad P(\underline c^{\sigma^{-1}};c) \otimes \{\sigma\} 
	\arrow[r] \arrow[d]
	& \coprod_{\mu}\operad P'(\phi(\underline c)^{\mu^{-1}};\phi(c)) \otimes \{\mu\}
	\arrow[d]
	\\
	\operad P(\underline c^{\sigma^{-1}};c) \otimes \{\sigma\}  \otimes \{\sigma\}  
	 \ar[d]
	& \coprod_{\mu}\operad P'(\phi(\underline c)^{\mu^{-1}};\phi(c)) \otimes \{\mu\} \otimes \{\mu\}
	\arrow[d, hookrightarrow]
	\\
	\left(\coprod_{\mu}\operad P'(\phi(\underline c)^{\mu^{-1}};\phi(c)) \otimes \{\mu\}\right)  \otimes \{\sigma\}  \arrow[r]
	& \coprod_{\mu, \nu}\operad P'(\phi(\underline c)^{\mu^{-1}};\phi(c)) \otimes \{\mu\} \otimes \{\nu\} .
\end{tikzcd}
\]
Thus the top horizontal map factors through the pullback of the span
\[
	\left(\coprod_{\mu}\operad P'(\phi(\underline c)^{\mu^{-1}};\phi(c)) \otimes \{\mu\}\right)  \otimes \{\sigma\} 
	\rightarrow \coprod_{\mu, \nu}\operad P'(\phi(\underline c)^{\mu^{-1}};\phi(c)) \otimes \{\mu\} \otimes \{\nu\}
	\leftarrow \coprod_{\mu}\operad P'(\phi(\underline c)^{\mu^{-1}};\phi(c)) \otimes \{\mu\}
\]
 which is by assumption $\operad P'(\phi(\underline c)^{\sigma^{-1}};\phi(c)) \otimes \{\sigma\}$.
 Subsequently we have maps 
 \[
 	g(\underline c;c; \sigma): \operad P(\underline c^{\sigma^{-1}};c) \to 
	\operad P'(\phi(\underline c)^{\sigma^{-1}};\phi(c))
\]
so that $f(\underline c;c)=\coprod_\sigma g(\underline c;c; \sigma)$.
The fact that $f$ is a natural with respect to the actions of symmetric groups
implies that $g(\underline c;c; \sigma) = g(\underline c^\mu;c;\sigma\circ \mu)$ for any two permutations $\sigma,\mu$. Thus, if we denote
$g(\underline c;c)\coloneqq g(\underline c;c; 1)$,
we have for any permutation $\sigma$
\[
	g(\underline c;c; \sigma) =g(\underline c^{\sigma^{-1}};c).
\]
A straightforward check shows that the maps $g(\underline c;c): \operad P(\underline c;c) \to \operad P'(\phi(\underline c);\phi(c))$
form a morphism of planar operads whose image in the category of
$\uAs$-comodules is $f$. So,
the functor from planar enriched operads
to $\uAs$-comodules is full.
 
 Let us show now that this functor is essentially surjective.
 Let $\operad P$ be a $\uAs$-comodule in enriched operads.
 Then, for any $n$-tuple of colours $\underline c$
 and any colour $c$, let $\operad Q (\underline c;c;\sigma)$ be the following pullback
 (which is actually a coreflexive equaliser)
 in $\categ E$
 \[
\begin{tikzcd}
 	\operad Q (\underline c;c;\sigma) 
	\arrow[rr]
	\arrow[d]
	&& \operad P (\underline c;c) \otimes \{\sigma\}
	\arrow[d]
	\\
	\operad P (\underline c;c)
	\arrow[r]
	& \operad P (\underline c;c)\otimes \uAs (n)
	\arrow[r,equal]
	& \coprod_{\mu} \operad P (\underline c;c)\otimes\{\mu\} .
\end{tikzcd}
 \]
 and let us denote $\operad Q (\underline c;c)\coloneqq \operad Q (\underline c;c;1)$. 
 Since coproducts are pullback stable, we have
 \[
 	\operad P(\underline c ; c) = \coprod_{\sigma} \operad Q (\underline c;c;\sigma) .
 \]
 Then, one can check that the operadic composition $\gamma_i$ in $\operad P$ decomposes 
 into maps
 \[
 	\operad Q(\underline c;c;\sigma) \otimes \operad Q(\underline c';\underline c[i];\mu) \to
	\operad Q(\underline c \triangleleft_i \underline c'; c; \sigma \triangleleft_i \mu) .
 \]
 where $\sigma \triangleleft_i \mu = \gamma_i^{\uAs}(\{\sigma\}\otimes \{\mu\})$.
 In particular, the collection of objects $\operad Q (\underline c; c)$ inherits from this operadic composition
 the structure of a planar operad. Our goal now is to exhibit an isomorphism of enriched operads 
 between $\operad P$ and $\operad Q_{\categ S}$. The commutativity of the following diagram
\[
\begin{tikzcd}
 	\operad P (\underline c;c) 
	\arrow[r] \arrow[d]
	& \coprod_{\mu} \operad P (\underline c;c)\otimes\{\mu\}
	\arrow[d]
	& \operad P (\underline c;c) \otimes \{\sigma\}
	\arrow[l] \arrow[d]
	\\
	\operad P (\underline c^\nu;c) 
	\arrow[r]
	& \coprod_{\mu} \operad P (\underline c^\nu;c)\otimes\{\mu\circ \nu\}
	& \operad P (\underline c^\nu;c) \otimes \{\sigma\circ \nu\}
	\arrow[l]
\end{tikzcd}
\]
whose vertical arrows are all isomorphisms
induces an isomorphism $\operad Q(\underline c;c;\sigma) \simeq \operad Q(\underline c^\nu;c;\sigma\circ \nu)$. for
any two permutations $\sigma,\nu$.
So in particular $\operad Q(\underline c;c;\sigma) \simeq \operad Q(\underline c^{\sigma^{-1}};c)$.
Thus, we get an isomorphism of objects indexed by tuples of colours
\[
	\operad P(\underline c;c) \simeq \coprod_{\sigma} \operad Q(\underline c^{\sigma^{-1}};c)
	= \coprod_{\sigma} \operad Q(\underline c^{\sigma^{-1}};c) \otimes \{\sigma\}
	=  \operad Q_{\categ S}(\underline c;c) .
\]
Actually, this map relating $\operad P$ to $\operad Q_{\categ S}$ commutes with the action of
$\categ S$ and is thus an isomorphism
of coloured symmetric sequences.
Finally, one can check that it commutes with the units and the operadic composition. Hence
this is
an isomorphism of operads. So the functor from planar enriched operads
to $\uAs$-comodules is essentially surjective.
\end{proof}

\begin{corollary}
For any planar enriched operad $\operad P$,
the category of $\operad P$-algebras in $\ecateg E$ and the
category of $\operad P$-coalgebras in $\ecateg E$
are both tensored-cotensored and enriched over $\uAs$-coalgebras.
\end{corollary}

\begin{corollary}
	For any planar enriched operad $\operad P$
	and any $\uAs$-algebra $A$, one has a linear duality functor
	\begin{align*}
		\catcog{\ecateg E}{\operad P}^\op &\to \catalg{\ecateg E}{\operad P}
		\\
		V & \mapsto \langle A, V \rangle
	\end{align*}
	that has a left adjoint.
\end{corollary}

\subsection{Categories and their modules as operads and their algebras}

The category $\sk(\Cate)$ of enriched-categories
is a coreflexive localisation of the category of enriched operads
$\sk(\Operade)$. Indeed, one can see an enriched category
$\ecateg A$ as an operad by declaring
\[
	\ecateg A(\underline c;c)=\emptyset
\]
whenever $|\underline c|\neq 1$. The right adjoint to this inclusion functor
just sends an enriched operad $\operad P$ to its underlying enriched category.

\begin{proposition}
 The category of enriched categories is canonically isomorphic to the category
 of Hopf $\III$-comodules within enriched operads.
\end{proposition}

\begin{proof}
Straightforward.
\end{proof}

Under this point of view, the inclusion $\sk(\Cate) 
\hookrightarrow \sk(\Operade)$
is the functor that forgets the structure of a
$\III$-comodule and its right adjoint
is the cofree $\III$-comodule functor.

\begin{proposition}
For any small enriched category $\operad A$,
we have a canonical isomorphism of categories
\[
	\catalg{\ecateg E}{\operad A} = \catcog{\ecateg E}{\operad A^\op} .
\]
that is natural is the sense that for any morphism of 
small enriched categories
$f :\operad A \to \operad B$, the following diagram of categories
commutes
\[
\begin{tikzcd}
 	\catalg{\ecateg E}{\operad B}
	\arrow[r,"f^\ast"] \arrow[d,equal] 
	&\catalg{\ecateg E}{\operad A} 
	\arrow[d, equal]
	\\
	\catcog{\ecateg E}{\operad B^\op}
	 \arrow[r,"f^\ast"']
	  &\catcog{\ecateg E}{\operad A^\op}.
\end{tikzcd}
\]
\end{proposition}

\begin{proof}
Straightforward.
\end{proof}

Subsequently, for any morphism of small enriched categories
$f :\operad A \to \operad B$
the functor $f^\ast : \catalg{\ecateg E}{\operad B} \to
\catalg{\ecateg E}{\operad A}$
has a left adjoint and a right adjoint.
The left adjoint is $f_!$ while
the right adjoint is the composite functor
 \[
	\catalg{\ecateg E}{\operad A}  \simeq \catcog{\ecateg E}{\operad A^\op}
	\xrightarrow{(f^{\op})^!} \catcog{\ecateg E}{\operad B^\op}
	\to \catalg{\ecateg E}{\operad B} . 
 \]
In particular, given a small enriched
category $\operad A$, the forgetful functor
$\forget^{\operad A}: \catalg{\ecateg E}{\operad A} \to \categ E^{\Ob(\operad A)}$
has a right adjoint
$$
X \mapsto \left( \prod_{c'} [\operad A(c,c'), X_{c'}] \right)_{c \in \Ob(\operad P)}
$$
and a left adjoint
$$
X \mapsto \left( \coprod_{c'} \operad A(c',c) \otimes X_{c'} \right)_{c \in \Ob(\operad P)}.
$$

\subsection{The cartesian case and beyond}

In this section, we suppose that
$(\categ{E}, \times, *)$ is a cartesian closed monoidal category.

\begin{lemma}
In the cartesian context, the two forgetful functors
\[
	\catcog{\ecateg E}{\uCom} \to \catcog{\ecateg E}{\uAs}
	\to \categ E
\]
are isomorphisms of categories.
The inverse functors with source
$\categ E$ send an object $X$
to itself equipped with the diagonal coalgebra structure.
\end{lemma}

\begin{proof}
 Straightforward.
\end{proof}

In that case :
\begin{itemize}
 \itemt Any enriched operad $\operad{P}$ has a unique structure of a cocommutative Hopf operad. Thus,
 the category of $\operad{P}$-algebras and the category of $\operad{P}$-coalgebras inherit
 the structures of closed symmetric monoidal categories.
 \itemt For any enriched operad $\operad{P}$, the category of $\operad{P}$-algebras and the category of $\operad{P}$-coalgebras are
 tensored-cotensored-enriched over the category $\categ{E}$.
\end{itemize}
Moreover, for any enriched operad $\operad Q$, a $\operad Q$-comodule (left or right, for the Hadamard tensor product) is just
an enriched operad $\operad P$ together with a morphism $f: \operad P \to \operad Q$.

\begin{proposition}
Since $\categ E$ is cartesian, for any enriched operad $\operad P$
there exists an enriched category $\operad P_{\mathrm{cart}}$,
and a canonical isomorphism of categories
\[
	\catcog{\ecateg E}{\operad P}
	\simeq \catcog{\ecateg E}{\operad P_{\mathrm{cart}}^\op}
	= \catalg{\ecateg E}{\operad P_{\mathrm{cart}}} 
\]
that commutes with the forgetful functor towards
$\categ E^{\Ob(\operad P)}$.
\end{proposition}

\begin{proof}
A $\operad P$-coalgebra is the data of objects $\coalgebra V(c)\in \categ E$
for any $c \in \Ob(\operad P)$ and of maps
\[
	a(\underline c ;c ;i) : \operad P(\underline c;c) \times \coalgebra V(c)   \to \coalgebra V(\underline c[i]) 
\]
for $|\underline c|\geq 1$, so that the following diagram commutes
\[
\begin{tikzcd}
 	\coalgebra V(c) \times \operad P (\underline c ; c) \times \operad P (\underline c' ; {\underline c[i]})
	\arrow[r,"a(\underline c;c;i)\times \id{}"] \arrow[d,"\id{}\times \gamma_i"']
	&{\coalgebra V(\underline c[i]) \times \operad P (\underline c' ; {\underline c[i]})}
	\arrow[d,"a(\underline c';{\underline c[i]};j)"]
	\\
	\coalgebra V(c) \times \operad P (\underline c \triangleleft_i \underline c'; c) 
	\arrow[r,"a(\underline c \triangleleft_i \underline c';c;i+j-1)"']
	& \coalgebra V(\underline c'{[j]})
\end{tikzcd}
\]
for any $1\leq i \leq |\underline c|$ and any $1\leq j \leq |\underline c'|$,
the following diagram commutes
\[
\begin{tikzcd}
 	\coalgebra V(c) \times \operad P (\underline c ; c) \times \operad P (\underline c' ; {\underline c[k]})
	\arrow[r,"a(\underline c;c;i)\times \id{}"] \arrow[d,"\id{}\times \gamma_k"']
	&{\coalgebra V(\underline c[i]) \times \operad P (\underline c' ; {\underline c[j]})}
	\arrow[d,"\mathrm{proj}_1"]
	\\
	\coalgebra V(c) \times \operad P (\underline c \triangleleft_k \underline c'; c) 
	\arrow[r,"a(\underline c \triangleleft_k \underline c';c;i)"']
	& \coalgebra V(\underline c{[i]})
\end{tikzcd}
\]
for $1 \leq i<k\leq |\underline c|$ ($\underline c'$ may be empty),
the following diagram commutes
\[
\begin{tikzcd}
 	\coalgebra V(c) \times \operad P (\underline c ; c) \times \operad P (\underline c' ; {\underline c[k]})
	\arrow[r,"a(\underline c;c;i)\times \id{}"] \arrow[d,"\id{}\times \gamma_k"']
	&{\coalgebra V(\underline c[i]) \times \operad P (\underline c' ; {\underline c[k]})}
	\arrow[d,"\mathrm{proj}_1"]
	\\
	\coalgebra V(c) \times \operad P (\underline c \triangleleft_k \underline c'; c) 
	\arrow[r,"a(\underline c \triangleleft_k \underline c';c;i+ |\underline c'|-1)"']
	& \coalgebra V(\underline c{[i]})
\end{tikzcd}
\]
for $1 \leq k<i \leq |\underline c|$ ($\underline c'$ may be empty),
and the following diagrams commute
\[
\begin{tikzcd}
 	\coalgebra V(c)
	\arrow[r,"\id{}\times \eta"] \arrow[rd,equal]
	&{\coalgebra V(c) \times \operad P (c ; c)}
	\arrow[d,"{a(c;c;1)}"]
	\\
	& \coalgebra V(c)
\end{tikzcd}
\quad
\begin{tikzcd}
 	\coalgebra V(c) \times \operad P (\underline c ; c)
	\arrow[r,"{a(\underline c;c;i)}"] \arrow[d,"{\id{}\times \sigma^\ast}"']
	&{\coalgebra V(\underline c[i])}
	\\
	\coalgebra V(c) \times \operad P (\underline c^\sigma ; c) .
	\arrow[ru,"{a(\underline c^\sigma;c;\sigma^{-1}(i))}"']
\end{tikzcd}
\]
for any permutation $\sigma \in \categ S_n$.
For any two colours $c',c$, let $X(c,c')$
be the quotient in $\categ E$ of
$$
\coprod_{n>0; 1 \leq i \leq n}  \coprod_{\underline{c}, |\underline{c}|=n, \underline{c}[i]=c'}
\operad P(\underline{c};c)
$$
by the following relations
\[
\begin{tikzcd}
	\coprod_{\underline c[i]=c'}\coprod_\sigma \operad P(\underline c;c)
	\arrow[rrd, shift left] \arrow[rrd, shift right]
	\\
	&& \coprod_{\underline c[i]=c'}
	\operad P(\underline{c};c)
	\\
	\coprod_{\underline c[i]=c'}
	\coprod_{ k \neq i ,\underline c'}
	\operad P (\underline c ; c) \times \operad P (\underline c' ; {\underline c[k]}).
	\arrow[rru, shift left] \arrow[rru, shift right]
\end{tikzcd}
\]
Then, let $\operad P_{\mathrm{cart}}$ be the quotient in enriched categories
of the enriched category 
$\treeoperad X$ freely generated by $X$
by the following relations
\[
\begin{tikzcd}
 	\operad P (\underline c ; c) \times \operad P (\underline c' ; {\underline c[i]})
	\arrow[r] \arrow[d]
	& X (c;\underline c{[i]} ) \times X ({\underline c[i]}; \underline c'{[j]} )
	\arrow[r,equal]
	& X ({\underline c[i]}; \underline c'{[j]} ) \times X (c;\underline c{[i]} ) 
	\arrow[d]
	\\
	\operad P(\underline c \triangleleft_i \underline c';c)
	\arrow[r]
	& X(c; \underline c'{[j]})
	\arrow[r]
	& \treeoperad X(c;\underline c'{[j]})
\end{tikzcd}
\]
\[
\begin{tikzcd}
 	\ast
	\arrow[r] \arrow[rrd]
	& \operad P(c;c)
	\arrow[r]
	& X(c;c)
	\arrow[d]
	\\
	&& \treeoperad X(c;c) .
\end{tikzcd}
\]
The small enriched category $\operad P_{\mathrm{cart}}$ having such a presentation,
one gets a canonical isomorphism
$$
\catcog{\ecateg E}{\operad P}
	\simeq \catalg{\ecateg E}{\operad P_{\mathrm{cart}}} .
$$
\end{proof}

\begin{corollary}
The forgetful functor from $\operad P$-coalgebras to $\categ E^{\Ob(\operad P)}$
has a right adjoint 
$$
X \mapsto \left( \prod_{c'} [\operad P_{\mathrm{cart}}(c,c'), X(c')] \right)_{c \in \Ob(\operad P)}
$$
and a left adjoint
$$
X \mapsto \left( \coprod_{c'} \operad P_{\mathrm{cart}}(c',c) \times X(c') \right)_{c \in \Ob(\operad P)}.
$$
\end{corollary}

For any symmetric monoidal category $\left(\categ F, \otimes , \II_{\categ F}\right)$, any oplax symmetric monoidal functor from the cartesian monoidal category $\categ E$ to $\categ F$ factorises essentially uniquely through $\uCom$-coalgebras in $\categ F$. Hence, in some sense $\uCom$-coalgebras in $\categ F$ form the best cartesian approximation of $\categ F$.

\subsection{Changing the ground category}

Consider two bilinear symmetric monoidal categories $(\categ E, \otimes , \II_{\categ E})$
and $(\categ F, \otimes , \II_{\categ F})$ and a functor $G: \categ E \to \categ F$. It induces a
functor from $\categ E$-enriched coloured symmetric sequences to
$\categ F$-enriched coloured symmetric sequences.

\begin{proposition}\label{propchangeone}
 Suppose that $G$ is lax symmetric monoidal. Then, for any $\categ E$-enriched operad $\operad P$,
 the coloured symmetric sequence $G(\operad P)$ inherits the structure of a $\categ F$-enriched
 operad. This gives us a functor from
 $\categ E$-enriched operads to $\categ F$-enriched operads.
\end{proposition}

\begin{proof}
 Straightforward.
\end{proof}

\begin{proposition}\label{propositionchange}
 Suppose that $G$ is both lax symmetric monoidal and is oplax monoidal so that the following diagrams
 commute
 \[
\begin{tikzcd}
 	G(X\otimes Y) \otimes G(U \otimes V) 
	\arrow[r] \arrow[d]
	& G(X\otimes Y \otimes U \otimes V )
	\arrow[d]
	\\
	G(X)\otimes G(Y) \otimes G(U) \otimes G(V)
	\arrow[d]
	& G(X\otimes U \otimes Y \otimes V )
	\arrow[d]
	\\
	G(X)\otimes G(U) \otimes G(X) \otimes G(V)
	\arrow[r]
	& G(X \otimes U) \otimes G(Y \otimes V)
\end{tikzcd}
\begin{tikzcd}
 	\II
	\arrow[r] \arrow[rd, equal]
	& G(\II)
	\arrow[d]
	\\
	& \II
\end{tikzcd}
 \]
 \[
\begin{tikzcd}
 	G(\II) \otimes G(\II)
	\arrow[r] \arrow[d]
	& G(\II\otimes \II)
	\arrow[d]
	\\
	\II \otimes \II
	\arrow[rd]
	& G(\II )
	\arrow[d]
	\\
	&\II
\end{tikzcd}
\begin{tikzcd}
 	G(\II) \otimes G(\II)
	& G(\II\otimes \II)
	\arrow[l]
	\\
	\II \otimes \II
	\arrow[u]
	& G(\II )
	\arrow[u]
	\\
	&\II
	\arrow[lu] \arrow[u]
\end{tikzcd}
 \]
 Then, the induced functor from operads enriched in $\categ E$ to
operads enriched in $\categ F$ is oplax monoidal.
 If $G$ is actually oplax symmetric monoidal, then so is the induced functor.
\end{proposition}

\begin{proof}
This amounts to prove that for any $\categ E$-enriched operads $\operad P$ and $\operad Q$,
the structural maps that make $G$ oplax monoidal
\begin{align*}
	G(\operad P(\underline c;c)\otimes \operad Q(\underline c;c))
	&\to G(\operad P(\underline c;c))\otimes G(\operad Q(\underline c;c))
	\\
	G(\III) &\to \III 
\end{align*}
form morphisms of operads. This follows from the commutation of the above diagrams.
\end{proof}

\begin{corollary}
 Under the assumption of Proposition \ref{propositionchange}, the functor $G$ induces a functor
 from Hopf operads enriched in $\categ E$ to Hopf operads enriched in $\categ F$ and a functor from Hopf $\operad Q$-comodules
 to Hopf $G(\operad Q)$-comodules, for any $\categ E$-enriched Hopf operad $\operad Q$. If $G$ is oplax symmetric
 monoidal the induced functor sends cocommutative Hopf operads to cocommutative Hopf operads.
\end{corollary}

\begin{example}
 For any bilinear symmetric monoidal category $(\categ E, \otimes , \II_{\categ E})$, the
essentially unique cocontinuous functor from sets to $\categ E$ that sends the one element set $\ast$ to
$ \II_{\categ E}$ is symmetric monoidal. Hence, the image through this functor of any operad in sets
in an enriched cocommutative Hopf operad.
\end{example}

Let $R$ be a commutative ring. The normalised Moore complex $N$ is the left adjoint functor from simplicial sets
to chain complexes of $R$-modules so that
\[
	N(X)_k \coloneqq \field \cdot \{\text{non degenerate k simplicies of }X\} = \left(\field \cdot X_k\right)/\{\text{degenerate k simplicies}\}.
\]
Moreover, $d(x)= \sum_{i=0}^k (-1)^i d_i (x)$, for any $x\in X_k$.

 On the one hand, $N$ has the structure of a lax symmetric monoidal functor given by the Eilenberg-Zilber shuffle map
that may be defined as follows. An k-simplex of $\Delta[n] \times \Delta[m]$ is a function $f=(f_1,f_2)$ from $\{0, \ldots, k\}$
to $\{0, \ldots, n\} \times \{0, \ldots, m\}$ whose two projections 
$f_1$ and $f_2$ are nondecreasing.
It is degenerate if and only if there exists $1\leq i\leq k$ so that $f(i-1)=f(i)$. Subsequently
a n+m-simplex $f=(f_1,f_2)$ is nondegenerate if and only if
for any $1\leq i \leq n+m$ either
\begin{itemize}
 \itemt $f_1(i)=f_1(i-1) +1$ and $f_2(i)=f_2(i-1)$;
\itemt or $f_1(i)=f_1(i-1)$ and $f_2(i)=f_2(i-1)+1$.
\end{itemize}
Such a simplex is determined by the subset of $\{1, \ldots n+m\}$ spanned
by elements $i$ so that  $f_1(i)=f_1(i-1) +1$. So, the set of
nondegenerate n+m-simplicies of $\Delta[n] \times \Delta[m]$ is
isomorphic to the set of subsets $a \subset\{1, \ldots n+m\}$ of cardinal $n$
which is isomorphic to the set of subsets $\overline a \subset\{1, \ldots n+m\}$ of cardinal $m$.
If $1_n$ and $1_m$ are the top nondegenerate simplicies of respectively $\Delta[n]$ and $\Delta[m]$, then
\[
	\mathrm{sh}(1_n\otimes 1_m) =\sum_{a \subset\{1, \ldots n+m\}; \#a=n} \mathrm{sign}(a) \cdot a ,
\]
 where $\mathrm{sign}(a)$ denotes the signature of the only (n,m)-shuffle permutation  that sends $\{1, \ldots, n\} \subset \{1, \ldots n+m\}$
 to $a$.

On the other hand, $N$ has the structure of an oplax monoidal functor given by the Alexander-Whitney map
\begin{align*}
 	\mathrm{aw}: N(X\times Y) &\to N(X) \otimes N(Y)
	\\
	(x,y) \in X_n \times Y_n &\mapsto \sum_{p=0}^n x_{|\{1, \ldots, p\}} \otimes y_{|\{p+1, \ldots, n\}}
\end{align*}
If $a$ is a subset of $\{1, \ldots, n+m\}$ of cardinal $n$ seen as an element of $N(\Delta[n]\times \Delta[m])$, then
$\mathrm{aw}(a)=0$ whenever $a\neq \{1, \ldots, n\}$. Otherwise, $\mathrm{aw}(a)=1_n \otimes 1_m$.

\begin{proposition}\label{propsimpdg}
The normalised Moore complex functor $N$ together with the Eilenberg-Zilber map and the Alexander-Whitney map
satisfy the conditions of Proposition \ref{propositionchange}. 
\end{proposition}

\begin{proof}
The commutation of the diagrams that involve the unit is clear. Thus, let us prove that the first diagram commutes, that is
the two maps
\[
	f,g: {N(X\times Y)\otimes N(U\times V)} \to 
	 {N(X\times U)\otimes N(Y\times V)}
\]
from the initial object to the final object of the first diagram of Proposition \ref{propositionchange} are equal. Since these two maps
are natural and since $N$ preserves colimits, it suffices to prove the result for objects $X,Y,U,V$ of the form
\[
\begin{cases}
 	X= \Delta[n];
	\\
	Y= \Delta[m];
	\\
	U= \Delta[l];
	\\
	V= \Delta[k].
\end{cases}
\]
Moreover, it suffices to prove the result for top degree elements of the form $a\otimes b$ where $a$ and $b$ are subsets of respectively $\{1, \ldots, n+m\}$
and $\{1, \ldots, l+k\}$ and of cardinals respectively
$n$ and $k$; indeed, elements of lower degrees are image of some
top degrees elements through a map
such as
\[
	\Delta[n'] \times \Delta[m'] \times \Delta[l'] \times \Delta[k'] \to \Delta[n] \times \Delta[m] \times \Delta[l] \times \Delta[k] .
\]
One can check
that $f(a\otimes b)=0=g(a\otimes b)$ whenever $a \neq \{1, \ldots, n\}$ and $b \neq \{1, \ldots, k\}$.
Otherwise, one can check that
\[
	f(a\otimes b)=(-1)^{mk} \{1, \ldots, n\} \otimes \{1, \ldots, m\} =g(a\otimes b) .
\]
\end{proof}

Since the monoidal category of simplicial sets is cartesian, then all simplicial operads are cocommutative Hopf. Moreover, Proposition
\ref{propsimpdg} tells us that the images under the functor $N$ of simplicial operads are Hopf dg operads. However, they are not in general
cocommutative Hopf operads. This is the reason why we deal so much about non necessarily cocommutative Hopf operads.
Actually, these dg operads should be cocommutative Hopf up to homotopy.

\subsection{Mapping operadic structures}

Monochromatic enriched operads (as well as operads over a fixed set of colours)
are themselves algebras over a set-theoretical operad $\catoperad{Op}$.
Hence, the category $\catalg{\categ E}{\catoperad{Op}}$
of monochromatic enriched operads have the structure of a symmetric monoidal category ; this is the Hadamard
tensor product.
Moreover, the category the category $\catalg{\categ E}{\catoperad{Op}}$ of $\catoperad{Op}$-coalgebras
endow a closed symmetric monoidal structure and operads are enriched-tensored-cotensored
over $\catoperad{Op}$-coalgebras. We call these coalgebras monochromatic enriched co-operads.
The case of operads and co-operads in chain complexes is for instance treated
in \cite{Roca2022}.

\begin{definition}
 A monochromatic enriched co-operad is a coalgebra in $\categ E$ over the
 operad $\catoperad{Op}$, that is a sequence $(\operad D(n))_{n \in \mathbb N}$ of objects of $\categ E$ together
 with maps
\begin{align*}
 	\delta_i &: \operad D (n+m-1) \to  \operad D(n) \otimes \operad D(m)
	\\
	\tau&: \operad D (1) \to  \II
	\\
	\sigma_\ast &: \operad D(n) \to \operad D(n)
\end{align*}
for any $n\geq 1,m\geq 0$, any $1\leq i\leq n$ and any $\sigma \in \categ S_n$, that satisfy relations that are dual to those defining an operad.
We denote $\coOperad_{\categ E,\ast}$ and $\Operad_{\categ E,\ast}$ respectively the category of monochromatic
enriched co-operads and
the category of monochromatic
enriched operads.
\end{definition}

\begin{corollary}
The category $\coOperad_{\categ E,\ast}$
of monochromatic co-operads admits a closed symmetric monoidal structure given by the Hadamard
monoidal structure on monochromatic symmetric sequences. Moreover, the category of monochromatic operads
is tensored-cotensored-enriched
over the category of monochromatic co-operads.
\end{corollary}

\begin{definition}
Let $\Perm$ be the monochromatic operad in sets generated by $m \in \Perm(2)$
that satisfies the relations
\[
\begin{cases}
  m \triangleleft (1, m) = m \triangleleft (m, 1) ;
  \\
  m \triangleleft (1, m) = m \triangleleft (1, m^\sigma).
\end{cases}
\]
\end{definition}

\begin{remark}
 A Perm algebra that admits a unit for the product is actually a unital commutative algebra. Indeed, if $u$ is a unit
 we have for any element $x,y$
 \[
 	m(x,y)= m(u,m(x,y))=m(u,m(y,x))= m(y,x).
 \]
I was communicated this fact that makes purposeless the study for themselves of unital Perm algebras by Joost Nuiten.
\end{remark}

Let us denote $\catoperad{nuOp}$
the operad that encodes nonunital monochromatic operads.
Its set of colours is $\mathbb N$ and it is
generated by composition products
$\gamma_i \in \catoperad{nuOp}((n,m);n+m-1)$ and
permutations $\sigma^\ast \in \catoperad{nuOp}(n;n)$.
We have a morphism of set theoretical operads from $\catoperad{nuOp}$ to
$\Perm$
that sends $\sigma^\ast$ to $1$ and sends any element $\gamma_i$ to $m$.

\begin{corollary}
 The category of monochromatic enriched non unital operads is
 tensored-cotensored-enriched over $\Perm$-coalgebras in $\categ E$.
\end{corollary}

\begin{definition}
 A nonunital wheeled PROP is the data of objects $\operad P(n,m) \in \categ E$ together
 with an action of $\categ S_n^\op \times \categ S_m$, with horizontal and vertical
 composition
\begin{align*}
	m_v:\operad P(n,m) \otimes \operad P(k,n) &\to \operad P(k,m) 
	\\
	m_h:\operad P(n,m) \otimes \operad P(k,l) &\to \operad P(n+k,m+l) 
\end{align*}
 and together with contraction operators
 \[
 	\xi^{i}_j: \operad P(n,m)\to \operad P(n-1,m-1), \quad \forall \ 1\leq i\leq n, \  1\leq i\leq m
 \]
 that satisfy coherence conditions.
 See \cite{MarklMerkulovShadrin09} for more details.
\end{definition}

We know from \cite{MarklMerkulovShadrin09} that wheeled PROPs
are encoded by a set-theoretical operad that
we denote $\catoperad{WP}$.
This operad is and coloured by
$\mathbb N\times \mathbb N$. The set
\[
	\catoperad{WP}\left((n_1,m_1), \ldots, (n_k,m_k) ; (n,m)\right)
\]
is spanned by the elements $(G,\phi, \psi)$ where
\begin{itemize}
 \itemt $G$ is an equivalence class of nonempty wheeled oriented graphs with $n$ input edges and $m$ output edges, with vertices counted from $1$ to $k$
 and with a counting of the inputs (resp. outputs) of the $i^{\mathrm{th}}$-vertex from $1$ to $n_i$ (resp. from $1$ to $m_i$);
 \itemt $\phi$ is a counting of the inputs of $G$ from $1$ to $n$;
  \itemt $\psi$ is a counting of the outputs of $G$ from $1$ to $m$.
\end{itemize}
The group $\categ S_k$ acts by changing the order of the vertices. Moreover, $\gamma_i((G,\phi, \psi); (G',\phi',\psi'))$ is the element
$(G'',\phi'',\psi'')$ defined as follows.
\begin{itemize}
 \itemt First $G''$ is the graph obtained by replacing the $i^{\mathrm{th}}$ vertex of $G$ by $G'$ and identifying the input (resp. output) $j$ of this vertex to the input (resp. output) $j$ of $G'$. The vertices of $G''$ are counted by inserting the vertices of $G'$ in the ordered set of vertices of $G$ at place $i$. The counting of the inputs and the outputs of each vertex of $G''$ follows from that of $G$ and $G'$.
 \itemt The way we have built the graph $G''$ induces an isomorphisms relating respectively the inputs of $G$ and $G''$ and the outputs of $G$ and $G''$.Then the maps $\phi''$ and $\psi''$ are defined respectively as the compositions of these isomorphisms with $\phi$ and $\psi$.
\end{itemize}

\begin{definition}
The wheeled commutative operad $\catoperad{WCom}$ is the monochromatic
set-theoretical operad generated
by $m \in \catoperad{WCom}(2)$ and $l\in \catoperad{WCom}(1)$
that satisfy the following relations
 \[
\begin{cases}
 	m\triangleleft (m ,1) = m\triangleleft (1,m);
	\\
	m^{(01)}=m;
	\\
	m\triangleleft (l,1)= l\triangleleft m =m.
\end{cases}	
 \]
\end{definition}

We have a morphism of set theoretical operads
from $\catoperad{WP}$ to $\catoperad{WCom}$
that sends a wheeled graph with a single vertex with
$k$ wheels to $l^k$ and a graph with $n\geq 2$ vertices to $m^{n-1}$.

\begin{corollary}
 The category of wheeled props is tensored-cotensored-enriched over
 $\catoperad{WCom}$-coalgebras.
\end{corollary}

\bibliographystyle{amsalpha}
\bibliography{bib-mapping-coalgebra}

\end{document}

%% file: macros.tex
\theoremstyle{plain}
\newtheorem{theorem}{Theorem}
\newtheorem*{theorem*}{Theorem}
\newtheorem{lemma}{Lemma}
\newtheorem{proposition}{Proposition}
\newtheorem{corollary}{Corollary}

\theoremstyle{definition}
\newtheorem{definition}{Definition}

\theoremstyle{remark}
\newtheorem{remark}{\sc Remark}
\newtheorem{example}{\sc Example}
\newtheorem*{notation}{\sc Notation}

\newcommand{\Cate}{\categ{Cat}_{\categ E}}

\newcommand{\Operade}{\Operad_{\categ E}}

\newcommand{\Operadesmall}{\Operad_{\categ E, \mathrm{small}}}
\newcommand{\eII}{\mathcal{I}}
\newcommand{\ecateg}[1]{\mathcal{#1}}
\newcommand{\cmonlax}{\categ{CMon}_\lax}
\newcommand{\cmonoplax}{\categ{CMon}_\oplax}
\newcommand{\cmonstrong}{\categ{CMon}_\strong}
\newcommand{\cmonstrict}{\categ{CMon}_\strict}
\newcommand{\cat}{\mathrm{cat}}
\newcommand{\lax}{\mathrm{lax}}
\newcommand{\oplax}{\mathrm{oplax}}
\newcommand{\strong}{\mathrm{strong}}
\newcommand{\strict}{\mathrm{strict}}
\newcommand{\smalll}{\mathrm{small}}

\newcommand{\Cats}{\mathsf{Cats}}

\newcommand{\sk}{\mathrm{sk}}

\newcommand{\Ob}{\mathrm{Ob}}

\newcommand{\categ}[1]{\mathsf{#1}}
\newcommand{\set}[1]{\mathrm{#1}}
\newcommand{\catoperad}[1]{\mathsf{#1}}
\newcommand{\operad}[1]{\mathcal{#1}}
\newcommand{\algebra}[1]{\mathrm{#1}}
\newcommand{\coalgebra}[1]{\mathrm{#1}}

\newcommand{\catofmod}[1]{{#1}\mathrm{-}\mathsf{mod}}

\newcommand{\tcatalg}[2]{\mathsf{ALG}_{\categ{#1}}\left(#2\right)}
\newcommand{\tcatcog}[2]{\mathsf{COG}_{\categ{#1}}\left(#2\right)}
\newcommand{\catcog}[2]{\Cog_{\categ{#1}}\left(#2\right)}
\newcommand{\catalg}[2]{\Alg_{\categ{#1}}\left(#2\right)}

\newcommand{\mbs}{\mathsf{S}}

\newcommand{\III}{\operad{I}}
\newcommand{\treeoperad}{\mathbb{T}}

\newcommand{\forget}{\mathrm{U}}

\newcommand{\Operad}{\mathsf{Operad}}
\newcommand{\coOperad}{\mathsf{coOperad}}

\newcommand{\verte}[1]{\mathrm{vert}(#1)}
\newcommand{\edge}[1]{\mathrm{edge}(#1)}
\newcommand{\leaves}[1]{\mathrm{leaves}(#1)}
\newcommand{\inner}[1]{\mathrm{inner}(#1)}

\newcommand{\field}{\mathbb{K}}

\newcommand{\mbn}{\mathbb{N}}

\newcommand{\id}{\mathrm{Id}}

\newcommand{\Dend}{\Omega}
\newcommand{\aDend}{\Omega^{\mathsf{act}}}
\newcommand{\cDend}{\Omega^{\mathsf{core}}}

\newcommand{\Hom}[3]{\mathrm{hom}_{#1}\left(#2 , #3 \right)}

\newcommand{\otimeshadamard}{\otimes_{\mathbb{H}}}

\newcommand{\II}{\mathbb{1}}

\newcommand{\Cog}{\mathsf{Cog}}

\newcommand{\Alg}{\mathsf{Alg}}

\newcommand{\Set}{\mathsf{Set}}

\newcommand{\op}{\mathrm{op}}

\newcommand{\End}{\mathcal{E}\mathrm{nd}}
\newcommand{\catEnd}{\mathsf{End}}

\newcommand{\As}{\mathcal{A}\mathrm{s}}
\newcommand{\uAs}{\mathrm{u}\As}

\newcommand{\Com}{\mathcal{C}\mathrm{om}}
\newcommand{\uCom}{\mathrm{u}\Com}
\newcommand{\Perm}{\catoperad{Perm}}

\newcommand{\itemt}{\item[$\triangleright$]}

\newcommand{\poubelle}[1]{}